\documentclass[11pt, a4paper]{amsart}

\usepackage{amsfonts,amsmath,amssymb,amscd}

\usepackage[all]{xy}

\newtheorem{theorem}{Theorem}[section]
\newtheorem{lemma}[theorem]{Lemma}
\newtheorem{prop}[theorem]{Proposition}
\newtheorem{corollary}[theorem]{Corollary}

\theoremstyle{definition}
\newtheorem{definition}[theorem]{Definition}
\newtheorem{rem}[theorem]{Remark}

\newtheorem{example}[theorem]{Example}

\newtheorem{convention}[theorem]{Convention}

\usepackage{color}

\newcommand\pf{\begin{proof}}
\newcommand\epf{\end{proof}}
\newcommand{\cmdblackltimes}{\mathop{\raisebox{0.2ex}{\makebox[0.92em][l]{${\scriptstyle\blacktriangleright\mathrel{\mkern-4mu}<}$}}}}

\newcommand{\mr}{\mathrm}
\newcommand{\ms}{\mathsf}

\newcommand{\op}{\operatorname}

\newcommand{\cA}{\mathcal{A}}
\newcommand{\cB}{\mathcal{B}}
\newcommand{\cC}{\mathcal{C}}
\newcommand{\cO}{\mathcal{O}}
\newcommand{\cP}{\mathcal{P}}

\newcommand{\cF}{\mathcal{F}}
\newcommand{\cR}{\mathcal{R}}

\newcommand{\hO}{\hat{\mathcal{O}}}
\newcommand{\hP}{\hat{\mathcal{P}}}

\newcommand{\hF}{\hat{\mathcal{F}}}
\newcommand{\hotimes}{\hat{\otimes}}

\newcommand{\sV}{\mathsf{V}}
\newcommand{\sW}{\mathsf{W}}

\newcommand{\sZ}{\mathsf{Z}}
\newcommand{\sT}{\mathsf{T}}
\newcommand{\sP}{\mathsf{P}}
\newcommand{\sQ}{\mathsf{Q}}

\newcommand{\m}{\mathfrak{m}}
\newcommand{\n}{\mathfrak{n}}

\newcommand{\B}{\mathfrak{B}}
\newcommand{\U}{\mathfrak{U}}

\newcommand{\f}{\mathfrak{f}}
\newcommand{\g}{\mathfrak{g}}
\newcommand{\T}{\mathfrak{T}}

\newcommand{\hy}{\operatorname{hy}}
\newcommand{\Lie}{\operatorname{Lie}}
\newcommand{\Hom}{\operatorname{Hom}}
\newcommand{\red}{\operatorname{red}}
\newcommand{\sdim}{\operatorname{sdim}}

\numberwithin{equation}{section}

\title[Super Lie groups over a complete field]
{Hopf-algebraic techniques applied to super Lie groups over a complete field}

\author[M.~Hoshi]{Mitsukazu Hoshi}
\address{Mitsukazu Hoshi,
Institute of Mathematics, 
Graduate School of Pure and Applied Sciences, 
University of Tsukuba, 
Ibaraki 305-8571, Japan}
\email{m-hoshi@math.tsukuba.ac.jp}

\author[A.~Masuoka]{Akira Masuoka}
\address{Akira Masuoka,
Institute of Mathematics, 
University of Tsukuba, 
Ibaraki 305-8571, Japan}
\email{akira@math.tsukuba.ac.jp}

\author[Y.~Takahashi]{Yuta Takahashi}
\address{Yuta Takahashi,
Graduate School of Pure and Applied Sciences, 
University of Tsukuba, 
Ibaraki 305-8571, Japan}
\email{y-takahashi@math.tsukuba.ac.jp}

\begin{document}

\begin{abstract}
We show
basic results on super-manifolds and super Lie groups over a complete field of
characteristic $\ne 2$, extensively using Hopf-algebraic techniques. 
The main results are two theorems.
The first main theorem shows a category equivalence 
between super Lie groups and Harish-Chandra pairs,
which is applied especially to construct the Hopf super-algebra of all analytic
representative functions on a super Lie group. 
The second constructs 
homogeneous super-manifolds by a new Hopf-algebraic method, showing their remarkable property.  
\end{abstract}

\maketitle



\section{Introduction}\label{sec:introduction}

\subsection{Our aim and achievements}\label{subsec:I1}
Varadarajan \cite[Section 11]{Vara} raised four problems for future research on super-symmetry, 
one of which is ``to extend super-symmetry to $p$-adic Lie groups". 
This paper aims to answer the problem. Indeed, our answer is not total, but we work in the more general situation
that the base field is a complete field of characteristic $\ne 2$. A \emph{complete field} 
\cite[Part II, Chaper I]{S} is a (necessarily infinite) field which is complete with respect to
a non-trivial absolute value. The examples include the real- and the complex number fields $\mathbb{R}$ and 
$\mathbb{C}$, the $p$-adic fields $\mathbb{Q}_p$ and their finite extension fields and the 
formal power series field $F((T))$ in one variable over an arbitrary field $F$. 

Choose such a field $\Bbbk$ of characteristic $\ne 2$, as our base field. We will prove basic results 
on super-manifolds and super Lie groups over $\Bbbk$. Our main results are 
Theorem \ref{thm:equivalence} and Theorem \ref{thm:quotient}. 
The former proves a category equivalence between super Lie groups
and Harish-Chandra pairs, and produces some consequences including Corollaries 
\ref{cor:super_Lie_group_split} and \ref{cor:algebraic_to_analytic}, and
Theorem \ref{thm:G-super-mod_struc}. The latter Theorem \ref{thm:quotient} constructs homogeneous 
super-manifolds, showing their remarkable property which is likely new even when $\Bbbk=\mathbb{R}$ or $\mathbb{C}$. 
Preliminaries include discussion of super-manifolds
from the scratch 
since even such basics on those over a complete field cannot be found in the literature. 
Our method is Hopf-algebraic, featured by
(1)~frequent use of discrete and topological Hopf algebras together with duality between them, and 
(2)~applications of well-known results on (discrete) Hopf algebras, such as those on Hopf modules; the latter applications 
in such analytic situation of ours might have been unexpected so far. Apart from Hopf algebra another feature is
(3)~careful treatment of sheaves with close investigation of stalks or their completions. 
Especially on Hopf (super-)algebras we will give enough details,  
since they are highly expected to play an important role in further investigations of super Lie groups, as well.

\subsection{Basic results on super-manifolds}\label{subsec:I2}
These will be presented in Section \ref{sec:super-mfd}. 
A super-manifold (see Definition \ref{def:super-mfd}) is a kind of a super-ringed space $X=(|X|, \cO_X)$, that is, 
a topological space $|X|$ equipped with a structure sheaf $\cO=\cO_X$
of super-algebras; this $|X|$, equipped with the associated sheaf $\cO_{\red}$ of reduced 
algebras, turns into an analytic manifold, $X_{\red}=(|X|, \cO_{\red})$. 
(Super-algebras are supposed to be super-commutative unless otherwise stated.)
Theorem \ref{thm:IFT} is the Inverse Function Theorem proved in our generalized situation;
an idea of the proof is to focus on the stalks $\cO_x$, $x \in X$, 
or rather on the associated graded super-algebras $\ms{gr}\, \cO_x$, and 
it is used in some other places as well.  
Proposition \ref{prop:finite_products} shows that the category $\ms{SMFD}$ of super-manifolds 
has finite products,
whence group objects in $\ms{SMFD}$ are defined. Such a group object is called a \emph{super Lie group};
see Definition \ref{def:super_Lie_group}. 

\subsection{Super Lie groups and Harish-Chandra pairs}\label{subsec:I3}
These will be discussed in Sections \ref{sec:super_Lie_group}--\ref{sec:equivalence}.
The notion of Harish-Chandra pairs and their category equivalence with super Lie groups in the
$C^{\infty}$-situation are attributed to Kostant \cite{Kostant} and Koszul \cite{Koszul}; see also
\cite[Section 7.4]{CCF}. The analogous category equivalences are proved by Vishnyakova \cite{V}, and
Carmeli and Fioresi \cite{CF} in the complex analytic situation, and by
the second-named author \cite{M3} (see also \cite{M4,G,MS1,MS2}) in the algebraic situation; the group 
objects in the last algebraic situation are affine algebraic super-group schemes 
over an arbitrary field of characteristic $\ne 2$; see Section \ref{subsec:Hopf_group_scheme}.

Given a super Lie group $G$, the associated analytic manifold $G_{\red}$ naturally turns into a Lie group,
which acts on the Lie super-algebra $\Lie(G)$ of $G$ by the right adjoint action. Regard the odd part 
$\Lie(G)_1$ of $\Lie(G)$ as an analytic right $G_{\red}$-module by restriction. Then 
the pair $(G_{\red}, \Lie(G)_1)$, equipped with the bracket $[ \ , \ ] : \Lie(G)_1\otimes \Lie(G)_1 \to \Lie(G_{\red}) (=\Lie(G)_0)$
restricted from that of $\Lie(G)$, is a \emph{Harish-Chandra pair} as we define; in fact, our definition
of the notion (see Definition \ref{def:HCP}) is different from (though equivalent to) the standard one. 
The first main theorem of ours, Theorem \ref{thm:equivalence}, states that $\Phi : G \mapsto (G_{\red}, \Lie(G)_1)$ gives an
equivalence from the category $\ms{SLG}$ of super Lie groups to the category $\ms{HCP}$ of Harish-Chandra pairs. 
To prove this, essential is to construct a super Lie group from a given Harish-Chandra pair, 
which shall lead to a quasi-inverse of the functor $\Phi$. Although the super Lie group we construct
turns out to be essentially the same as the ones constructed by Carmeli et al. \cite{CCF, CF} in the $C^{\infty}$- 
and the complex 
analytic situations, we choose the (hopefully, rigorous) 
manner of constructing it from a certain Hopf super-algebra $\mathcal{B}$
associated with the given Harish-Chandra pair; see Section \ref{subsec:quasi-inverse}. 
The Hopf super-algebra $\mathcal{B}$ was already constructed in \cite{MS1,M3}.
But we will re-produce the construction because
one could then realize the construction (of $\mathcal{B}$, and hence of the super Lie group as well) 
to be quite natural, and because 
it is used to prove the subsequent Theorem \ref{thm:G-super-mod_struc}; this theorem shows that
given a super Lie group $G$, the associated $\mathcal{B}$ is the \emph{Hopf super-algebra 
of analytic representative functions} on $G$ in the sense that the category
$G\text{-}\ms{SMod}$ of analytic $G$-super-modules is identified with the 
category $\ms{SComod}\text{-}\mathcal{B}$ of $\mathcal{B}$-super-comodules. 
We would emphasize: once one defines the notion of Harish-Chandra pairs as we do, 
it is quite natural to formulate and to prove the theorem. 

Corollary \ref{cor:algebraic_to_analytic}, a corollary to Theorem \ref{thm:equivalence}, 
constructs a natural functor from 
the category of affine algebraic super-group schemes satisfying some additional condition 
to the category of super Lie groups. 
One sometimes encounters such presentations of real or complex super Lie groups
that look presenting, by use of matrices, affine algebraic super-group schemes. 
The last mentioned functor justifies the presentations (under some mild restriction in the real case),
since they can be understood to
indicate the super Lie groups which arise through the functor from 
the actually presented super-group schemes. 

\subsection{Quotients $G/H$}\label{subsec:I3a}
Their construction will be done in the final Section~\ref{sec:quotient}. 

The second main theorem, Theorem \ref{thm:quotient}, constructs the quotient super-manifold $G/H$,
where $G$ is a super Lie group and $H$ is a super Lie subgroup of $G$. The construction of ours is 
rather Hopf-algebraic, and should be new. Moreover, we will show a remarkable property of $G/H$, which is
likely new  even when $\Bbbk=\mathbb{R}$ or $\mathbb{C}$. 
To describe the property we remark that the Lie group $H_{\red}$ associated with $H$ is a Lie subgroup of the Lie group $G_{\red}$ 
associated with $G$, so that the quotient manifold $G_{\red}/H_{\red}$ is constructed. 
In addition, the underlying 
topological space $|G/H|$ of $G/H$ coincides with that space $|G_{\red}/H_{\red}|$ of $G_{\red}/H_{\red}$. It is known that
the classical quotient $\pi : G_{\red}\to G_{\red}/H_{\red}$ is a principal $H_{\red}$-bundle. This means every point of 
$|G_{\red}/H_{\red}|$ is contained in an open set $U$ such that 
there exists a right $H_{\red}$-equivariant isomorphism $U\times H_{\red}\overset{\simeq}{\longrightarrow} \pi^{-1}(U)$,
called a \emph{trivialization}, which is compatible with the natural epimorphisms onto $U$. The property we will show is that
every trivialization such as above can \emph{lift} to one in the super context. Or more precisely, 
there exists a right $H$-equivariant isomorphism 
$(U, \cO_{G/H}|_U)\times H \overset{\simeq}{\longrightarrow} (\pi^{-1}(U), \cO_G|_{\pi^{-1}(U)})$ of super-manifolds
which is compatible with the natural epimorphisms onto $(U, \cO_{G/H}|_U)$ and which is reduced to the original trivialization; 
note that $U$ is here regarded an open set of $|G/H|$. To prove this, Hopf modules as well as the construction of the quasi-inverse 
of the functor $\Phi$ above play a role. 
The proof is given by analogous argument of proving \cite[Theorem 4.12]{MT}, 
a very recent result by the second- and the third-named authors; it turns to have proved an analogous
result to our Theorem \ref{thm:quotient} in the algebraic situation mentioned above at the first paragraph of Section \ref{subsec:I2}. 
See Remark \ref{rem:algebra_analogue}. 

\subsection{Key duality}\label{subsec:I4}
The next Section \ref{sec:preliminaries} is devoted to preliminaries, which include
introducing topological super-(co)modules; see Section \ref{subsec:top_super-vect_space}. 
The following Section \ref{sec:complete_Hopf} discusses duality between complete local super-algebras
and connected super-coalgebras. Theorem \ref{thm:duality} proves an anti-equivalence between the
category of
Noetherian complete (topological) Hopf super-algebras and the category of connected
Hopf super-algebras of finite type. (The latter connected Hopf super-algebras may not be super-commutative,
and are called \emph{hyper-super-algebras}; see Definition \ref{def:super-hyper}.)
An important consequence of this and Proposition \ref{prop:topological_comodule} is that
a topological super-comodule over a Noetherian complete Hopf super-algebra is identified with 
a topological super-module over the corresponding connected Hopf super-algebra of finite type. 
Let $G$ be a super Lie group with structure sheaf $\cO$. Then the completion $\hO_e$ of the
stalk $\cO_e$ at the identity element $e$ is a Noetherian complete Hopf super-algebra. 
The corresponding finite-type connected Hopf super-algebra 
is the \emph{hyper-super-algebra} $\hy(G)$ \emph{of} $G$ (see Definition \ref{def:hy_Lie}); 
it coincides with the universal envelope of $\Lie(G)$ in characteristic zero. 
The identification mentioned above is used effectively, for example, when we prove
in Corollary \ref{cor:hy(G_red)-module_sheaf}
that $\cO$ is a sheaf of $\hy(G)$-super-module super-algebras.

\section{Preliminaries}\label{sec:preliminaries}

Suppose that $\Bbbk$ is a field of characteristic $\ne 2$. 
Throughout we work over it. Vector spaces and linear maps are thus
those over $\Bbbk$;\ $\otimes$ denotes the tensor product over $\Bbbk$;\
$\op{Hom}$ indicates the vector space of all $\Bbbk$-linear maps. 

\subsection{Super-vector spaces}\label{subsec:super-vec}
A \emph{super-vector space} is a vector space $\sV = \sV_0 \oplus \sV_1$ graded by the order $2$ group
$\mathbb{Z}_2=\{ 0,1 \}$ which consists of the even $0$ and the odd $1$. 
It is said to be \emph{purely even} (resp., \emph{purely odd})
if $\sV=\sV_0$ (resp., if $\sV = \sV_1$).
The super-vector spaces and the super- (or $\mathbb{Z}_2$-graded) 
linear maps form a symmetric tensor category, $\ms{SMod}_{\Bbbk}$,
with respect to the obvious tensor product $\sV \otimes \sW$ and the so-called super-symmetry
\[
\tau=\tau_{\sV,\sW}: \sV \otimes \sW \to \sW \otimes \sV,\ \tau(v \otimes w)=(-1)^{|v||w|}w\otimes v, \]
where $v \in \sV$ and $w \in \sW$ are homogeneous elements of degree $|v|$ and $|w|$, respectively. 
The unit object is the purely even $\Bbbk$. 
The ordinary objects, such as Lie algebra or Hopf algebra, defined in the tensor category of vector
spaces based on its trivial symmetry are generalized by the ``super" objects defined in 
$\ms{SMod}_{\Bbbk}$, which are called with the adjective ``super" attached, so as Lie super-algebra
or Hopf super-algebra. Ordinary objects are precisely purely even super-objects. 

Let $\g=\g_0\oplus \g_1$ be a Lie super-algebra, to give a remark on its definition. 
If $\op{char} \Bbbk =3$, we require the 
associated bracket $[ \ , \ ] : \g \otimes \g \to \g$ to satisfy (B3) in addition to (B1) and (B2), below.
\begin{itemize}
\item[(B1)] $[\ , \ ] \circ (\op{id}_{\g \otimes \g} + \tau_{\g, \g})= 0$,
\item[(B2)] $[[ \ , \ ], \ ]\circ (\op{id}_{\g \otimes \g \otimes \g}
+\tau_{\g, \g \otimes \g}+\tau_{\g \otimes \g, \g})=0$, 
\item[(B3)] $[[v,v],v]=0$,\ $v \in \g_1$.
\end{itemize}
See \cite[Section 3.1]{MS1}, for example. If $\op{char} \Bbbk \ne 3$, then (B2) implies (B3). 

\subsection{Hopf super-algebras and affine super-group schemes}\label{subsec:Hopf_group_scheme}
A \emph{super}-(\emph{co})\emph{algebra} is precisely a $\mathbb{Z}_2$-graded (co)algebra.

A super-algebra $A=A_0\oplus A_1$ is said to be \emph{super-commutative} 
if the product $A \otimes A \to A,\ a \otimes b\mapsto ab$
is unchanged, composed with $\tau_{A,A}$, or more explicitly, if $A_0$ is central in $A$
and if arbitrary two elements $a, b\in A_1$ are anti-commutative, or $ab=-ba$. We set the following.
\begin{convention}\label{conv:super-com}
Super-algebras
are supposed to be super-commutative unless otherwise stated.
\end{convention}

Accordingly, algebras are supposed to be commutative. 

For a super-coalgebra $C$, the coproduct and the counit are denoted by
\[
\Delta=\Delta_C : C\to C \otimes C,\ \Delta(c)=c_{(1)}\otimes c_{(2)}\ \text{and}\ 
\epsilon =\epsilon_C : C \to \Bbbk,
\]
respectively. We use the sigma notation \cite[Section 1.2]{Sw} as above, to present $\Delta(c)$. 
A super-coalgebra is said to be \emph{super-cocommutative} if $\Delta=\tau_{C,C}\circ \Delta$. 

The Hopf super-algebras we will consider are super-commutative \emph{or} super-cocommutative;
we emphasize that they may not be super-commutative.  
The antipode of a Hopf super-algebra, say $C$, is denoted by $S=S_C : C \to C$. 
We let $\mathbb{N}=\{ 0,1,2,\dots\}$ denote the monoid of non-negative integers. 
An $\mathbb{N}$-\emph{graded Hopf super-algebra} is an algebra and coalgebra graded by $\mathbb{N}$
which, regarded by the mod $2$ reduction as a super-algebra and super-coalgebra, is a Hopf super-algebra. 

An \emph{affine super-group scheme} is a representable group-valued functor $\ms{G}$ defined on the category
of super-algebras. It is uniquely represented by a super-commutative
Hopf super-algebra, which we denote by $O_{\ms{G}}$. Thus, $\ms{G}$ is the functor which associates
to each super-algebra $R$, the group $\mr{SAlg}(O_{\ms{G}}, R)$ of all super-algebra maps 
$O_{\ms{G}}\to R$ with respect to the convolution-product \cite[Definition 1.4.1]{Mo}.
This $\ms{G}$ is called an \emph{affine algebraic super-group scheme}
if $O_{\ms{G}}$ is finitely generated. 

Suppose that $C$ is a Hopf super-algebra. 
We let
\begin{equation}\label{eq:SMod_SComod}
C\text{-}\ms{SMod},\ \ms{SMod}\text{-}C \quad (\text{resp}.,\ C\text{-}\ms{SComod},\ \ms{SComod}\text{-}C) 
\end{equation}
denote the tensor categories of left $C$-super-modules and of right $C$-super-modules
(resp., of left $C$-super-comodules and of right $C$-super-comodules).
A (co)algebra object in
$C\text{-}\ms{SMod}$ is called a \emph{left $C$-super-module super-(co)algebra}; 
cf. \cite[p.203]{Mo}.
Note that the last definition involves the tensor product of 
super-(co)algebras, and hence the super-symmetry.  
The categories above are symmetric with respect to the super-symmetry 
if $C$ is super-cocommutative (resp., super-commutative); in fact, the symbols will be used only 
when that is the case, and in addition, $C$ is purely even. Note that if $C$ is purely even, 
then a $C$-super-(co)module is precisely a super-vector space whose homogeneous components are both
$C$-(co)modules.

Given a super-vector space $\sV =\sV_0\oplus \sV_1$, we let $\sV^*=\mr{Hom}(\sV, \Bbbk)$ denote
the dual vector space; it is a super-vector space with respect to the parity $(\sV^*)_i=(\sV_i)^*$, $i=0,1$. 
When we discuss duality concerning Hopf super-algebras 
there are two choices, which are, however, equivalent (more precisely, the resulting difference is made up
by category equivalences),
as was shown in \cite{MS2}. 
Let $\sW$ be another super-vector space, and suppose that you have to define 
a pairing $\langle \ , \ \rangle : \sV^*\otimes \sW^* \times \sV \otimes \sW \to \Bbbk$. 
The choices depend on whether you define
$\langle v^*\otimes w^*, v \otimes w\rangle$ to be 
\[ v^*(v)\, w^*(w)\quad \text{or}\quad (-1)^{|v||w|}v^*(v)\, w^*(w),\]
where $v \in \sV$, $w \in \sW$, $v^*\in \sV^*$ and $w^*\in \sW^*$. 
We choose the first, simpler one, as the article \cite{M3} does with explicit citation. 
Consequently, for example, the dual super-algebra $C^*$ of a super-coalgebra $C$ coincides
with the dual algebra \cite[Section 1.1]{Sw} constructed as usually, when one regards $C$ as an ordinary coalgebra. 

\subsection{(Graded) algebras associated with a super-algebra}\label{subsec:associated_super-algebras}\
Let $A=A_0\oplus A_1$ be a super-algebra, which is super-commutative 
by Convention \ref{conv:super-com}. 

Let $I_A=(A_1)$ denote the super- (or homogeneous) ideal generated by the odd component $A_1$. 
Define
\begin{equation}\label{eq:barA_grA}
\overline{A} := A/I_A, \quad 
\ms{gr}\, A := \bigoplus_{n \ge 0} I_A^n/I_A^{n+1} \, (=\overline{A} \oplus I_A/I_A^2\oplus \dots).  
\end{equation}
Then $\overline{A}$ is the largest purely even quotient super-algebra of $A$, and
$\ms{gr}\, A$ is an $\mathbb{N}$-graded algebra, which is super-commutative, regarded as a super-algebra
by the mod $2$ reduction. 
Moreover, $\ms{gr}$ gives rise to a functor from the category 
of (super-commutative) super-algebras to the category of super-commutative $\mathbb{N}$-graded 
algebras. 

Since every odd element of $A$ is square-zero, it follows that 
the set $\sqrt{0}$ of all nilpotent elements in $A$ turns into a super-ideal including $I_A$. Define
\begin{equation}\label{eq:Ared}
A_{\mr{red}} :=A/\sqrt{0}. 
\end{equation}
One sees that $A_{\mr{red}}$ is the largest reduced quotient algebra of $A$,
and that $A\mapsto A_{\mr{red}}$ gives rise to a functor. 
In the (mostly, geometric) circumstances where we will apply the construction, $\overline{A}$ is already reduced, 
whence $\overline{A}=A_{\mr{red}}$. But we then prefer to use the notation \eqref{eq:Ared}, following custom. 

\begin{rem}\label{rem:gr}
In fact, the construction of $\ms{gr}$ will be used in the special situation that
$A=\ms{gr}\, A$, and $I_A\, (=\bigoplus_{n>0}A(n))$ is nilpotent, whence $A(n)=0$ for $n \gg 0$. 
For another such super-algebra $B$, we will consider a super-algebra map $f : A \to B$. 
Note that $\ms{gr}\, f : A \to B$ may be different from $f$, but it is an isomorphism if and only if $f$ is. 
\end{rem}

\subsection{Topological super-vector spaces}\label{subsec:top_super-vect_space}
An important role will be played by an obvious super-analogue of
the concept of topological vector spaces \cite[p.507]{T2}, defined as follows. 
A \emph{topological super-vector space} is a super-vector space $\sV$ equipped with a topology
such that 
\begin{itemize}
\item[(i)] for every $v \in \sV$, the translation $w \mapsto v+w$ is continuous, and
\item[(ii)] there exists a base $(\sV_{\alpha})_{\alpha}$ for neighborhoods of $0$ which
consists of super-vector subspaces $\sV_{\alpha}$ of $\sV$; such a base will be called 
a \emph{topological base}.
\end{itemize}

Let $\sV$ and $(\sV_{\alpha})_{\alpha}$ be as above. The \emph{completion} of $\sV$ 
is defined by
\[ 
\hat{\sV}=\underset{\longleftarrow\alpha}{\lim}\, \sV/\sV_{\alpha}. 
\]
The definition does not depend on choice of topological bases. This $\hat{\sV}$ is a topological
super-vector space with the topological base consisting of all
\[
\hat{\sV}_{\alpha}= \underset{\longleftarrow\beta}{\lim}\, (\sV_{\alpha}+\sV_{\beta})/\sV_{\beta}.
\]
Alternatively, $\hat{\sV}$ is the projective limit of the discrete super-vector spaces $\sV/\sV_{\alpha}$. 
We have a canonical, continuous super-linear map $\sV\to \hat{\sV}$. If this is an isomorphism, then
$\sV$ is said to be \emph{complete}. For example, every discrete topological super-vector 
space, that has the single $\{ 0 \}$ as its topological base, is complete.

Given another topological super-vector space $\sW$ with a topological base $(\sW_{\beta})_{\beta}$, 
the tensor product $\sV \otimes \sW$ is a topological super-vector space with the topological
base $(\sV_{\alpha}\otimes \sW+ \sV\otimes \sW_{\beta})_{\alpha,\beta}$; see \cite[p.509]{T2}. 
The topology does not depend of choice of topological bases. 
The completion of $\sV \otimes \sW$ is denoted by $\sV~\hotimes~\sW$, and is called 
the \emph{complete tensor product} of $\sV$ and $\sW$. We thus have
\[
\sV~\hotimes~\sW=
\underset{\longleftarrow}{\lim}_{\alpha,\beta}\, \sV/\sV_{\alpha}\otimes \sW/\sW_{\beta}.
\]
Note that $\sV/\sV_{\alpha}= \hat{\sV}/\hat{\sV}_{\alpha}$, whence
$\sV~\hotimes~\sW=\hat{\sV}~\hotimes~\hat{\sW}$.
The \emph{complete tensor product} $f~\hotimes~g$ of 
continuous super-linear maps $f$ and $g$ is 
the completion of the continuous super-linear map $f \otimes g$. 

If $\sW$ is discrete and finite-dimensional, then we have
\begin{equation}\label{eq:tensor_product_with_discrete}
\sV~\hotimes~\sW= \hat{\sV}\otimes \sW. 
\end{equation}
When regarded as a (purely even) super-vector space, $\Bbbk$ is always supposed to be 
discrete. Hence, in particular, 
\begin{equation}\label{eq:tensor_product_with_k}
\sV~\hotimes~\Bbbk = \hat{\sV}=\Bbbk~\hotimes~\sV. 
\end{equation}
See \cite[(1.9), p.510]{T2}.

The notion of complete topological coalgebras defined in \cite[p.510]{T2} is directly generalized
to the super context as follows.
A complete topological super-vector space $C$ 
equipped with continuous super-linear maps $\Delta : C \to C~\hotimes~C$, 
$\epsilon : C \to \Bbbk$ 
is called a \emph{complete topological super-coalgebra}, if 
$\Delta$ and $\epsilon$ satisfy
the coalgebra axioms which are modified from the ordinary, discrete situation 
with $\otimes$ replaced by $\hotimes$. Note that the counit property then makes sense
since we have $C~\hotimes~\Bbbk = C = \Bbbk~\hotimes~C$ by \eqref{eq:tensor_product_with_k}. 

Let $R$ be a super-algebra which may not be super-commutative, 
and let $m : R \otimes R \to R$ and $u : \Bbbk \to R$ denote
the product and the unit, respectively. Let $(R_{\gamma})_{\gamma}$ denote the
system of all finite-dimensional super-subspaces. 
Then $R^*$ is a complete topological super-vector space with the topological base 
$\big((R/R_{\gamma})^*\big)_{\gamma}$, so that
we have 
\[
R^*~\hotimes~R^*=
\underset{\longleftarrow\gamma}{\lim}\, R_{\gamma}^*\otimes R_{\gamma}^*=(R\otimes R)^*.
\]
Moreover, $m^* : R^* \to (R\otimes R)^*=R^*~\hotimes~R^*$ and $u^* : R^* \to \Bbbk$ 
are continuous maps, with which $R^*$ is a complete topological 
super-coalgebra; see \cite[Example 1.15]{T2}.  

Let $\sV$ be a complete topological super-vector space with a topological base 
$(\sV_{\alpha})_{\alpha}$. A \emph{topological right} $R^*$-\emph{super-comodule structure} on $\sV$
is a continuous super-linear map $\rho : \sV \to \sV~\hotimes~R^*$ that satisfies the 
the right-comodule axioms modified as before; see \cite[p.512]{T2}.  Note that
\[
\sV~\hotimes~R^*=
\underset{\longleftarrow}{\lim}_{\alpha,\gamma} \Hom(R_{\gamma}, \sV/\sV_{\alpha})=\Hom(R,\sV).
\]
Therefore, given a topological right $R^*$-super-comodule structure
$\rho$ on $\sV$ as above, there arises a super-linear map $\rho' : R \otimes \sV \to \sV$
which is adjoint to the $\rho : \sV \to \Hom(R,\sV)$ above. 

\begin{prop}[\text{\cite[Theorem 1.19]{T2}}]\label{prop:topological_comodule}
This $\rho'$ is a left $R$-super-module structure on $\sV$ which is 
topological in the sense that every element of $R$ acts on $\sV$ 
as a continuous super-linear endomorphism.
Moreover, $\rho \mapsto \rho'$ gives a bijection from 
the set of all topological right $R^*$-super-comodule structures on $\sV$ to 
the set of all topological left $R$-super-module structures on $\sV$.
\end{prop}

\section{Complete Hopf super-algebras and hyper-super-algebras}\label{sec:complete_Hopf}

We continue to suppose that $\Bbbk$ is a field of characteristic $\ne 2$. 

\subsection{Local super-algebras}\label{subsec:local_super-algebra}
Let $A=A_0\oplus A_1$ be a super-algebra. 
Recall the following definitions; see \cite[Appendix]{MZ}, for example. 
$A$ is \emph{Noetherian}, if the super-ideals of $A$ satisfies the ACC, 
or equivalently, if the algebra $A_0$ is Noetherian and the $A_0$-algebra $A$ 
is generated by finitely many odd elements. 
$A$ is \emph{smooth}, if given a super-algebra surjection $B \to C$ with nilpotent kernel, 
every super-algebra map $A \to C$ factors through the given surjection.
$A$ is \emph{local}, if it is non-zero and has a unique maximal super-deal; 
we denote this maximal by $\m_A$. 
The condition is equivalent to that
the algebra $A_0$ is local, and as well, to that 
$\overline{A}$ (see \eqref{eq:barA_grA}) is local.
By convention a local super-algebra $A$ is regarded as a topological super-vector space with the topological
base $\{ \m_A^n\}_{n \ge 0}$, or namely, with the $\m_A$-adic topology.
It is, therefore, complete if and only if the canonical map
\[ 
A\to \hat{A} = \underset{\longleftarrow n}{\lim}\, A/\m_A^n 
\]
is isomorphic. In general, this canonical map is injective if $A$ is Noetherian. 
In addition, $\hat{A}$ is a complete local super-algebra 
with the maximal super-ideal 
\[
\hat{\m}_A=
\underset{\longleftarrow n}{\lim}\, \m_A/\m_A^n. 
\]

By a \emph{connected super-algebra} we mean a local super-algebra $A$ with
the residue field $A/\m_A = \Bbbk$.  

\begin{rem}\label{rem:local_map}
A super-algebra map $f : A \to B$ 
between local super-algebras is continuous 
if and only if $f(\m_A) \subset \m_B$. 
This is necessarily the case if $B$ is connected.
In fact, the composite $A \to B/\m_B=\Bbbk$ of $f$ with the projection $B \to B/\m_B$
then must be a surjection with kernel $\m_A$, whence
$A$ is connected, too.
\end{rem}

Given non-negative integers $m$ and $n$, the \emph{formal power series super-algebra}
\begin{equation}\label{eq:formal_power_series}
\Bbbk[[\sT_1,\dots, \sT_{m+n}]] 
\end{equation}
in $m$ even and $n$ odd variables is the tensor product 
\begin{equation*}
\Bbbk[[\sT_1,\dots, \sT_m]]\otimes \wedge(\sT_{m+1}, \dots,\sT_{m+n})
\end{equation*}
of the formal power series algebra $\Bbbk[[\sT_1,\dots, \sT_m]]$ in even variables
$\sT_1,\dots, \sT_m$ and
the exterior algebra $\wedge(\sT_{m+1}, \dots,\sT_{m+n})$
on the purely odd super-vector space with a basis $\sT_{m+1},\dots, \sT_{m+n}$. 

\begin{prop}\label{prop:smooth_super-algebra}
$\Bbbk[[\sT_1,\dots, \sT_{m+n}]]$
is a smooth Noetherian complete connected super-algebra. Conversely, a super-algebra $A$
with these properties is isomorphic to such a formal power series super-algebra
for some $m, n \ge 0$. 
\end{prop}
\begin{proof}
This follows easily from \cite[Proposition on Page 81]{Schmitt}; see also \cite[Appendix]{MZ}. 
\end{proof}

Given local super-algebras $A$ and $B$, their complete tensor
product is the super-algebra 
\[
A~\hotimes~B = \underset{\longleftarrow n}{\lim}\, A/\m_A^n \otimes B/\m_B^n.
\]
If $A$ and $B$ are connected, then $A~\hotimes~B$ is, too.

Given continuous super-algebra maps $f : A \to A'$ and $g : B \to B'$
between local super-algebras, the compete tensor product
$f~\hotimes~g : A~\hotimes~B \to A'~\hotimes~B'$
is a continuous super-algebra map.

\subsection{Complete Hopf super-algebras}\label{subsec:complete_Hopf}
Let $A$ be a complete connected super-algebra, and let
$\epsilon : A \to A/\m_A=\Bbbk$ denote the natural projection. 

\begin{definition}\label{def:complete_Hopf}
Given a (necessarily continuous) super-algebra map 
$\Delta : A \to A~\hotimes~A$, we say that
$A=(A,\Delta,\epsilon)$ is a \emph{complete \textup{(}topological\textup{)} Hopf super-algebra}, if $\Delta$ and $\epsilon$,
as coproduct and counit, respectively, 
satisfy the coalgebra axioms modified as before.

\end{definition}

\begin{rem}\label{rem:complete_Hopf}
(1)\
A complete Hopf super-algebra is a complete topological super-coalgebra, in particular.

(2)\
Every complete Hopf super-algebra has a unique super-algebra involution that acts 
as antipode; see \cite[Proposition 3.4]{M1}. A map of complete Hopf super-algebras
is a super-algebra map compatible with the coproduct, which is necessarily compatible
with the counit and the antipode. 

(3)\
For simplicity and most of the needs of the paper, the above definition of complete Hopf super-algebras, assuming that the
underlying complete super-algebras are connected, was made more restrictive than 
would be expected; 
indeed, it excludes discrete Hopf super-algebras, for example. 
The expected one would
be a complete super-vector space $A$ equipped with continuous super-linear maps 
\[
m : A\, \hotimes \, A \to A,\ u : \Bbbk \to A,\ \Delta : A \to A\, \hotimes\, A,\ \epsilon : A \to \Bbbk,\
S : A \to A 
\]
which satisfy the Hopf-algebra axioms. 
This version of the notion will be exceptionally used only in a few places, referred to as 
a \emph{complete Hopf super-algebra
in the wider sense}. A \emph{Hopf super-algebra map} between those objects is  
a continuous
super-algebra map compatible with the coproduct and the counit; it is necessarily compatible with 
the antipode. 
\end{rem}

We see easily the following. 

\begin{lemma}\label{lem:co-group_object}
Complete connected super-algebras and super-algebra maps
form a category, which has finite coproducts given by $\hotimes$ and the initial object $\Bbbk$. 
Co-group objects in the category
are precisely complete Hopf super-algebras. 
\end{lemma}

\begin{rem}\label{rem:formal_group}
A smooth Noetherian complete Hopf super-algebra $A$ is defined on some formal power series super-algebra
$\Bbbk[[\sT_1,\dots,\sT_{m+n}]]$ in $m$ even and $n$ odd variables,
as in \eqref{eq:formal_power_series}.  
Note that $\epsilon(\sT_i)=0$, $1 \le i \le m+n$.
We write simply as $A=\Bbbk[[\sT]]$ with $\sT=(\sT_1,\dots, \sT_{m+n})$. Then we can write as 
$A~\hotimes~A=\Bbbk[[\sT,\sP]]$, where $\sT$ is as above, and
$\sP=(\sP_1,\dots,\sP_{m+n})$ is another sequence of $m$ even and $n$ odd variables. 
Defining $\Delta$ on $A$ is the same as giving, as values $(\Delta(\sT_1),\dots, \Delta(\sT_{m+n}))$, a sequence 
\[ F(\sT,\sP)=(F_1(\sT,\sP),\dots, F_{m+n}(\sT,\sP)) \]
of elements $F_i(\sT,\sP) \in \Bbbk[[\sT,\sP]]$ which are even for $i \le m$, and odd for $i>m$,  
such that
\begin{itemize}
\item[(i)]
$F(\sT,0)=\sT,\ F(0,\sP)=\sP$, 
\item[(ii)]
$F(F(\sT,\sP), \sQ)= F(\sT, F(\sP,\sQ))$~~in~~$\Bbbk[[\sT,\sP,\sQ]]$. 
\end{itemize}
In view of the definition of formal group laws \cite[Page 111]{S},
a sequence $F(\sT)$ as above may be called a \emph{formal super-group law}. 
\end{rem}

\subsection{Connected super-coalgebras and hyper-super-algebras}\label{subsec:connected_super-coalgebra}
Let $C$ be a super-cocommutative super-coalgebra. 
Regarded as an ordinary coalgebra, it has the coradical 
filtration \cite[p.185]{Sw}
\[ C_{(0)} \subset C_{(1)} \subset C_{(2)} \subset \dots \]
which starts with the coradical $C_{(0)} =\mr{Corad}\, C$ of $C$, and amounts to
$\bigcup_{n \ge 0}C_{(n)}=C$. 
By definition $C_{(0)}$ is the direct sum of
all simple subcoalgebras, each of which is the dual $L^*$ of some finite field extension $L/\Bbbk$;
it is, therefore, purely even. For $n >0$, $C_{(n)}$ is the kernel of the composite
\begin{equation}\label{eq:corad_filt} 
C \overset{\Delta_n}{\longrightarrow} C^{\otimes (n+1)} \longrightarrow (C/C_{(0)})^{\otimes (n+1)}
\end{equation}
of the $n$-times iterated coproduct $\Delta_n$ with the natural projection; it is a super-subcoalgebra
of $C$. 
Moreover, we have $\Delta(C_{(n)})\subset \sum_{i=0}^n C_{(i)}\otimes C_{(n-i)}$. 
Note that the composite above induces injections  
\begin{equation}\label{eq:induced_injection}
C_{(n)}/C_{(n-1)} \hookrightarrow (C_{(1)}/C_{(0)})^{\otimes n},\quad n > 0. 
\end{equation}

We say that $C$ is \emph{smooth}, if given an inclusion $D \subset E$ of super-cocommutative
super-coalgebras such that $D \supset \mr{Corad}\, E$, every super-coalgebra map $D \to C$ 
extends to such a map $E \to C$. 

We say that $C$ is \emph{connected}, 
if $\dim C_{(0)}=1$.  
Suppose that this is the case. Then $C_{(0)}$ is uniquely spanned by a group-like; we denote
this group-like by $1_C$. Every super-coalgebra map between connected super-coalgebras
preserves the group-like. Let 
\[
P(C) = \{ \, x \in C \mid \Delta(x) = 1_C \otimes x + x \otimes 1_C \, \}
\]
denote the super-vector subspace of $C$ consisting of all primitives. Then we have
\[
C_{(1)}= \Bbbk 1_C \oplus P(C). 
\]
See \cite[Proposition 10.0.1]{Sw}.

A connected super-coalgebra $C$ is said to be \emph{of finite type}, if $\dim P(C) < \infty$,
or equivalently, if $\dim C_{(n)}< \infty$ for all $n >0$; see \eqref{eq:induced_injection}. 

\begin{definition}\label{def:super-hyper}
A \emph{hyper-super-algebra} is a super-bialgebra which is connected
as a super-coalgebra. It may not be super-commutative as a super-algebra. 
\end{definition}

\begin{rem}[added in revision]
Historically, hyper-algebras, which were eariler called \emph{hyper-Lie-algebras},
were first intensively studied by M.~Takeuchi \cite{T1} as objects which play
a generalized role of Lie algebras in arbitrary characteristic. The super analogue defined as above has been already
studied by several authors (see \cite{M3}, for example), but in most cases unlike in the
present paper, the super objects associated with super Lie (or algebraic) groups 
(see Definition \ref{def:hy_Lie} below) are only studied under the alternative name 
\emph{\textup{(}Hopf\textup{)} super-algebras of distributions} supported
at the identity; see \cite{CCF}, for example. 
\end{rem}

Making a connected super-coalgebra $C$ into a hyper-super-algebra is the same
as giving a super-coalgebra map $C \otimes C \to C$ which together with the group-like $1_C$
make $C$ into an algebra as its product and identity element, respectively. 

Every hyper-super-algebra has an antipode, and is, therefore, a Hopf super-algebra.
We now see from \cite[Corollary 11.0.6]{Sw} the following.

\begin{lemma}\label{lem:group-object}
The connected super-coalgebras and the super-coalgebra maps form a category, which has
finite products given by $\otimes$ and the terminal object $\Bbbk$. 
Group objects in the category are precisely hyper-super-algebras.
\end{lemma}

Let $C$ be a hyper-super-algebra. Set $\g= P(C)$. Then this super-vector subspace of $C$
is closed under the super-commutator
\[ 
[x,y] = xy - (-1)^{|x||y|}yx,
\]
with which $\g$ forms a Lie super-algebra. The universal envelope $\U(\g)$ of $\g$
is a hyper-super-algebra in which every element of $\g$ is primitive. The embedding
$\g \to C$ uniquely extends to a map 
\begin{equation}\label{eq:U(g)toH}
\U(\g) \to C
\end{equation}
of hyper-super-algebras, which is an isomorphism if $\mr{char}\, \Bbbk=0$;
this last fact will be re-stated in Proposition \ref{prop:smooth_super-hyper}. 

\begin{example}\label{ex:exterior_algebra}
The exterior algebra $\wedge(\ms{V})$ on a vector space $\ms{V}$ is a super-commutative
hyper-super-algebra in which every element of $\ms{V}$ is an odd primitive; it is in fact an 
$\mathbb{N}$-graded Hopf super-algebra with respect to the natural grading. 
If $\dim \ms{V} < \infty$, then the natural parings 
$\langle \ , \ \rangle : \wedge^n(\ms{V}^*) \times \wedge^n(\ms{V}) \to \Bbbk$, $n \ge 0$, given by
$\langle 1,1 \rangle =1$ and 
\[
\langle v^*_1\wedge\dots \wedge v^*_n,\ v_1\wedge\dots \wedge v_n\rangle=
\det(v^*_i(v_j)), 
\]
where $n>0$ and $v^*_i \in \ms{V}^*$,\ $v_i \in \ms{V}$, $1\le i \le n$, 
induce an isomorphism
\begin{equation}\label{eq:selfdual}
\wedge(\ms{V}^*) \simeq \wedge(\ms{V})^* 
\end{equation}
of $\mathbb{N}$-graded Hopf super-algebras, through which we will identify the two objects; 
see \cite[Eq.(5), Remark 1]{M3}, for example. 
\end{example}

\subsection{Smoothness criteria}\label{subsec:smoothness_criteria}
Let $\ms{U}=\ms{U}_0 \oplus \ms{U}_1$ be a super-vector space. A \emph{connected super-coalgebra on}
$\ms{U}$ is a pair $(C,\pi)$ of a connected super-coalgebra $C$ and a super-linear map
$\pi : C \to \ms{U}$ such that $\pi(1_C)=0$. 
The connected super-coalgebras on $\ms{U}$ form a category;
morphisms are \emph{super-coalgebra} maps on $\ms{U}$, by which we mean super-coalgebra maps 
compatible with the maps to $\ms{U}$.  
A terminal object is said to be \emph{cofree}; see \cite[Chapter XII]{Sw}. 

\begin{prop}\label{prop:smooth_criteion}
There exists uniquely, up to isomorphism, a cofree connected
super-coalgebra, $(\B(\ms{U}), \pi)$, on $\ms{U}$. 
\end{prop}
\begin{proof}
The uniqueness is obvious. To prove the existence,
suppose $\ms{U}=\ms{U}'\oplus \ms{U}''$, and that $(\B(\ms{U}'),\pi')$ and $(\B(\ms{U}''),\pi'')$
are found. Then one sees that 
$(\B(\ms{U}') \otimes \B(\ms{U}''), \pi'\otimes \epsilon +\epsilon \otimes \pi'')$ 
is a cofree connected super-coalgebra on $\ms{U}$. 
Therefore, it suffices to prove the existence for $\ms{U}_0$ and $\ms{U}_1$. 

As for $\ms{U}_0$, note that
given a super-coalgebra $C$, $C_1$ is a super-coideal of $C$, and every super-coalgebra
map from $C$ to an ordinary coalgebra uniquely factors through the quotient coalgebra $C_0=C/C_1$.
Then we see that
$(C, \pi) \mapsto (C_0, \pi|_{C_0})$ gives a functor from the category
of connected super-coalgebras on $\ms{U}_0$ to the full subcategory consisting of all purely
even objects, which is left adjoint to the inclusion functor. 
Therefore, the terminal object $(\B(\ms{U}_0), \pi)$
of the full subcategory, which exists by \cite[Theorem 12.2.5]{Sw}, is the desired object.

As for $\ms{U}_1$, the exterior algebra $\wedge(\ms{U}_1)$, regarded as a hyper-super-algebra
as in Example \ref{ex:exterior_algebra},
and the natural projection $\pi : \wedge(\ms{U}_1)\to \wedge^1(\ms{U}_1)=\ms{U}_1$ give 
the desired object. Given a connected super-coalgebra $(C, \pi)$ on $\ms{U}_1$, 
the composites 
\[ 
C \overset{\Delta_{n-1}}{\longrightarrow} C^{\otimes n} \overset{\pi^{\otimes n}}{\longrightarrow}
\ms{U}_1^{\otimes n} \overset{\text{cano}}{\longrightarrow} \wedge^n(\ms{U}_1),\quad n > 0,
\]
and $\epsilon : C \to \Bbbk =\wedge^0(\ms{U}_1)$ 
amount to the unique super-coalgebra map $C \to \wedge(\ms{U}_1)$ on $\ms{U}_1$. 
This is shown by modifying the proof of the last cited theorem; one can suppose $\dim \ms{U}_1<\infty$,
and reduce to the fact that 
the dual connected super-algebra 
$\wedge(\ms{U}_1^*)\, (=\wedge(\ms{U}_1)^*)$ is universal among 
those connected super-algebras each of which is equipped with a super-linear
map from $\ms{U}_1^*$ to the maximal super-ideal; see also Proposition \ref{prop:duality} below. 
\end{proof}

As is just proved, we have
\[
\B(\ms{U})=\B(\ms{U}_0)\otimes \wedge(\ms{U}_1). 
\]
Recall from \cite[Section 12.3]{Sw} that $\B(\ms{U}_0)$ includes $\ms{U}_0$ so that $P(\B(\ms{U}_0))= \ms{U}_0$ and 
$\pi|_{\ms{U}_0}= \mr{id}_{\ms{U}_0}$. Moreover, $\B(\ms{U}_0)$ is a commutative hyper-algebra
which has as its product, the unique super-coalgebra map 
$\mu : \B(\ms{U}_0)\otimes \B(\ms{U}_0)\to \B(\ms{U}_0)$ that makes the diagram
\[
\begin{xy}
(0,0)   *++{\B(\ms{U}_0)\otimes \B(\ms{U}_0)}  ="1",
(33,0)  *++{\B(\ms{U}_0)}    ="2",
(0,-15) *++{\ms{U}_0\oplus \ms{U}_0}          ="3",
(33,-15)*++{\ms{U}_0}    ="4",
{"1" \SelectTips{cm}{} \ar @{->}^{\hspace{5mm}\mu} "2"},
{"1" \SelectTips{cm}{} \ar @{->}_{\pi \otimes \epsilon + \epsilon \otimes \pi} "3"},
{"3" \SelectTips{cm}{} \ar @{->}^{\text{sum}} "4"},
{"2" \SelectTips{cm}{} \ar @{->}^{\pi} "4"}
\end{xy}
\]
commutative.
Given $0\ne b \in \ms{U}_0$, $\B(\ms{U}_0)$ includes uniquely a hyper-subalgebra $\B_b$ which has a basis
$b_0=1, b_1=b, b_2, b_3,\dots$ of infinite length, such that $\pi(b_r)=\delta_{1,r}b$ and
\begin{equation}\label{eq:divided_power_sequence}
\Delta(b_r) = \sum_{k=0}^r b_k \otimes b_{r-k},\quad \epsilon(b_r)=\delta_{0,r},\quad 
b_rb_s=\binom{r+s}{r}\, b_{r+s},
\end{equation}
where $r, s \ge 0$.
Given a basis $(b^{(i)})_{i \in I}$ of $\ms{U}_0$, $\B(\ms{U}_0)$ is the tensor
product of $\B_{b^{(i)}}$, $i \in I$. 

With $\wedge(\ms{U}_1)$ as in Example \ref{ex:exterior_algebra}, 
$\B(\ms{U})=\B(\ms{U}_0)\otimes \wedge(\ms{U}_1)$ is a super-commutative hyper-super-algebra. 
The product $\B(\ms{U})\otimes \B(\ms{U})\to \B(\ms{U})$ on $\B(\ms{U})$ gives a commutative diagram analogous to
the last one. 
If $\dim \ms{U}_0 = m < \infty$ and $\dim \ms{U}_1=n<\infty$, then one obtains from 
the description $\B(\ms{U}_0)=\bigotimes_i\B_{b^{(i)}}$ above
and \eqref{eq:selfdual}, an isomorphism of super-algebras
\begin{equation}\label{eq:dual_of_B}
\B(\ms{U})^*\simeq \Bbbk[[\sT_1,\dots, \sT_{m+n}]],
\end{equation}
where the variables $\sT_1,\dots \sT_m$ are even, and $\sT_{m+1},\dots, \sT_{m+n}$ odd. 

Let $C$ be a connected super-coalgebra, and set $\ms{U} = P(C)$. Choose
a super-linear map $\pi : C \to \ms{U}$ such that $\pi(1_C)=0$ and 
$\pi|_{\ms{U}}=\mr{id}_{\ms{U}}$. We then have a unique super-coalgebra map
\begin{equation}\label{eq:phi_on_U}
f : C \to \B(\ms{U}) 
\end{equation}
on $\ms{U}$. This is injective by \cite[Lemma 11.0.1]{Sw}, since it is identical on 
$P(C)$. 

\begin{prop}\label{prop:smoothness_criterion}
$C$ is smooth if and only if $f$ is an isomorphism. 
\end{prop}
\begin{proof}
If $C$ is smooth, then $f$ has a retraction, which
is necessarily injective again by \cite[Lemma 11.0.1]{Sw}. Therefore, $f$ is an isomorphism.

Conversely, suppose $C= \B(\ms{U})$. To see that $C$ is smooth, it suffices to prove the following, as is seen by a familiar
argument using push-outs:
an inclusion $C \hookrightarrow D$ into a super-cocommutative 
super-coalgebra $D$ such that $C \supset \mr{Corad}\, D$, or $D$ is connected, necessarily splits. 
To regard $D$ as a connected super-coalgebra on $\ms{U}$, 
extend the super-linear map $\pi : C=\B(\ms{U}) \to \ms{U}$ to $D\to \ms{U}$. 
Then the unique super-coalgebra map $D \to C$ on $\ms{U}$ is seen to be the desired retraction.
\end{proof}

Let $C$ be a hyper-super-algebra. 
The pull-back of $C_0 \otimes C_0$ along $\Delta$
\begin{equation}\label{eq:underlineH}
\underline{C}=\Delta^{-1}(C_0 \otimes C_0)
\end{equation}
is the largest purely even super-subcoalgebra of $C$, and is in fact a Hopf subalgebra; see \cite[p.291]{M2}.
This $\underline{C}$ is now a hyper-algebra. Let $\g =P(C)$. Then $\g_0=P(\underline{C})$.   
Choose arbitrarily a totally ordered basis $(v_i)_{i\in I}$ of $\g_1$. Then in $C$, the products
$v_{i_1}\dots v_{i_n}$ 
of increasing basis elements $v_{i_1}<\dots< v_{i_n}$ of arbitrary length $n\ge 0$ form a left $\underline{C}$-free
basis. Hence we have the unit-preserving isomorphism of left $\underline{C}$-super-module super-coalgebras
\begin{equation}\label{eq:isom_module_coalgebra}
\underline{C}\otimes \wedge(\g_1) \overset{\simeq}{\longrightarrow} C 
\end{equation}
which sends $v_{i_1}\wedge \dots \wedge v_{i_n}$ 
in $\wedge(\g_1)$ to the product $v_{i_1}\dots v_{i_n}$ above; see \cite[Theorem 3.6]{M2}.  

\begin{prop}\label{prop:smooth_super-hyper}
For a hyper-super-algebra $C$, the following are equivalent:
\begin{itemize}
\item[(a)] $C$ is smooth as a super-commutative super-coalgebra;
\item[(b)] $\underline{C}$ is smooth as a cocommutative coalgebra;
\item[(c)] $\underline{C}$ is isomorphic to $\B(\g_0)$ as a coalgebra. 
\end{itemize}
If $\mr{char}\, \Bbbk=0$, then we have
\[
C=\U(\g),\quad \underline{C}=\U(\g_0),
\] 
and the equivalent conditions above are necessarily satisfied. 
\end{prop}
\begin{proof}
The last assertion in characteristic zero follows from Kostant's Theorem \cite[Theorem 13.0.1]{Sw}
and the PBW Theorem; note that if $0 \ne x \in \g_0$, the sequence $1, x, \dots, x^n/n!, \dots$
in $\U(\g_0)$ satisfies the first two equalities of \eqref{eq:divided_power_sequence}. 

(a) $\Rightarrow$ (b).\ 
This follows from the definition of smoothness since
a super-coalgebra map from an ordinary coalgebra to $C$ takes values in $\underline{C}$. 

(b) $\Leftrightarrow$ (c).\ This follows by Proposition \ref{prop:smoothness_criterion} 
applied in the non-super situation. 
See also \cite[Theorem 1.8.1, p.48]{T1} for a more generalized result. 

(c) $\Rightarrow$ (a).\ Suppose $\underline{C}\simeq \B(\g_0)$. Then by \eqref{eq:isom_module_coalgebra}
we have an isomorphism 
$C \simeq \B(\g_0)\otimes \wedge(\g_1)=\B(\g)$
of connected super-coalgebras,
which may be supposed to be an isomorphism on $\g$. 
Proposition \ref{prop:smoothness_criterion} ensures (a). 
\end{proof}

\subsection{Duality theorem}\label{subsec:duality}

Let $C$ be a connected super-coalgebra of finite type, and let 
$\Bbbk 1_C =C_{(0)}\subset C_{(1)}\subset \dots$ be the coradical filtration on it. 
Set $\ms{U}=P(C)$, and choose a super-coalgebra injection 
$f : C \to \B(\ms{U})=\B(\ms{U}_0)\otimes \wedge(\ms{U}_1)$ 
on $\ms{U}$ as in \eqref{eq:phi_on_U}. 
Let $A=C^*$ denote the dual super-algebra.
By definition of $C_{(n)}$ (see \eqref{eq:corad_filt}),
the restriction maps $A=C^* \to (C_{(n)})^*$ induce isomorphisms 
\begin{equation}\label{eq:isom_n}
A/\m^{n+1} \simeq (C_{(n)})^*,\quad n >0,
\end{equation}
where $\m$ denotes the kernel of $A \to (C_{(0)})^*=\Bbbk$. 
Therefore, we have $A \simeq \underset{\longleftarrow n}{\lim}\, A/\m^n$, whence $A$ is 
a complete connected super-algebra with the maximal $\m$. Since $(C_{(1)}/\Bbbk 1_C)^* \simeq \m/\m^2$, 
it follows from \eqref{eq:dual_of_B} that $f$ is dualized to a super-algebra surjection
\[ \Bbbk[[\sT_1,\dots,\sT_{m+n}]] \to A, \]
which maps $\sT_1,\dots,\sT_{m+n}$ to a basis of $\m$ modulo $\m^2$. Therefore, 
$A$ is Noetherian. It follows from Propositions \ref{prop:smooth_super-algebra} and
\ref{prop:smoothness_criterion} that the surjection is an isomorphism, if and only if $A$ is smooth,
if and only if $C$ is smooth.

Let $A$ be a Noetherian complete connected super-algebra. 
Since $\m_A$ is then finitely generated, $\m_A^n/\m_A^{n+1}$, $n \ge 0$, are all finite-dimensional.
Therefore, $A/\m_A^{n+1}$, $n \ge 0$, are finite-dimensional super-algebras.
Let $C_{(n)}=(A/\m_A^{n+1})^*$ be the dual super-coalgebras. 
One sees that $\dim C_{(0)}=1$, and $C_{(0)} \subset C_{(1)} \subset \dots$ in $A^*$.  We have
the filtered union 
\[ 
A^{\star} = \bigcup_{n \ge 0} C_{(n)}
\]
of the super-coalgebras. 
This $A^{\star}$, consisting of all continuous linear maps $A \to \Bbbk$,
may be called the \emph{continuous dual} of $A$. 
The composite
\[
A \to (A^{\star})^* \simeq \underset{\longleftarrow n}{\lim}\, (C_{(n)})^* \simeq 
\underset{\longleftarrow n}{\lim}\, A/\m_A^n
\]
of the canonical map $A \to (A^{\star})^*$ with natural isomorphisms
coincides with the canonical isomorphism
$A \simeq \underset{\longleftarrow n}{\lim}\, A/\m_A^n$. Therefore, the canonical 
$A \to (A^{\star})^*$ is an isomorphism. 
The sequence $C_{(n)}\hookrightarrow A^{\star} \to (A^{\star}/C_{(0)})^{\otimes (n+1)}$
given by the iterated coproduct is exact, since it is dualized to the exact sequence
$A/\m_A^{n+1} \leftarrow A \leftarrow \m_A^{\hotimes (n+1)}$ given by the
iterated product. Therefore,
$C_{(0)} \subset C_{(1)} \subset \dots$ is the coradical filtration on $A^{\star}$, whence
$A^{\star}$ is a connected super-coalgebra of finite type. 
 
One sees that the assignments $C \mapsto C^*$ and $A\mapsto A^{\star}$ above give 
rise to contra-variant functors. Moreover, we have the following.
 
\begin{prop}\label{prop:duality}
$C \mapsto C^*$ and $A \mapsto A^{\star} $ give an anti-equivalence between 
\begin{itemize}
\item the category of Noetherian complete connected super-algebras, and 
\item the category of connected super-coalgebras of finite type.
\end{itemize}
This restricts to an anti-equivalence between the full subcategories consisting
of the smooth objects. 
\end{prop}
\begin{proof}
Since we have seen that given an $A$, we have a natural isomorphism $A \simeq (A^{\star})^*$,
it remains 
to prove that given a $C$, we have a natural isomorphism $C \simeq (C^*)^{\star}$;
recall we have seen that $C$ is smooth if and only if $C^*$ is.
Indeed, the isomorphisms $C_{(n)}\simeq (A/\m^{n+1})^*$, $n >0$, dual to \eqref{eq:isom_n}
amount to a desired isomorphism. 
\end{proof}

\begin{theorem}\label{thm:duality}
The anti-equivalence above induces an anti-equivalence between
\begin{itemize}
\item the category of \textup{(}smooth\textup{)} Noetherian complete Hopf super-algebras $A$ and 
\item the category of \textup{(}smooth\textup{)} hyper-super-algebras $C$ of finite type.
\end{itemize} 
\end{theorem}
\begin{proof}
Suppose that we are in the situation of Proposition \ref{prop:duality}. 
In the former (resp., latter) category, objects are closed under the coproduct
$\hotimes$ (resp., the product $\otimes$). 
The proved anti-equivalence induces an anti-equivalence
between the co-group objects and the group objects; 
this proves the theorem.
\end{proof}

Note that $C\mapsto \underline{C}$ gives a functor from the category of 
hyper-super-algebras (of finite type) to the category of hyper-algebras (of finite type). 

\begin{corollary}\label{cor:reduced_co-group}
We have the following. 
\begin{itemize}
\item[(1)] If $A$ is a Noetherian complete Hopf super-algebra, 
then $\overline{A}$ is naturally a Noetherian complete Hopf algebra; $A$ is smooth if and only if
$\overline{A}$ is.
The construction gives rise to a functor $A\mapsto \overline{A}$. 
\item[(2)] $A \mapsto \overline{A}$ and $C \mapsto \underline{C}$ are compatible with
the anti-equivalence of Theorem \ref{thm:duality} and the restricted one between the purely even objects.
To be more explicit, if $A = C^*$ or $A^{\star}=C$, then 
$\overline{A} = \underline{C}^*$ or $\overline{A}^{\star}=\underline{C}$. 
\end{itemize}
\end{corollary}
\begin{proof}
(1)\ 
By Theorem \ref{thm:duality} we have $A = C^*$, where $C$ is a hyper-super-algebra of finite type. 
We see from \eqref{eq:isom_module_coalgebra} that $\overline{A}=\underline{C}^*$, and this $\overline{A}$ is
a Noetherian complete Hopf algebra such that the canonical $A \to \overline{A}$ preserves the coproduct. 
The assertion on smoothness follows from Propositions \ref{prop:smooth_super-hyper} and
\ref{prop:duality}. The functoriality is easy to see. 

(2)\ 
This is seen from the argument above. 
\end{proof}

\begin{corollary}\label{cor:duality_in_char_zero}
Assume $\mr{char}\, \Bbbk=0$. Then every Noetherian complete Hopf super-algebra is
smooth, and we have a natural anti-equivalence between 
\begin{itemize}
\item the category of Noetherian complete Hopf super-algebras and 
\item the category of finite-dimensional Lie super-algebras.
\end{itemize}
\end{corollary}
\begin{proof}
This follows from Theorem \ref{thm:duality}, since if $\mr{char}\, \Bbbk=0$, then 
$\g \mapsto \U(\g)$ gives an equivalence from the category of (finite-dimensional) Lie
super-algebras to the category of hyper-super-algebras (of finite-type), and hence
every hyper-super-algebra is smooth; see Proposition \ref{prop:smooth_super-hyper}. 
\end{proof}

Let $A$ be a Noetherian complete Hopf super-algebra, and let $C=A^{\star}$ be the corresponding 
hyper-super-algebra;
Note that the complete topological right $A$-super-comodules and the complete topological 
left $C$-super-modules (see Proposition \ref{prop:topological_comodule} for definition)
respectively form symmetric tensor categories; for both the tensor product is given by
$\hotimes$, and the unit object is the discrete $\Bbbk$ equipped with the trivial structure. 

\begin{prop}\label{prop:tensor_isomorphic}
The two symmetric tensor categories are isomorphic to each other. 
\end{prop}
\begin{proof}
This follows from Proposition \ref{prop:topological_comodule} for $R$ applied to the present $C$. 
\end{proof}


\section{Super-manifolds}\label{sec:super-mfd}

In what follows, $\Bbbk$ is supposed to be a complete field (see Section \ref{subsec:I1}) of 
characteristic $\ne 2$, unless otherwise stated.

\subsection{Convergent power series super-algebra}\label{subsec:convergent_power}

Given non-negative integers $m$ and $n$, the \emph{convergent power series super-algebra}
\begin{equation}\label{eq:convergent_power}
\Bbbk\{ \sT_1,\dots, \sT_{m+n} \} 
\end{equation}
in $m$ even and $n$ odd variables is the tensor product 
\[ 
\Bbbk\{ \sT_1,\dots, \sT_m\} \otimes \wedge(\sT_{m+1}, \dots,\sT_{m+n})
\]
of the convergent power series algebra $\Bbbk\{ \sT_1,\dots, \sT_m\}$ and
the exterior algebra $\wedge(\sT_{m+1}, \dots,\sT_{m+n})$
as in \eqref{eq:formal_power_series}. This is a smooth 
(or geometrically regular \cite[Theorem A.2]{MZ}) 
Noetherian connected super-algebra,
whose completion is the formal power series super-algebra 
$\Bbbk[[\sT_1,\dots,\sT_{m+n}]]$. 
Note that the variables $\sT_1,\dots, \sT_{m+n}$ give  
a basis of the maximal super-ideal $\m$ of $\Bbbk\{\sT_1,\dots \sT_{m+n} \}$
modulo $\m^2$.

\subsection{Super-manifolds}\label{subsec:super-mfd}
By a \emph{manifold} we mean an analytic manifold $X$ over $\Bbbk$
as defined in \cite[Part II, Chapter III, Section 2]{S}. Thus a \emph{chart} on $X$ is a triple $(U, \psi, n)$
consisting of an open sub-set $U$, a non-negative integer $n$ and a homeomorphism 
$\psi : U \to \psi(U)$ onto an open set of $\Bbbk^n$, where $n$ is not required to be constant. 
In addition, the underlying topological space $|X|$ of 
$X$ is not assumed to be Hausdorff or second countable. 

Let $X$ be a manifold. We let $\cF=\cF_X$ denote the associated sheaf of analytic functions. 
Let $\ms{W}$ be a finite-dimensional vector space. We let 
\begin{equation}\label{eq:ssmfd}
\cF \otimes \wedge(\ms{W})
\end{equation}
denote the sheaf of super-algebras on $|X|$ which associates $\cF(U)\otimes \wedge(\ms{W})$
to every open set $U$. 

\begin{definition}\label{def:super-mfd}
(1)\
A \emph{super-ringed space} is a pair $X=(|X|,\cO_X)$ of a topological space $|X|$ 
and a sheaf $\cO_X$ of super-algebras on $|X|$; 
by convention, $\cO(\emptyset)=0$. We will often write $\cO$ for $\cO_X$.
The super-ringed spaces form a category $\ms{SRS}$.
A morphism $\phi=(|\phi|, \phi^*) : X \to Y$ in $\ms{SRS}$ consists of
a continuous map $|\phi| : |X| \to |Y|$  
and a sheaf-morphism $\phi^* : \cO_Y \to \phi_*\cO_X$ of
super-algebras.  
 
(2)\ 
The pair $(|X|, \cF\otimes \wedge(\ms{W}))$ obtained as above from a manifold $X$
is a super-ringed space. 
A super-ringed space of this form is called a \emph{strongly split super-manifold}, or an 
\emph{s-split super-manifold}, in short; compare with the notion of \emph{split super-manifold}
(see \cite{V}, for example), 
whose structure sheaf is only required to be the exterior algebra on some \emph{locally free} $\mathcal{F}$-module. 
Let $\ms{SSMFD}$ denote the full subcategory of $\ms{SRS}$ which consists of all s-split
super-manifolds; it includes the category $\ms{MFD}$ of manifolds as a full subcategory. 

(3)\
A \emph{super-manifold} is a super-ringed space $X=(|X|, \cO)$ which is locally isomorphic to
an s-split super-manifold, or namely, which has an open covering $|X|=\bigcup_{i\in I}U_i$
such that each $(U_i, \cO|_{U_i})$ is isomorphic to an s-split super-manifold. 
For a non-empty open set  $U$ of $|X|$, $(U, \cO|_U)$ is obviously a super-manifold, and such is called
an \emph{open super-submanifold} of $X$. 
Let $\ms{SMFD}$ denote the full subcategory $\ms{SRS}$ which consists
of all super-manifolds; it includes $\ms{SSMFD}$ and $\ms{MFD}$ as full subcategories. 
\end{definition}

We thus have the chain of categories
\[
\ms{MFD}\subset \ms{SSMFD}\subset \ms{SMFD}\subset \ms{SRS}.
\]
The category $\ms{SRS}$ has the
terminal object $(\{*\}, \Bbbk)$, the one-point set $\{*\}$ equipped with the sheaf
$\{*\} \mapsto \Bbbk$. 
Contained in $\ms{MFD}$, this is necessarily a terminal object of
each of the rest. 

Given a super-manifold $X=(|X|,\cO)$, 
let $\cO_{\red}$ denote the sheafification of the presheaf on $|X|$ which associates 
$\cO(U)_{\red}\, (=\overline{\cO(U)})$ to an open set $U$; see \eqref{eq:Ared}. 
If $U$ is included in an open set $V$ such that $(V,\cO|_V)$ is an s-split super-manifold, then
$\cO_{\red}(U)=\cO(U)_{\red}$. 
One sees that
$(|X|, \cO_{\red})$ is a manifold, which will be presented as
\[ 
X_{\red}=(|X|, \cO_{\red}).
\]
We see easily the following.

\begin{lemma}\label{lem:red}
$X \mapsto X_{\red}$ gives rise to a functor
\[
(\ )_{\red}: \ms{SMFD}\to \ms{MFD},
\]
which, restricted to $\ms{MFD}$, is the identity functor. 
The functor is right adjoint to the
inclusion functor $\ms{MFD}\to \ms{SMFD}$, or explicitly, we have
\begin{equation}\label{eq:right_adjoint}
\operatorname{Mor}(Y,X)=\operatorname{Mor}(Y,X_{\red})
\end{equation}
if $Y \in \ms{MFD}$. 
\end{lemma}

Here and in what follows we let
$\operatorname{Mor}(Y,X)$ denote the set of all morphisms $Y \to X$ in $\ms{SMFD}$.  

Let $X=(|X|, \cO)$ be a super-manifold, and let $x \in |X|$. If one chooses an s-split open 
super-submanifold $(U, \cF\otimes \wedge(\sW))$ of $X$ such that $x \in U$, then 
the stalk $\cO_x$ at $x$ is identified with $\cF_x \otimes \wedge(\sW)$, which is
isomorphic to the convergent power series super-algebra
$\Bbbk\{\sT_1,\dots, \sT_{m+n}\}$
in $m$ even and $n$ odd variables (see \eqref{eq:convergent_power}), 
where 
\[
m=\mr{Kdim}\, \cF_x,\quad n=\dim \sW;
\]
$\mr{Kdim}$ indicates the Krull dimension. One sees easily
\begin{align}\label{eq:red_x}
\begin{split}
(\cO_x)_{\red}&= (\cO_{\red})_x \simeq \Bbbk\{\sT_1,\dots, \sT_m\},\\
(\hO_x)_{\red}&= (\hO_{\red})_x \simeq \Bbbk[[\sT_1,\dots, \sT_m]]. 
\end{split}
\end{align}
Let $\m_x$ (or $\m_{X,x}$) denote the maximal super-ideal of $\cO_x$. 
The \emph{tangent super-space} $T_xX$ of $X$ at $x$ 
is the dual super-vector space 
\[
T_xX=(\m_x/\m_x^2)^*
\]
of $\m_x/\m_x^2$, the \emph{cotangent super-vector space}. We have
\[
m=\dim (T_xX)_0,\quad n=\dim (T_xX)_1.
\]

\begin{definition}\label{def:super-dimension}
We present as
\[
m|n= \sdim_x X,
\] 
and call this the \emph{super-dimension} of $X$ at $x$.
\end{definition}
 
One sees 
that $x \mapsto \sdim_x X$ is locally constant; 
see \cite[Part II, Chapter 3, Section 2]{S}. We have 
\begin{equation}\label{eq:TXzero}
(T_xX)_0=T_x(X_{\red}),
\end{equation} 
and so the dimension $m$ equals the dimension $\dim_x(X_{\red})$
of the manifold $X_{\red}$ at $x$, as defined in \cite{S}. 

We see easily the following.

\begin{lemma}\label{lem:split_smfd}
A super-manifold $X=(|X|, \cO)$ is s-split if and only if $\cO\simeq \cO_{\red}\otimes \wedge(\ms{W})$,
where $\ms{W}=(T_xX)_1^*$ for some/any point $x \in|X|$.
\end{lemma}

If these equivalent conditions are satisfied, we will often write
$\cF \otimes \wedge(\ms{W})$ for $\cO=\cO_X$ (as we have done so far), supposing
$\cF=\cF_{X_{\red}}\, (=\cO_{\red})$. 

Let $\phi : X \to Y$ be a morphism of super-manifolds. 
Let $x\in |X|$, and set $y=|\phi|(x)$. 
The sheaf-morphism $\phi^* : \cO_Y \to \phi_*\cO_X$ induces
a super-algebra map
$\phi^*_x : \cO_{Y,y} \to \cO_{X,x}$. This $\phi^*_x$ is necessarily \emph{local} in the sense
$\phi^*_x(\m_{Y,y})\subset \m_{X,x}$. Therefore, it induces a super-linear map
$\m_{Y,y}/\m_{Y,y}^2 \to \m_{X,x}/\m_{X,x}^2$. 
The dual super-linear map is denoted by
\[ 
d\phi_{x} : T_{x}X \to T_{y}Y.
\]

The following is the Inverse Function Theorem for super-manifolds; cf \cite[Proposition 5.1.1]{CCF}. 
This basic result will not be used in what follows, but may be worth to record here.
In addition, the same argument as proving the theorem will be 
used to prove Lemma \ref{lem:isomorphism}, which will be used at the final step
of proving Theorem \ref{thm:equivalence}. 

\begin{theorem}\label{thm:IFT}
Let $\phi : X \to Y$ be a morphism of super-manifolds. Let $x \in |X|$, and set $y=|\phi|(x)$. 
Assume that
$d\phi_{x} : T_{x}X \to T_{y}Y$ is an isomorphism. 
Then there exist open neighborhoods $U$ of $x$ and $V$ of $y$, such that
$\phi$ restricts to an isomorphism $U \to V$ of super-manifolds.  
\end{theorem}
\begin{proof}
Note that $(d\phi_x)_0$ coincides with $d(\phi_{\red})_x : T_x(X_{\red}) \to T_y(Y_{\red})$. 
The Inverse Function Theorem \cite[Theorem 2, Page 83]{S} for manifolds, applied to this now isomorphism,
allows us to suppose, replacing $|X|$ and $|Y|$ with some
open sub-sets, that $\phi_{\red} : X_{\red} \to Y_{\red}$ is an isomorphism of manifolds and, moreover,
$X$ and $Y$ are s-split so that
\[ \cO_X = \cF_X \otimes \wedge(\sW_X), \quad \cO_Y=\cF_Y \otimes \wedge(\sW_Y). \]
(For simplicity we write $\cF_X$, $\cF_Y$ as above, which should be denoted by $\cF_{X_{\red}}$, 
$\cF_{Y_{\red}}$, to be precise; see the paragraph following Lemma \ref{lem:split_smfd}.)

We wish to prove the sheaf-morphism $\phi^* : \cO_Y \to \phi_*\cO_X$,  restricted to $V=|\phi|(U)$, 
is an isomorphism, where $U \subset |X|$ is some open neighborhood of $x$.
This is equivalent to proving that the map $\phi^*_z : \cO_{Y,w} \to \cO_{X,z}$ of connected super-algebras
is an isomorphism for each $z \in U$ with $w=|\phi|(z)$. 
Applying the functor $\ms{gr}$ given in Section \ref{subsec:associated_super-algebras}, we have an 
$\mathbb{N}$-graded algebra map 
$\ms{gr}(\phi^*_z) : \cF_{Y,w} \otimes \wedge(\sW_Y)\to \cF_{X,z} \otimes \wedge(\sW_X)$; 
note that we are in the special situation as noted in Remark \ref{rem:gr}. 
We see that $\phi^*_z$ is isomorphic if and only if $\ms{gr}(\phi^*_z)$ is isomorphic if and only
if the maps $\ms{gr}(\phi^*_z)(0)$ and $\ms{gr}(\phi^*_z)(1)$ 
in degree $0$ and $1$ are isomorphic. 
The first map in degree $0$ is identified with $(\phi_{\red})^*_z$, and is, therefore,
isomorphic.

Let us turn to the second map 
$\ms{gr}(\phi^*_z)(1) : \cF_{Y,w} \otimes \sW_Y \to \cF_{X,z} \otimes \sW_X$ in degree $1$.
This is regarded as a morphism of finitely generated free modules
over a connected algebra, when we identify the connected algebra 
$\cF_{Y,w}$ with $\cF_{X,z}$ via $\ms{gr}(\phi^*_z)(0)$. 
When $z=x$, this $\ms{gr}(\phi^*_x)(1)$ is isomorphic, 
since the map modulo the maximal ideal is dual to 
the isomorphism $(d\phi_x)_1$, and is, therefore, isomorphic. 
In particular, it follows that $\dim \sW_X=\dim \sW_Y$; we denote this dimension by $n$.  

Associated with the $\phi^*$ above is a sheaf-morphism 
$\ms{gr}(\phi^*) : \cF_Y \otimes \wedge(\sW_Y)\to (\phi_{\red})_*\cF_X \otimes \wedge(\sW_X)$ 
of (super-commutative) $\mathbb{N}$-graded algebras.
Passing to stalks we have the $\mathbb{N}$-graded algebra maps $\ms{gr}(\phi^*)_z$, $z \in |X|$.
They coincide with the $\ms{gr}(\phi^*_z)$ before,
since in the present situation, taking $\ms{gr}$ commutes with the construction of stalks. 
The morphism $\ms{gr}(\phi^*)(0)$ in degree $0$, coinciding with $(\phi_{\red})^*$, is
isomorphic. We suppose that $|\phi_{\red}|$ is the identity map, and identify the two sheaves so that
$\cF_X=(\phi_{\red})_*\cF_X=\cF_Y$. By choosing bases of 
$\sW_X$ and of $\sW_Y$, the morphism 
$\ms{gr}(\phi^*)(1): \cF_Y \otimes \sW_Y \to \cF_X \otimes \sW_X$ in degree $1$ is 
presented by
an $n \times n$ matrix with entries in $\cF_X$. 
Let $\det \in \cF_X(|X|)$ denote the determinant of the matrix evaluated in $|X|$. 
Given $z \in |X|$, we see that $\ms{gr}(\phi^*_z)(1)$ is isomorphic if and only if the natural 
image of $\det$ in $\cF_{X,z}$ is invertible if and only if the analytic function 
$\det : |X| \to \Bbbk$ does not vanish at $z$. Since this is the case for $x$, 
there exists an open neighborhood $U$ of $x$ such that for every $z\in U$, 
the last condition and hence the equivalent other conditions are all satisfied.
This completes the proof. 
\end{proof}

The following is crucial to define super Lie groups; see Definition \ref{def:super_Lie_group}. 

\begin{prop}\label{prop:finite_products}
The category $\ms{SMFD}$ of super-manifolds has finite products. 
\end{prop}
\begin{proof}
This follows,  
by modifying the proof of the result \cite[Remark 2.6 (iv)]{DM} in the $C^{\infty}$ situation,
so as to fit in with our analytic situation including the positive characteristic case.
Let us give some details below.

The category contains the terminal object $(\{ *\}, \Bbbk)$, as was seen before.

It remains to show the existence of the product $X_1\times X_2$ of two super-manifolds $X_1$, $X_2$.
One sees that the s-split super-manifold
$(|X|, \cF\otimes \wedge({\sW}))$ decomposes into the (categorical) product
\begin{equation}\label{eq:product}
(|X|, \cF\otimes \wedge({\sW}))=(|X|, \cF) \times (\{*\}, \wedge(\sW)),
\end{equation} 
where $Y:=(\{*\},\wedge(\sW))$ is the one-point set $\{*\}$ equipped with the sheaf 
$\{*\}\mapsto \wedge(\sW)$. 
Essential for this is to see that
given a super-manifold $Z=(|Z|,\cO)$, the set $\operatorname{Mor}(Z, Y)$ of the morphisms is 
identified with the set $\Hom(\sW, \cO(|Z|)_1)$ of all linear maps from $\sW$ to the odd component of $\cO(|Z|)$.
This also shows that the product $(\{*\}, \wedge(\sW_1)) \times (\{*\}, \wedge(\sW_2))$ 
is given by $(\{*\}, \wedge(\sW_1\oplus \sW_2))$.
Therefore,
we may suppose that $X_i$, $i=1,2$, are manifolds, and are, moreover (by gluing), 
open sub-manifolds of some $\Bbbk^{m_i}$, and it suffices to show that their product $X_1\times X_2$
in $\ms{MFD}$ gives the desired product. 

Suppose that $X$ is a non-empty open sub-manifold of $\Bbbk^m$, and let $Z$ be as above. 
The desired result follows from the following.
\medskip

\noindent
\textbf{Claim.}\
\emph{$\operatorname{Mor}(Z, X)$ is in a natural one-to-one
correspondence with the set 
\[
\{ \, (f_1,\dots,f_m) \in \cO(|Z|)_0^m \mid 
(\overline{f}_1(z),\dots,\overline{f}_m(z))\in |X|~~\text{for all}~~z \in |Z|  \, \},
\]
where $\overline{f}_i$ denotes the natural image of $f_i$ in $\cO_{\red}(|Z|)$.}
\medskip

This is proved by \cite[Proposition 2.4]{DM} in the $C^{\infty}$ situation. 
Choose $(f_1,\dots,f_m)$ from the set above. 
Note from \cite{DM} that the main point of the
proof of Claim
is to construct the corresponding morphism $Z \to X$. Let us confirm only that the construction
indeed works in our situation. 
We have a morphism of manifolds,
\[
\phi : Z_{\red}\to X,\quad \phi(z)=(\overline{f}_1(z),\dots,\overline{f}_m(z)).
\]
The associated sheaf-morphism $\phi^* : \cF \to \phi_*\cO_{\red}$
is given by
\begin{equation}\label{eq:assoc_morphism}
\phi^* : \cF(U) \to \cO_{\red}(V),\quad \phi^*(F)=F(\overline{f}_1,\dots,\overline{f}_m),
\end{equation}
where $U \subset |X|$ and $V \subset |Z|$ are open with $\phi(V)\subset U$. 
The desired morphism $(\phi, \widetilde{\phi^*})$ consists of the same continuous map $\phi$ as above,
and the sheaf-morphism $\widetilde{\phi^*} : \cF \to \phi_*\cO$ given by
\begin{equation}\label{eq:desired_morphism}
\widetilde{\phi^*} : \cF(U) \to \cO(V),\quad \widetilde{\phi^*}(F)=F(f_1,\dots,f_m)
\end{equation}
where $U$ and $V$ are as above. 
In \eqref{eq:assoc_morphism}, \eqref{eq:desired_morphism} and in what follows we wrote and will write
$f_i$, $\overline{f}_i$ for $f_i|_V$, $\overline{f}_i|_V$.
Our task is to clarify what the composite $F(f_1,\dots,f_m)$ in \eqref{eq:desired_morphism} is. 

Now, we assume in addition that $V$ is s-split as a super-manifold, and proceed,
depending on \cite[Part II, Chaper II]{S} for basic results and the notation
which uses multi-indexes, $\boldsymbol{k}=(k_1,\dots, k_m)$,\ $k_i \ge 0$. 
We see from \cite[Lemma on Page 70]{S} the following.
\medskip

\noindent
\textbf{Fact.}\
\emph{For each $\boldsymbol{k}$, we have a unique operator $\varDelta^{\boldsymbol{k}} : \cF(U) \to \cF(U)$ 
that gives rise to the operator 
\begin{equation}\label{eq:Delta_operator}
\varDelta^{\boldsymbol{k}}(\sum_{\boldsymbol{i}} c_{\boldsymbol{i}} \, \sT^{\boldsymbol{i}})=
\sum_{\boldsymbol{i}\ge \boldsymbol{k}} \binom{\boldsymbol{i}}{\boldsymbol{k}}\, c_{\boldsymbol{i}}
\, \sT^{\boldsymbol{i}-\boldsymbol{k}}
\end{equation}
on $\cF_x=\Bbbk\{\sT_1,\dots,\sT_m\}$ at every $x \in U$, where
$\sT^{\boldsymbol{i}}=\sT_1^{i_1}\dots \sT_m^{i_m}$ and
$\binom{\boldsymbol{i}}{\boldsymbol{j}}=\binom{i_1}{j_1}\dots \binom{i_m}{j_m}$; 
see Example \ref{ex:Delta-operator}.}
\medskip

One may understand that the formula above 
defines $\varDelta^{\boldsymbol{k}} f$ locally, when $f\in\cF(U)$ is expressed
as a convergent power series in some strict polydisk. 
Since $V$ is s-split, we have the decomposition 
\[ f_i=\overline{f}_i + g_i,\quad 1 \le i \le m, \]
in $\cO(V)$, where $g_i \in \cO(V)$ are nilpotent even elements. 
Define $F(f_1,\dots,f_m)$ by
\[
F(f_1,\dots,f_m)
= \sum_{\boldsymbol{k}}\, \varDelta^{\boldsymbol{k}}\hspace{-0.2mm}F\hspace{0.2mm} (\overline{f}_1,\dots,\overline{f}_m)
\hspace{0.8mm} g^{\boldsymbol{k}},
\]
where $g^{\boldsymbol{k}}=g_1^{k_1}\dots g_m^{k_m}$. Note that this sum is finite.
The formula given in Part~3 of the last cited \cite[Lemma on Page 70]{S} shows that 
$F(f_1,\dots,f_m)$ is indeed the section over $V$ which assigns 
to every $z \in V$, the composite
$F_x((f_1)_z,\dots, (f_m)_z)\in \cO_z$, 
where $x =\phi(z)$. 

We see that
the last assignment, being of local nature, 
defines a section 
without the assumption that $V$ is s-split; apply the result with the assumption to 
open subsets which are s-split as super-manifolds. 
In conclusion,
$F(f_1,\dots,f_m)$ is the section thus defined over any $V$ with $\phi(V)\subset U$. 
\end{proof}

\begin{rem}\label{rem:product_smfd}
Let $X$ and $Y$ be super-manifolds, and set $Z=X\times Y$. 

(1) 
We have $Z_{\red}=X_{\red}\times Y_{\red}$ in $\ms{MFD}$, as is seen from \eqref{eq:right_adjoint}. 
In particular, $|Z|$ is the product $|X| \times |Y|$ of topological spaces. 

(2) 
Suppose that 
$X=(|X|,\cF_{X_{\red}}\otimes \wedge(\ms{W}_X))$ and $Y=(|Y|,\cF_{Y_{\red}}\otimes \wedge(\ms{W}_Y))$ 
are s-split. 
As is seen from the last proof, $Z$ is then s-split so that
\[ Z=\big(|X|\times |Y|,\ \cF_{X_{\red}\times Y_{\red}}\otimes \wedge(\ms{W}_X \oplus \ms{W}_Y)\big). \]
In addition, the canonical morphism from $Z$ to $X$, for example, is the projection
$\mr{pr}:|X|\times |Y| \to |X|$ equipped with the tensor product of (i)~the canonical sheaf-morphism 
$\cF_{X_{\red}} \to \mr{pr}_*\cF_{X_{\red}\times Y_{\red}}$ and (ii)~the super-algebra
map $\wedge(\ms{W}_X)\to \wedge(\ms{W}_X \oplus \ms{W}_Y)$ induced from the inclusion
$\ms{W}_X\hookrightarrow \ms{W}_X \oplus \ms{W}_Y$. 
\end{rem}

\begin{prop}\label{prop:identify}
Let $X$ and $Y$ be super-manifolds, and suppose $x\in |X|$ and $y\in |Y|$. 
The maps
\[ 
\hO_{X,x}\to \hO_{X\times Y, (x,y)} \leftarrow \hO_{Y,y},\quad
T_xX \leftarrow T_{(x,y)}(X\times Y) \to T_yY
\]
of complete connected super-algebras, and of super-vector spaces that are
induced from the canonical morphisms $X \leftarrow X \times Y \to Y$ naturally give rise to
isomorphisms
\begin{equation}\label{eq:isom_couple} 
\hO_{X,x}~\hotimes~\hO_{Y,y} \simeq \hO_{X\times Y,(x,y)},\quad
T_{(x,y)}(X\times Y) \simeq T_xX\oplus T_yY.
\end{equation}
\end{prop}
\begin{proof}
We may suppose that $X$ and $Y$ are s-split. The desired result then follows easily from Remark 
\ref{rem:product_smfd} (2). 
\end{proof}

Through the isomorphisms \eqref{eq:isom_couple} we will identify the relevant objects. 

\section{Super Lie groups}\label{sec:super_Lie_group}

\subsection{Super Lie groups and their hyper-super-algebras}\label{subsec:super_Lie_group}
The following definition makes sense by Proposition \ref{prop:finite_products}.

\begin{definition}\label{def:super_Lie_group}
A \emph{super Lie group} is a group object in $\ms{SMFD}$.
We let $\ms{SLG}$ denote the category of
those group objects. See Remark \ref{rem:super_Lie_morphism} (2) for the definition of morphisms. 
\end{definition}

We will see that every super Lie group is s-split as a super manifold; 
see Corollary \ref{cor:super_Lie_group_split}. The underlying topological space
of a super Lie group is a group, in particular. The identity element will be
denoted by $e$. 

By Remark \ref{rem:product_smfd} (2) the functor
$(\ )_{\red}: \ms{SMFD}\to \ms{MFD}$ given by Lemma \ref{lem:red}
preserves finite products, and so it induces a functor between the categories of the group
objects. This proves the following. 

\begin{lemma}\label{lem:associated_Lie}
If $G=(|G|,\cO)$ is a super Lie group, then
\[
G_{\red}=(|G|, \cO_{\red})
\]
is a Lie group,
that is, a group object in $\ms{MFD}$. The assignment $G \mapsto G_{\mr{red}}$ gives rise to
a functor from $\ms{SLG}$ to the category of Lie groups. 
\end{lemma}

We see easily the following.

\begin{lemma}\label{lem:constant_sdim}
Let $G=(|G|, \cO)$ be a super Lie group.
\begin{itemize}
\item[(1)]
Let $g\in |G|$. The left multiplication by $g$
\[ L_g : G \to G, \ L_g(h) = gh \]
is an isomorphism in $\ms{SMFD}$, which induces isomorphisms
\[
(L_g^*)_e : \cO_g \overset{\simeq}{\longrightarrow}\cO_e,\quad 
(dL_g)_e : T_eG \overset{\simeq}{\longrightarrow} T_gG.
\]
\item[(2)] $g \mapsto \mr{sdim}_g\, G$ is constant; see Definition \ref{def:super-dimension}. 
\end{itemize}
\end{lemma}

The last constant value is denoted by $\mr{sdim}\, G$.

Let $G=(|G|,\cO)$ be a super Lie group. 
The product and the inverse on the underlying group $|G|$ give
morphisms 
\[
\mu : G \times G\to G,\ \mu(g,h)=gh; \quad \iota : G \to G,\ \iota(g) = g^{-1}
\]
in $\ms{SMFD}$. The associated sheaf-morphisms are denoted by
\begin{equation}\label{eq:Delta_S}
\Delta =\mu^* : \cO \to \mu_*\cO_{G\times G},\quad S=\iota^* : \cO \to \iota_*\cO.
\end{equation}
Recall that for every $g \in |G|$, $\cO_g$ is a smooth Noetherian connected super-algebra
with the completion $\hO_g$. We let
\begin{equation}\label{eq:proj_to_residue}
\epsilon_g : \cO_g \to \Bbbk,\quad \hat{\epsilon}_g : \hO_g \to \Bbbk
\end{equation}
denote the projections to the residue field.
By \eqref{eq:isom_couple} we have the identification $\hO_g~\hotimes~\hO_h = \hO_{G\times G,(g,h)}$,
where $g,h \in |G|$. Using this we have the super-algebra map
\begin{equation}\label{eq:Delta_gh} 
\hat{\Delta}_{g,h} : \hO_{gh} \to \hO_g~\hotimes~\hO_h. 
\end{equation}
The following is easy to see.

\begin{lemma}\label{lemma:hatO}
$(\hO_e,\hat{\Delta}_{e,e})$ is a smooth Noetherian complete Hopf super-algebra
\textup{(}see Definition \ref{def:complete_Hopf}\textup{)},
which has $\hat{\epsilon}_e$ and $\hat{S}_e$ as its counit and antipode, respectively.
\end{lemma}

\begin{rem}\label{rem:super_Lie_morphism}
(1)\ Let $g, h \in |G|$. It is easy to show the formulas
\begin{gather}
(\widehat{L^*_g})_h=(\hat{\epsilon}_g \, \hat{\otimes} \, \mathrm{id}_{\hO_h})\circ \hat{\Delta}_{g,h},\quad
(\widehat{R^*_g})_h=(\mathrm{id}_{\hO_h}\, \hat{\otimes}\, 
\hat{\epsilon}_g)\circ \hat{\Delta}_{h,g}, \label{eq:L_Delta}\\
(\widehat{L^*_g})_h\circ \hat{S}_{gh}= \hat{S}_h\circ (\widehat{R^*_{g^{-1}}})_{h^{-1}}. \label{eq:L_S}
\end{gather}

(2)\ Let $G=(|G|,\cO)$, and $G'=(|G'|,\cO')$ be super Lie groups, and let $\Delta$, $S$ and $\Delta', S'$ 
denote the associated sheaf-morphisms as given by \eqref{eq:Delta_S}. 
A morphism $\phi : G \to G'$
in $\ms{SLG}$ is a morphism of super manifolds such that $|\phi|$ is a group homomorphism and
the sheaf-morphism $\phi^*:\cO' \to \phi_*\cO$ is compatible with $\Delta^{(')}$. We claim that
this $\phi^*$ is necessarily compatible with $S^{(')}$. Let $g'$ denote $|\phi|(g)$
for every $g\in |G|$. 
To prove the claim it suffices, passing to the completed stalks, to prove
\begin{equation*}\label{eq:compatible_with_S}
\widehat{S}_{g}\circ \hat{\phi}^*_{g^{-1}} =
\hat{\phi}^*_g\circ \widehat{S}'_{g'},\quad g \in |G|.
\end{equation*}
When $g=e$, this holds, since the super-algebra map $\hat{\phi}^*_e : \widehat{\cO}'_e \to \widehat{\cO}_e$
between the complete Hopf super-algebras
is compatible with the coproducts $\widehat{\Delta}^{(')}_{e,e}$, 
and hence is necessarily compatible with 
the antipodes $\widehat{S}^{(')}_e$. This, together with the formula \eqref{eq:L_S} applied when $h=e$,  
reduces the proof to proving the formulas
\begin{equation}\label{eq:LR_phi}
(\widehat{L^*_g})_e\circ \hat{\phi}^*_g = \hat{\phi}^*_e\circ (\widehat{L^*_{g'}})_e,\quad
(\widehat{R^*_g})_e\circ \hat{\phi}^*_g = \hat{\phi}^*_e\circ (\widehat{R^*_{g'}})_e
\end{equation}
for $g\in |G|$. By Remark \ref{rem:local_map} we have
$\hat\epsilon_{g}\circ \hat{\phi}^*_g=  \hat{\epsilon}_{g'}$ for $g \in |G|$. 
Since in addition, the compatibility of $\phi^*$ with $\Delta^{(')}$
is equivalent to 
\[
\widehat{\Delta}_{g,h}\circ \hat{\phi}^*_{gh} = (\hat{\phi}^*_g\, \hotimes\, \hat{\phi}^*_h)
\circ \widehat{\Delta}'_{g',h'},\quad g,\ h \in |G|,
\]
the desired formulas \eqref{eq:LR_phi} follow from \eqref{eq:L_Delta}. 

An alternative, shorter proof of the claim is the following. Note that $G$ represents
the group-valued functor $\operatorname{Mor}(-, G)$ defined on $\ms{SMFD}$. Since
$\phi^*$ is assumed to be compatible with $\Delta^{(')}$, the natural transformation
$\operatorname{Mor}(-, \phi):\operatorname{Mor}(-, G)\to \operatorname{Mor}(-, G')$ induced by $\phi$ preserves
the product, or namely, for every $X \in \ms{SMFD}$, the map 
$\operatorname{Mor}(X, \phi)$ between groups preserves the product. It necessarily 
preserves the inverse, and sends in particular, the inverse of $\mr{id}_G$ to the inverse of $\phi$. 
This shows the desired compatibility. 
\end{rem}

Let $G=(|G|,\cO)$ be a super Lie group. 
By Lemma \ref{lemma:hatO} and Theorem \ref{thm:duality},
$(\hO_e)^{\star}$ is a smooth hyper-super-algebra of finite type, and 
hence $P((\hO_e)^{\star})$ is a finite-dimensional Lie super-algebra. 

\begin{definition}\label{def:hy_Lie}
We call
\[
\hy(G) = (\hO_e)^{\star},\quad \Lie(G) =P((\hO_e)^{\star})
\]
the \emph{hyper-super-algebra} of $G$, and the \emph{Lie super-algebra} of $G$,
respectively. Note that $\Lie(G)$ is identified with $T_eG$. 
\end{definition}

The following is easy to see.

\begin{prop}\label{prop:hy_Lie} $\hy$ and $\Lie$ define functors from 
$\ms{SLG}$ to the category of smooth hyper-super-algebras of finite type, and to the category
of finite-dimensional Lie super-algebras, respectively. Given a morphism $\phi : G \to G'$
of super Lie groups, we have
\[
\hy(\phi)= (\phi^*_e)^{\star},\quad \Lie(\phi) = \hy(\phi)|_{\Lie(G)}= d\phi_e.
\]
\end{prop}

\begin{prop}\label{prop:hy_Lie_of_Gred}
Let $G$ be a super Lie group.
The hyper-algebra $\hy(G_{\red})$ of the Lie group $G_{\red}$ is the
largest purely even super-subcoalgebra $\underline{\hy(G)}$ of $\hy(G)$, 
while the Lie algebra $\Lie(G_{\red})$ of $G_{\red}$ is the
even component $\Lie(G)_0$ of $\Lie(G)$. We thus have
\begin{equation}\label{eq:hyGred}
\hy(G_{\red})=\underline{\hy(G)},\quad \Lie(G_{\red})=\Lie(G)_0.
\end{equation}
\end{prop}
\begin{proof}
This follows from Theorem \ref{thm:duality} and \eqref{eq:TXzero}. 
\end{proof}

\begin{lemma}\label{lem:isomorphism}
A morphism $\phi : G \to G'$ of super Lie groups is an isomorphism if and only if the associated morphisms
\[ \phi_{\mr{red}} : G_{\mr{red}} \to G'_{\mr{red}},\quad 
\Lie(\phi)_1 : \Lie(G)_1 \to \Lie(G')_1 \]
of Lie groups and of super-vector spaces 
are both isomorphisms. 
\end{lemma}
\begin{proof} 
To prove the non-trivial ``if" part assume that $\phi_{\mr{red}}$ and $\Lie(\phi)_1$ are isomorphisms. 
By the assumption on $\phi_{\mr{red}}$ it suffices to prove that the sheaf-morphism 
$\phi^* : \cO_{G'} \to \phi_*\cO_G$ is locally isomorphic, in fact at $e$ by 
Lemma \ref{lem:constant_sdim} (1).
The same argument as proving Theorem \ref{thm:IFT} (see the last two paragraphs of the proof) shows:
(1)~$\ms{gr}(\phi^*_e) : \ms{gr}(\cO_{G',e}) \to \ms{gr}(\cO_{G,e})$ is isomorphic
since it is so in degree $0$, $1$ (as for degree $1$, this is ensured by the assumption on $\Lie(\phi)_1$), 
and (2)~the result implies the desired local-isomorphy. 
\end{proof}

\subsection{The Hopf algebra of representative functions}\label{subsec:repres_function}

In this subsection we let $F=(|F|,\cF)$ be a Lie group. 

A matrix representation $\pi: F \to \mathrm{GL}_r(\Bbbk)$ of $F$ is said to be \emph{analytic}, if the composites 
$\op{pr}_{ij}\circ \, \pi$, where $\op{pr}_{ij} : \mathrm{GL}_r(\Bbbk)\to \Bbbk$ projects
the $(i,j)$-th entry, 
$1\le i,j\le r$, are all analytic, or are contained in $\cF(|F|)$. 
A finite-dimensional left super-module $\ms{M}$ over the group $F$ 
is said to be \emph{analytic}, if the matrix representation associated with $\ms{M}$ with respect to some/any basis
is analytic. A left $F$-super-module
is said to be \emph{analytic}, if it is a union of finite-dimensional $F$-super-submodules which are analytic. 
An \emph{analytic right $F$-super-module} is defined analogously. 

\begin{convention}\label{con:F-super-module}
By \emph{left} or \emph{right} $F$-(\emph{super-})\emph{modules}
we will mean analytic ones until the end of Section \ref{sec:equivalence}.
\end{convention}

Those left (resp., right) $F$-super-modules naturally form a symmetric tensor 
category, which we denote by
\begin{equation}\label{eq:F-SMod}
F\text{-}\mathsf{SMod}\quad (\text{resp.,}\ \mathsf{SMod}\text{-}F). 
\end{equation}
Recall the analogous notation \eqref{eq:SMod_SComod}, which will be used soon below.

Let $\cR(F)$ denote the (possibly, infinitely generated) Hopf algebra 
of all analytic representative functions on $F$; this is by definition the sum $\sum_{\pi}\mathrm{Im}\, \pi^*$
of the images $\mathrm{Im}\, \pi^*$ in $(\Bbbk F)^*$, where $\pi$ runs over all analytic matrix representations
$F \to \mathrm{GL}_r(\Bbbk)$ of arbitrary degree $r$, and the duals of their linearizations 
$\Bbbk F \to \mathrm{M}_r(\Bbbk)$
are denoted by $\pi^* : \mathrm{M}_r(\Bbbk)^* \to (\Bbbk F)^*$.  
This $\cR(F)$ is a Hopf subalgebra of the \emph{dual Hopf algebra} $(\Bbbk F)^{\circ}$ of $\Bbbk F$;
see \cite[Section 6.2]{Sw}. 
Since $\Bbbk F$, being spanned by group-likes, is cocommutative, $\cR(F)$ is commutative.
We have the canonical isomorphisms
\begin{equation}\label{eq:tensor_isom}
F\text{-}\mathsf{SMod}\simeq \mathsf{SComod}\text{-}\cR(F),\quad 
\mathsf{SMod}\text{-}F\simeq \cR(F)\text{-}\mathsf{SComod}. 
\end{equation}
of symmetric tensor categories.

We have $\cR(F) \subset \cF(|F|)$. This, composed with the canonical map $\cF(|F|) \to \hF_e$,
leads to an algebra map $\cR(F) \to \hF_e$. 

\begin{lemma}\label{lem:Hopf_pairing}
The thus obtained $\cR(F) \to \hF_e$ is a Hopf algebra map
\textup{(}see Remark \ref{rem:complete_Hopf} (3)\textup{)}, through which a \textup{(}left or right\textup{)}
$\cR(F)$-super-comodule turns into a discrete topological $\hF_e$-super-comodule.
The construction is functorial, and results as symmetric tensor functors
\begin{equation}\label{eq:FhyF}
F\text{-}\mathsf{SMod}\to \hy(F)\text{-}\mathsf{SMod},\quad 
\mathsf{SMod}\text{-}F\to \mathsf{SMod}\text{-}\hy(F).
\end{equation}
\end{lemma}
\begin{proof}
Note that $\cR(F) \otimes \cR(F)$ is naturally embedded into $\cF_{F \times F}(|F| \times |F|)$.
Since the coproduct $\Delta(p)$ of $p \in \cR(F)$, regarded as an analytic function on $|F| \times |F|$,
takes the value $p(gh)$ at $(g,h) \in |F|\times |F|$, the map $\cR(F) \to \hF_e$ preserves the coproduct.
It is easily seen to preserve the counit, as well. 
The functors arise by Proposition \ref{prop:topological_comodule}, since $\hF_e=\hy(F)^*$. 
\end{proof}

\begin{rem}\label{rem:explicit_induced_action}
(1)\
One may present the construction above alternatively, as follows. The map $\cR(F)\to \hF_e=\hy(F)^*$
naturally derives a \emph{Hopf pairing}
\begin{equation}\label{eq:Hopf_pairing}
\langle \ , \ \rangle :\hy(F) \times \cR(F) \to \Bbbk.
\end{equation}
Thus, this is a bilinear form which satisfies
\[ \langle a, pq \rangle = \langle a_{(1)}, p\rangle\langle a_{(2)},q\rangle,
\quad \langle a,1\rangle =\epsilon(a),\quad a \in \hy(F),\ p, q \in \cR(F) \]
and the mirror-symmetric conditions. 
To an object $\sV \in \mathsf{SMod}\text{-}F$ given by a left $\cR(F)$-super-comodule structure
\begin{equation}\label{eq:-1_0_sigma_notation}
\sV \to \cR(F) \otimes \sV, \ v \mapsto v^{(-1)} \otimes v^{(0)},
\end{equation}
the second functor above associates the right $\hy(F)$-super-module structure
\begin{equation}\label{eq:hy(F)-action}
v \triangleleft a = \langle a, v^{(-1)}\rangle v^{(0)},\quad v \in \sV,\ a \in \hy(F). 
\end{equation} 

(2)\ Let $\f =\Lie(F)$. 
We have a canonical hyper-algebra map
$\U(\f) \to \hy(F)$; see \eqref{eq:U(g)toH}.
The associated, restriction functors extend
the symmetric tensor functors \eqref{eq:FhyF} so as (for the second)
\[
\mathsf{SMod}\text{-}F \to \mathsf{SMod}\text{-}\hy(F) \to \mathsf{SMod}\text{-}\U(\f). 
\]
\end{rem}

\subsection{The $\hy(G)$-action on $\cO$}\label{subsec:action_by_hy(Gred)}
Let $G=(|G|,\cO)$ be a super Lie group.
Let $g \in |G|$. 
By $\hat{\Delta}_{g,e} : \hO_{g}\to \hO_{g}~\hotimes~\hO_{e}$,\ 
$\hO_g$ is a topological right $\hO_e$-super-comodule;
it is in fact an algebra-object in the tensor category of those super-comodules. 
By Proposition \ref{prop:tensor_isomorphic} for $A$ applied to $\hO_e$,
it follows that $\hO_g$ turns into a
left $\hy(G)$-super-module super-algebra. 
Here one should recall $\hy(G)=(\hO_e)^{\star}$,
and note that
every element $a \in \hy(G)$ gives a continuous linear map 
$\hO_e \to \Bbbk$. 
The result is formulated as follows. 

\begin{lemma}\label{lem:hy(G_red)-action}
Let $g \in |G|$. Let every element $a$ of $\hy(G)$ act on $\hO_g$ as the 
continuous super-linear endomorphism 
\[ 
\hO_g \overset{\hat{\Delta}_{g,e}}{\longrightarrow}\hO_g~\hotimes~\hO_e 
\overset{\mr{id} \hotimes a}{\longrightarrow} \hO_g~\hotimes~\Bbbk=\hO_g.
\]
Then $\hO_g$ turns into a left $\hy(G)$-super-module super-algebra.
\end{lemma}

\begin{prop}\label{prop:hy(G_red)-module}
Let $U \subset |G|$ be an open set. Through the natural embedding
$\cO(U) \to \prod_{g\in U} \hO_g$,\
$\cO(U)$ is a super-submodule
of the product $\prod_{g\in U} \hO_g$ of the
left $\hy(G)$-super-modules $\hO_g$, $g \in U$. Therefore, $\cO(U)$ turns into 
a left $\hy(G)$-super-module super-algebra. 
\end{prop}
\begin{proof}
Suppose $m|n=\dim G$. 
Let $G_{\red}=(|G|, \cF)$, where $\cF=\cO_{\red}$, denote the Lie group associated with $G$. 
We may suppose $\cO_e=\cF_e \otimes \wedge(\sW)$, where $\sW$ is a vector space of dimension $n$, and hence 
$\hy(G)=\cF_e^{\star}\otimes \wedge(\sW)^*$. Let $\n_e$ be the maximal ideal of $\cF_e$.

Let $a \in \hy(G)$. Then there exists a positive integer $\ell$ 
such that $a(\hat{\n}_e^{\ell}\otimes \wedge(\sW))=0$, or in other words, $a \in (\cF_e/\n_e^{\ell}\otimes \wedge(\sW))^*$. 
Set 
\begin{equation}\label{eq:Me}
\mathsf{M}_e:=\cF_e/\n_e^{\ell}\otimes \wedge(\sW).
\end{equation}
This $\ms{M}_e$ is a finite-dimensional super-vector space. 
To show that the action by $a$ stabilizes $\cO(U)$ in $\prod_{g\in U}\hO_g$, 
it suffices to prove that the composite
\[
\cO(U) \to \prod_{g\in U} \hO_g \xrightarrow{\hat{\Delta}_{g,e}}
\prod_{g\in U}(\hO_g~\hotimes~\hO_e) \to \prod_{g\in U} (\hO_g \otimes \mathsf{M}_e )
= \big(\prod_{g\in U} \hO_g\big) \otimes \mathsf{M}_e, 
\]
where the last arrow arises 
from the natural projection $\hO_e =\hF_e\otimes \wedge(\sW)\to \cF_e/\n_e^{\ell}\otimes \wedge(\sW)=
\mathsf{M}_e$ (see \eqref{eq:tensor_product_with_discrete}), 
takes values in $\cO(U) \otimes \mathsf{M}_e$. 
 
Let $s=\dim \ms{M}_e$. Given a basis $m_1,\dots, m_s$ of $\ms{M}_e$, 
let $p \in \cO(U)$, and denote its image in $\prod_{g\in U}\hO_g \otimes \ms{M}_e$ by
$\sum_{i=1}^s(p_{g,i})_{g\in U} \otimes m_i$ with $p_{g,i}\in \hO_g$. 
We wish to show that every $g_0 \in U$ has an open neighborhood 
$V\, (\subset U)$ such that for every $i$, the elements $p_{g, i}$, $g \in V$, arise
(necessarily uniquely) from some $p_{V, i}$ in $\cO(V)$. 

Let $g_0 \in U$. There exist neighborhoods $V$ of $g_0$ and $V'$ of $e$,
such that $V \subset U$ and $VV' \subset U$. 
The desired result will follow if we prove that the map
\begin{equation*}\label{eq:O(U)}
\cO(U)\to \prod_{g\in V} \hO_g \otimes \ms{M}_e
\end{equation*}
analogous to the composite above takes values in $\cO(V) \otimes \ms{M}_e$. 
Since the map naturally factors through
\begin{equation}\label{eq:O(VV')}
\cO_{G \times G}(V \times V') \to \prod_{g\in V} \hO_{G\times G,(g,e)}
= \prod_{g\in V} \hO_g~\hotimes~\hO_e
\to \prod_{g\in V} \hO_g \otimes \ms{M}_e,
\end{equation}
it suffices to prove that this last composite takes values in $\cO(V) \otimes \ms{M}_e$. 
We may re-choose $V$ and $V'$ so that they are open sub-manifolds of $\Bbbk^m$, and 
$(V, \cO|_V)$ and $(V',\cO|_{V'})$ are s-split, whence in particular,
\begin{equation}\label{eq:O(V)}
\cO(V) = \cF(V) \otimes \wedge(\ms{W}), 
\end{equation}
where $\ms{W}$ is as before. 

We use the operators $\varDelta^{\boldsymbol{k}}$ on $\cF(V\times V')\, (=\cF_{\Bbbk^{2m}}(V\times V'))$
used in the proof of Proposition \ref{prop:finite_products}. 
As ${\boldsymbol{k}}$ we use only the multi-indexes of the form $(\boldsymbol{0}, \boldsymbol{i})$,
where $\boldsymbol{0}=(0,\dots,0)$ and $\boldsymbol{i}=(i_1,\dots,i_m)$ both have $m$ entries. 
Suppose $\cF_e=\Bbbk\{ \ms{T}_1,\dots, \ms{T}_m\}$ naturally. Then $\cF_e/\n_e^{\ell}$ has
$\ms{T}^{\boldsymbol{i}}\, \mr{mod}\, \n_e^{\ell}$, $|\boldsymbol{i}|<\ell$, as a basis, where
$\ms{T}^{\boldsymbol{i}}=\ms{T}_1^{i_1}\dots \ms{T}_m^{i_m}$. 
Let 
\[ \sigma : \cF(V\times V')\to \cF(V) \otimes \cF_e/\n_e^{\ell} \] 
be the map which associates to each $f=f(g,h) \in \cF(V\times V')$, the sum
\[ 
\sum_{|\boldsymbol{i}| < \ell} \varDelta^{(\boldsymbol{0},\boldsymbol{i})}f|_{h=e} \, \otimes 
\, \ms{T}^{\boldsymbol{i}}\, \mr{mod}\, \n_e^{\ell}.
\] 
As is seen from \eqref{eq:Delta_operator}, if  
$(\varDelta^{(\boldsymbol{0},\boldsymbol{i})}f|_{h=e})_g\, (\in \cF_g)$ denotes the natural image
of $\varDelta^{(\boldsymbol{0},\boldsymbol{i})}f|_{h=e} \, (\in \cF(V))$, then 
the natural image 
$f_{(g,e)} \, (\in \cF_{(g,e)})$ of $f$ equals the sum 
$\sum_{\boldsymbol{i}}(\varDelta^{(\boldsymbol{0},\boldsymbol{i})}f|_{h=e})_g\ms{T}^{\boldsymbol{i}}$. 
We have proved that the natural map 
$\cF(V\times V')\to \prod_{g\in V} \hF_g \otimes \cF_e/\mathfrak{n}_e^{\ell}$ factors through $\sigma$, 
whence it takes values in $\cF(V) \otimes \cF_e/\mathfrak{n}_e^{\ell}$. 

The composite \eqref{eq:O(VV')} is identified with
\begin{align*}
\cF(V\times V')\otimes \wedge(\ms{W})\otimes \wedge(\ms{W})\overset{\sigma \otimes \mr{id}}{\longrightarrow} 
&~\cF(V) \otimes \cF_e/\n_e^{\ell}\otimes \wedge(\ms{W})\otimes \wedge(\ms{W})\\
&= \cO(V) \otimes \ms{M}_e \to \prod_{g\in V} \hO_g \otimes \ms{M}_e,
\end{align*}
where the last equality is given by \eqref{eq:O(V)} and \eqref{eq:Me}.
Therefore, the map \eqref{eq:O(VV')}
takes values in $\cO(V) \otimes \cF_e/\n_e^{\ell}$, as desired. 
\end{proof}

\begin{corollary}\label{cor:hy(G_red)-module_sheaf}
With each $\cO(U)$ given the structure of a left $\hy(G)$-super-module super-algebra as above,
$\cO$ is a sheaf of such super-algebras. 
\end{corollary}
\begin{proof}
Indeed, the given structure is compatible with the restriction maps $\cO(U) \to \cO(V)$ associated with 
an inclusion $U\supset V$ of open sets. 
\end{proof}

\begin{example}\label{ex:Delta-operator}
Let $m$ be a positive integer. Let $F=(\Bbbk^m,\cF)$ be the Lie group 
constructed on $\Bbbk^m$ with respect to the addition. Note that $\hF_e=\Bbbk[[\sP_1,\dots,\sP_m]]$
with $\sP_i$ primitive, and the hyper-algebra $\hy(F)$ of $F$ is
the commutative hyper-algebra with basis
\[
b_{\boldsymbol{k}}=b^{(1)}_{k_1}\dots b^{(m)}_{k_m},\quad \boldsymbol{k}=(k_1,\dots,k_m),\ k_i \ge 0
\]
such that 
\[ b_{\boldsymbol{k}}(\sP^{\boldsymbol{j}})=\delta_{\boldsymbol{k},\boldsymbol{j}},\ 
\text{where} \
\sP^{\boldsymbol{j}}=\sP_1^{j_1}\dots \sP_m^{j_m},\] 
and each sequence $b^{(i)}_0, b^{(i)}_1, 
b^{(i)}_2, \dots, \
1\le i \le m$, satisfies \eqref{eq:divided_power_sequence}. By Corollary \ref{cor:hy(G_red)-module_sheaf}
applied in the purely even situation,
$\cF$ is a sheaf of left $\hy(F)$-module algebras. 
The action $b_{\boldsymbol{k}} : \cF \to \cF$ by the basis element coincides with 
the operator $\varDelta^{\boldsymbol{k}}$, as is seen from the formula 
\[
\hat{\Delta}_{g,e}(\sT^{\boldsymbol{i}})
=\sum_{\boldsymbol{i}\ge \boldsymbol{j}} \binom{\boldsymbol{i}}{\boldsymbol{j}}\,
\sT^{\boldsymbol{i}-\boldsymbol{j}}\otimes \sP^{\boldsymbol{j}},\quad g \in \Bbbk^m
\]
compared with \eqref{eq:Delta_operator},
where we suppose $\hF_g=\Bbbk[[\sT_1,\dots,\sT_m]]$ naturally.
\end{example}


\subsection{The adjoint $G_{\red}$-action on $\hy(G)$}\label{subsec:adjoint_G-action}
Let $G=(|G|,\cO)$ be a super Lie group,
and let $G_{\red}=(|G|,\cF)$, where $\cF=\cO_{\red}$, denote the associated Lie group.  
Suppose $g \in |G|$. Recall from \eqref{eq:red_x} that $(\hO_g)_{\red}=\hF_g$, and let
\begin{equation}\label{eq:red_g}
\widehat{\red}_g : \hO_g \to \hF_g
\end{equation}
denote the natural projection.

The right $G$-adjoint action $G\times G\to G,\ (h, g) \mapsto g^{-1}hg$ on $G$ is a morphism
in $\ms{SMFD}$, which induces
a super-algebra map $\hO_e \to \hO_e~\hat{\otimes}~\hO_g$ for each $g \in |G|$. 
The composite of the last map 
with $\mr{id}~\hotimes\, \widehat{\red}_g$
is denoted by 
\[
\alpha_g : \hO_e \to \hO_e~\hat{\otimes}~\hF_g,
\]
whose composite with 
$\mr{id}\otimes \hat{\epsilon}_g: \hO_e~\hotimes~\hF_g\to \hO_e~\hotimes~\Bbbk=\hO_e$
is denoted by 
\begin{equation}\label{eq:beta_g}
\beta_g : \hO_e \to \hO_e.
\end{equation}
Since one sees that $\alpha_g$ is compatible with $\hat{\Delta}_{e,e}$ on $\hO_e$, it follows 
that 
$\beta_g$ is an endomorphism of the Noetherian complete Hopf algebra $\hO_e$.  
Let
\[
\gamma_g=(\beta_g)^{\star} : \hy(G) \to \hy(G) 
\]
denote the corresponding endomorphism of the hyper-super-algebra $\hy(G)$.

\begin{prop}\label{prop:Gred-super-module}
We have the following. 
\begin{itemize}
\item[(1)] 
By $\beta_g$, $g \in |G|$,\ the abstract group $G$ acts on $\hO_e$ from the left as automorphisms  
of the complete Hopf algebra.
\item[(2)]
By $\gamma_g$, $g \in |G|$,\ the Lie group $G_{\red}$ acts analytically on $\hy(G)$ from the right 
as automorphisms of the hyper-super-algebra. Thus, $\hy(G)$ 
turns into a Hopf-algebra object in $\ms{SMod}\text{-}G_{\red}$.
\item[(3)]
By Lemma \ref{lem:Hopf_pairing} there is induced a right $\hy(G_{\red})$-super-module
structure on $\hy(G)$, by which $\hy(G)$ turns, moreover, into a Hopf-algebra object in 
$\ms{SMod}\text{-}\hy(G_{\red})$. The induced action is precisely the right adjoint action
$\mr{ad}_r(a)$, $a \in \hy(G_{\red})$, given by
\begin{equation}\label{eq:adr}
\mr{ad}_r(a)(b) = S(a_{(1)})b a_{(2)},\quad  b \in \hy(G). 
\end{equation}
\end{itemize} 
\end{prop}
\begin{proof}
(1)\
We see
\[ 
\beta_g=(\widehat{L^*_{g^{-1}}})_e\circ (\widehat{R^*_{g}})_e
=(\widehat{R^*_{g}})_e\circ (\widehat{L^*_{g^{-1}}})_e,
\]
which implies that $\beta_e= \mr{id}$ and
$\beta_{gh}=\beta_g \circ \beta_h$,\ $g,h \in |G|$. This proves Part~1.

(2)\
The last equalities shows that $\gamma_e = \mr{id}$ and
$\gamma_{gh}=\gamma_h \circ \gamma_g$,\ $g,h \in |G|$. 

It remains to prove that the $G_{\red}$-action on $\hy(G)$ is analytic. 
Fix $\ell >0$. Then maps $\alpha_g$, $g\in |G|$, induce super-algebra maps
\[
\alpha'_g : \cO_e/\m_e^{\ell}\to \cO_e/\m_e^{\ell} \otimes \hF_g,\quad g \in |G|.
\]
Let $s=\dim \cO_e/\m_e^{\ell}$. Choose a basis $z_1,\dots, z_s$ of $\cO_e/\m_e^{\ell}$, and write 
$\alpha'_g(z_j)=\sum_{i=1}^s z_i\otimes p_{g,ij}$ with $p_{g,ij} \in \hF_g$. 
We have an open neighborhood $U$ of $e$ such that $z_j$, $1\le j \le s$, are all contained
in the image of the natural composite $\cO(U)\to \hO_e \to \cO_e/\m_e^{\ell}$. 
Let $g_0 \in |G|$.
Then there exist neighborhoods $U'$ of $e$ and $V$ of $g_0$, such that
$U' \times V$ is mapped to $U$ by the $G$-adjoint action, $(U',\cO|_{U'})$ is s-split, and
$U'_{\red}$ and $V$ are open sub-manifolds of $\Bbbk^m$, $m=\dim G_{\red}$. 
We see as in the proof of Proposition \ref{prop:hy(G_red)-module} that
the natural map $\cO_{G\times G_{\red}}(U'\times V) \to \cO_e/\mathfrak{m}_e^{\ell}\otimes \prod_{g\in V}\hF_g$
takes values in $\cO_e/\mathfrak{m}_e^{\ell}\otimes \cF(V)$. It follows that
the composite
\[
\cO(U) \to \hO_e \to \cO_e/\m_e^{\ell} \xrightarrow{(\alpha'_g)_{g\in V}}
\prod_{g\in V}\cO_e/\m_e^{\ell}\otimes \hF_g = \cO_e/\m_e^{\ell}\otimes \prod_{g\in V}\hF_g
\]
takes values in $\cO_e/\m_e^{\ell}\otimes \cF(V)$. Fix $1 \le i,j \le s$. It follows that 
$p_{g,ij}$, $g \in V$,
uniquely arise from some $p_{V,ij} \in \cF(V)$. 
Note that $V$ can run over an open covering of $|G|$.
All $p_{V,ij}$ are then glued together into an analytic function, say 
$p_{ij}$, on the whole $|G|$. Since
\[
\beta_g(z_j) =\sum_{i=1}^s\hat{\epsilon}_g(p_{g,ij})z_i
=\sum_{i=1}^s p_{ij}(g)z_i,
\]
the $G_{\red}$-action on $\cO_e/\m_e^{\ell}$ is given by the matrix representation 
$g \mapsto (p_{ij}(g))$, and is, therefore, analytic. Hence
the dual $G_{\red}$-action on $\hy(G)$ is analytic, as desired. 

(3)\
Let $F=G_{\red}$. We aim to prove that the induced $\hy(F)$-action $\hy(G) \otimes \hy(F) \to \hy(G)$
coincides with $(\alpha_e)^{\star}$, which implies Part~3. 
With the notation above, restrict the left $\cR(F)$-coaction on $\hy(G)$ 
to $(\cO_e/\m_e^{\ell})^*$. One then sees that the dual right $\cR(F)$-coaction
$\cO_e/\m_e^{\ell}\to \cO_e/\m_e^{\ell} \otimes \cR(F)$ induces, through the natural map 
$\cR(F) \to \hF_e$, the $\alpha'_e$ above.
Therefore, the induced 
$\hy(F)$-action restricted to $(\cO_e/\m_e^{\ell})^*$
coincides with $(\alpha'_e)^*$. This implies the desired result. 
\end{proof}

\begin{rem}\label{rem:adjoint_action}
We call the $G_{\red}$-action on $\hy(G)$ above, the \emph{right adjoint action}.
Note that by restriction, $\Lie(G)=P(\hy(G))$ turns into a Lie-algebra object 
in $\ms{SMod}\text{-}G_{\red}$. The restricted action as well is called
the \emph{right adjoint action}. It follow 
by Proposition \ref{prop:Gred-super-module} (3) that 
the induced $\Lie(G_{\red})$-action on $\Lie(G)$ coincides with the right adjoint action
for Lie super-algebras
\[
\mr{ad}_r(x)(y)=[y, x],\quad x \in \Lie(G_{\red}),\ y \in \Lie(G).
\]
\end{rem}


\section{Category equivalence $\mathsf{SLG} \approx \mathsf{HCP}$}\label{sec:equivalence}

\subsection{Functor to Harish-Chandra pairs}\label{subsec:Phi_to_HCP}
The concept of Harish-Chandra pairs goes back to Kostant \cite{Kostant}. 
The following version of definition is a reformulation of \cite[Definition 6.4]{M4} 
into our analytic situation. 

Let $(F, \sV)$ be a pair of a Lie group $F$ and a finite-dimensional 
right $F$-module $\sV$; see Convention \ref{con:F-super-module}. The action by $F$ on $\sV$ (or any action
by a Lie group on a super-module) will be presented as
$v^g$, where $g \in F,\ v\in \sV$. 
There is induced a right $\hy(F)$-module structure on $\sV$; this action as well as the restricted
one by $\Lie(F)=P(\hy(F))$ (see (P2) below)
will be indicated by $\triangleleft$, as in \eqref{eq:hy(F)-action}. 

\begin{definition}\label{def:HCP}
(1)\ The pair $(F, \sV)$ is called
a \emph{Harish-Chandra pair}, if it is equipped with 
an $F$-equivariant linear map $[\ , \ ] : \sV \otimes \sV \to \Lie(F)$,
where $\Lie(F)$ is supposed to be given the right adjoint $F$-action, 
such that 
\begin{itemize}
\item[(P1)] $[v, w]=[w, v]$,\quad $v, w \in \sV$,  
\item[(P2)] $v \triangleleft [v,v]= 0$,\quad $v \in \sV$.  
\end{itemize}

(2)\
A \emph{morphism} $(\phi, \psi) : (F_1,\sV_1) \to (F_2,\sV_2)$ of Harish-Chandra pairs consists 
of a morphism $\phi : F_1 \to F_2$ of Lie groups and a linear map $\psi : \sV_1\to \sV_2$ such that
\begin{itemize}
\item[(P3)] $\psi(v^g) =\psi(v)^{\phi(g)},\quad v\in \sV_1,\ g \in F_1$, 
\item[(P4)] $[\psi(v), \psi(w) ] =\Lie(\phi)([v, w])$, \quad $v, w \in \sV_1$.
\end{itemize}
Here, $\Lie(\phi)$ in (P4) may be replaced by $\hy(\phi)$; see 
Remark \ref{rem:explicit_induced_action} (2). 
The composite of morphisms is defined in the obvious manner. 

(3)\
$\mathsf{HCP}$ denotes the category of Harish-Chandra pairs.
\end{definition}

Recall that $\ms{SLG}$ denotes the category of super Lie groups.
Suppose $G=(|G|,\cO)\in \ms{SLG}$, and let $\g = \Lie(G)$ denote the Lie super-algebra of $G$. 
Recall from Remark \ref{rem:adjoint_action}
that $\g$ is a Lie-algebra object in $\mathsf{SMod}\text{-}G_{\red}$
with respect to the right adjoint action by the associated Lie group $G_{\red}$. 
Regard the odd component $\g_1$ of $\g$ as a right $G_{\red}$-module by restriction. 
The bracket on $\g$, restricted to $\g_1$, is a $G_{\red}$-equivariant linear map which
takes values in $\g_0$, the Lie algebra of $G_{\red}$; see \eqref{eq:hyGred}. 

\begin{prop}\label{prop:Phi}
The pair $(G_{\red},\g_1)$, equipped with the restricted bracket, turns into a Harish-Chandra pair. 
Moreover, this construction gives rise to the functor
\[ \Phi : \mathsf{SLG} \to \mathsf{HCP},\ \Phi(G)= (G_{\red},\Lie(G)_1). \]
\end{prop}
\begin{proof}
Obviously, (P1) is satisfied. By Remark \ref{rem:adjoint_action}, (P2) now reads
$[v, [v,v]]=0$, which holds in $\Lie(G)$. 
Since the assignment $G \mapsto (\Lie(G)\ \text{in}\ \mathsf{SMod}\text{-}G_{\red})$ 
is functorial, so is $\Phi$ as well. 
\end{proof}

\begin{theorem}\label{thm:equivalence}
The functor $\Phi : \mathsf{SLG} \to \mathsf{HCP}$ is an equivalence.
\end{theorem}

Modifying the construction in \cite{M3}
we will construct an explicit quasi-inverse functor. 

\subsection{Construction of a quasi-inverse}\label{subsec:quasi-inverse}
Let $(F,\sV) \in \mathsf{HCP}$. Let $\cF$ and $D=\hy(F)$ denote the structure sheaf of $F$
and the hyper-algebra of $F$, respectively. 

By \cite[Theorem 10]{M3} we have a unique hyper-super-algebra $C$ with the properties: 
(1)~$\underline{C}=D$,\ $P(C)_1=\sV$, and (2)~$C$  
is generated by $D$ and $\sV$ (to be more precise, $C$ is generated by $\sV$ as a $D$-ring) and 
is defined by the relations
\begin{equation}\label{eq:relations_H}
v a = a_{(1)}(v \triangleleft a_{(2)}), \quad vw+wv =[v,w], 
\end{equation}
where  $v, w \in \sV$ and $a \in D$. Choosing arbitrarily a basis $(v_i)_{1\le i \le n}$, $n=\dim \sV$, of $\sV$, 
we have as in \eqref{eq:isom_module_coalgebra}, the unit-preserving isomorphism of left $D$-super-module 
super-coalgebras
\begin{equation*}
D \otimes \wedge(\sV)\overset{\simeq}{\longrightarrow}C,\quad 
a \otimes (v_{i_1}\wedge v_{i_2}\wedge \dots \wedge v_{i_r})
\mapsto a\, v_{i_1}v_{i_2}\dots v_{i_r},
\end{equation*}
where $1\le i_1 <i_2<\dots<i_r \le n$, $0\le r \le n$. Fixing an isomorphism as above we will identify 
so as 
\begin{equation}\label{eq:H}
C=D \otimes \wedge(\sV). 
\end{equation}

By Proposition \ref{prop:Gred-super-module} (2) applied in the purely even situation,  
$F$ acts analytically on $D$ from the right as Hopf-algebra automorphisms.   

\begin{lemma}\label{lem:H_SMod-F}
This $F$-action on $D$, together with the original one on $\sV$, uniquely extend onto $C$ so that
$C$ turns into a Hopf-algebra object in $\mathsf{SMod}\text{-}F$.
\end{lemma}

The extended action as well will be presented as $a^g$, where $a \in C$, $g\in F$. 

Recall from \eqref{eq:tensor_isom} that $\sV$ is 
a left comodule over the
Hopf algebra $\cR(F)$ of analytic representative functions on $F$, whose structure we present
as $v \mapsto v^{(-1)}\otimes v^{(0)}$; see \eqref{eq:-1_0_sigma_notation}.  

\begin{proof}
It suffices to show that the relations given in \eqref{eq:relations_H} are compatible
with the $F$-action. For the second relation this follows from the $F$-equivariance of $[ \ , \ ]$. 
We wish to prove $(v \triangleleft \, a)^g= v^g \triangleleft \, a^g$ for
$g \in F$, $a \in D$ and $v \in \sV$; this implies the compatibility for the first relation.
Given $p\in \cR(F)$, Lemma \ref{lem:Hopf_pairing} (see also \eqref{eq:F-balanced} below) shows
\begin{align}\label{eq:agp1}
\begin{split}
&\langle a^g, p \rangle = p_{(1)}(g^{-1})\, \langle a, p_{(2)} \rangle \, p_{(3)}(g)\\
&\text{or}\quad
p_{(1)}(g)\, \langle a^g, p_{(2)} \rangle 
=\langle a, p_{(1)} \rangle \, p_{(2)}(g),
\end{split}
\end{align}
where $\langle \ , \ \rangle$ denotes the pairing \eqref{eq:Hopf_pairing}. 
If we write $v^{(-2)}\otimes v^{(-1)}\otimes v^{(0)}$ for $\Delta(v^{(-1)})\otimes v^{(0)}$, 
then the desired equality follows from the last equation and \eqref{eq:hy(F)-action}, so as
\begin{align*}
v^g \triangleleft \, a^g&=v^{(-2)}(g) \, \langle a^g, v^{(-1)}\rangle \, v^{(0)}
=\langle a, v^{(-2)}\rangle \, v^{(-1)}(g)\, v^{(0)}\\
&=\langle a, v^{(-1)}\rangle (v^{(0)})^g
=(v \triangleleft \, a)^g. 
\end{align*}
\end{proof}

By Corollary \ref{cor:hy(G_red)-module_sheaf}, 
$\cF$ is a sheaf of left $D$-module algebras.
Given an open set $U \subset |F|$, let 
$\cP(U) = \Hom_{D}(C, \cF(U))$
denote the super-vector space of all left $D$-module maps $C \to \cF(U)$, which is given the
parity $\cP(U)_i = \Hom_{D}(C_i, \cF(U))$, $i=0,1$. 
This forms a super-algebra with respect to the convolution-product 
\cite[p.6]{Mo}, and gives rise to a presheaf
\begin{equation}\label{eq:cP}
\cP=\Hom_D(C,\cF)
\end{equation}
of super-algebras on $|F|$. 
By \eqref{eq:H} this $\cP$ is identified with $\cF \otimes \wedge(\sV^*)$ 
(see \eqref{eq:ssmfd}), 
so that $\cP$ is indeed a sheaf, and $(|F|, \cP)$ is an s-split super-manifold. 
Given $g \in |F|$, $\cF_g$ and $\hF_g$ are naturally (topological) left $D$-module algebra
(see Lemma \ref{lem:hy(G_red)-action}), and we have
\begin{equation}\label{eq:P_g}
\cP_g = \Hom_D(C, \cF_g), \quad \hP_g = \Hom_D(C, \hF_g).
\end{equation}

Let $n>0$ be an integer, and let
$\cF^{(n)}=\cF_{F^n}$ denote the structure sheaf of the product $F^n$ of $n$ copies of the Lie group 
$F=(|F|,\cF)$.
It is a sheaf of left $D^{\otimes n}$-module algebra, since $\hy(F^n)=D^{\otimes n}$. 
The construction of $\cP$ is generalized so that the super-algebra associated with every open set
$U \subset |F|^n$ 
\begin{equation*}
\cP^{(n)}(U) = \Hom_{D^{\otimes n}}(C^{\otimes n}, \cF^{(n)}(U))
\end{equation*}
which consists of all left $D^{\otimes n}$-module maps gives rise to a sheaf $\cP^{(n)}$;
by \eqref{eq:H}, this is identified with $\cF^{(n)}\otimes \wedge(\sV^*)^{\otimes n}
\, (= \cF^{(n)} \otimes \wedge((\sV^*)^{\oplus n}))$, and 
is, therefore, regarded as 
the structure sheaf of the product $(|F|,\cP)^n$ of $n$ copies of $(|F|,\cP)$; 
see Remark \ref{rem:product_smfd} (2). 

Let $\mu : F \times F \to F,\ \mu(g,h)=gh$ and
$\iota : F \to F, \ \iota(g)=g^{-1}$ denote the product and the inverse of 
the underlying group of $F$. The first step of making $(|F|,\cP)$ into a super Lie group
is to give to it the group structure the same as $\mu$, $\iota$. 
For any open set $V \subset |F|$ we let $\langle \ , \ \rangle : V \times \cF(V) \to \Bbbk$, $\langle g, p\rangle
=p(g)$ denote the evaluation map. The next step is to define sheaf-morphisms 
\[
\Delta : \cP \to \mu_*\cP^{(2)},\quad S: \cP \to \iota_*\cP
\]
which shall satisfy $\Delta = \mu^*$,\ $S=\iota^*$.
Given open sets $U \subset |F| \times |F|$,\ $V \subset |F|$ with 
$\mu(U) \subset V$, and an element $p \in \cP(V)$, we define  
the element $\Delta(p)$ of $\cP^{(2)}(U)= \op{Hom}_{D\otimes D}(C\otimes C, \cF^{(2)}(U))$ 
to be the map which sends $a\otimes b \in C\otimes C$
to the analytic function on $U$ with the value
\begin{equation}\label{eq:sheaf_mor_Delta}
\Delta(p)(a\otimes b)(g,h):=\langle g h, \ p(a^h\, b)\rangle \ \text{at} \ (g,h) \in U.
\end{equation}
Given open sets $U, V \subset |F|$ with $\iota(U) \subset V$, and an element $p \in \cP(V)$, 
we define the element
$S(p)$ of $\cP(U)=\op{Hom}_D(C,\cF(U))$ to be the map which
sends $a \in C$ to the analytic function on $U$ with the value 
\begin{equation}\label{eq:sheaf_mor_S}
S(p)(a)(g) := \langle g^{-1}, \ p(S_C(a)^{g^{-1}})\rangle \ \text{at} \ g \in U. 
\end{equation}

\begin{prop}\label{prop:super_Lie_from_HCP}
These indeed define sheaf-morphisms $\Delta$, $S$ of super-algebras, with which
 $(|F|,\cP)$, given the group structure $\mu$, $\iota$,
is a super Lie group such that $\Delta = \mu^*$,\ $S=\iota^*$. 
\end{prop}

In the algebraic situation treated by \cite{M3,MS1}, 
every affine super-group scheme (see Section \ref{subsec:Hopf_group_scheme})
is constructed uniquely from an (algebraic) Harish-Chandra pair.
The construction is conceptual, involving quite similar expressions to
\eqref{eq:sheaf_mor_Delta} and \eqref{eq:sheaf_mor_S}.
To prove Proposition \ref{prop:super_Lie_from_HCP} we choose to deduce it from the algebraic counterpart
just cited, because the resulting construction
is conceptual, again, and because a Hopf super-algebra, $\cB$, 
which we will construct in the course, will play a role in Section \ref{subsec:super-repres_function}.

We apply the algebraic construction of \cite[Section 4.4]{MS1} to the Hopf algebra $\cR(F)$
and the hyper-super-algebra $C$ as above. The Hopf pairing \eqref{eq:Hopf_pairing} makes
$\cR(F)$ into a left module algebra over $D=\hy(F)$, so that 
$a \triangleright p = p_{(1)} \langle a, p_{(2)} \rangle$,
where $a \in D,\ p \in \cR(F)$. Let 
\[ 
\cB = \mr{Hom}_D(C, \cR(F)) 
\]
denote the super-algebra (with respect to the convolution-product and the parity induced from $C$,
as before) consisting of all left $D$-module
maps $D \to \cR(F)$. Note that
\[ 
\cB^{\otimes n}= \mr{Hom}_{D^{\otimes n}}(C^{\otimes n}, \cR(F)^{\otimes n}),\quad n \ge 1. 
\]

\begin{prop}\label{prop:calB}
$\cB$ forms a Hopf super-algebra with respect to the counit
\begin{equation}\label{eq:epsilon(p)}
\epsilon: \cB \to \Bbbk,\quad \epsilon(p)=\epsilon_{\cR(F)}(p(1))
\end{equation}
together with the coproduct and the antipode
\[
\Delta : \cB \to \cB \otimes \cB \quad \text{and} \quad 
S : \cB \to \cB 
\]
given by formulas analogous to \eqref{eq:sheaf_mor_Delta} and \eqref{eq:sheaf_mor_S}.
\end{prop}

To describe explicitly the formulas last mentioned, let 
\begin{equation*}\label{eq:Fal}
F^{\mr{al}}=\op{Sp}(\cR(F))
\end{equation*}
denote the affine group scheme represented by $\cR(F)$;  as in \cite[p.1141]{HM} it may be called
\emph{the universal algebraic hull of} $F$.
By definition $F^{\mr{al}}$ is
the functor which associates to each algebra $R$, the group $F^{\mr{al}}(R)=\mr{Alg}(\cR(F), R)$ 
of all algebra maps $\cR(F) \to R$. This group scheme may not be algebraic, that is, 
$\cR(F)$ may not be finitely generated.
A right (resp., left) $F^{\mr{al}}$-(super-) module 
is identified with a left (resp., right) $\cR(F)$-(super-)comodule, and so by \eqref{eq:tensor_isom}, with
a right (resp., left) $F$-(super-)module.  
For example, the $\sV$ of the Harish-Chandra pair is a right $F^{\mr{al}}$-module,
on which each point $g \in F^{\mr{al}}(R)$ acts as the $R$-linear automorphism of $\sV \otimes R$ determined
by
\[
v \otimes 1 \mapsto v^g:= v^{(0)}\otimes g(v^{(-1)}),
\]
where we have used the notation \eqref{eq:-1_0_sigma_notation}. 
Let $p \in \cB$. The value $\Delta(p)(a \otimes b)$ of the element $\Delta(p)$ in 
$\cB \otimes \cB=\mr{Hom}_{D\otimes D}(C\otimes C, \cR(F)\otimes \cR(F))$ at $a \otimes b \in C\otimes C$ 
is an element of the Hopf algebra $\cR(F) \otimes \cR(F)$ representing $F^{\mr{al}}\times F^{\mr{al}}$;
it takes at every $R$-point $(g,h) \in F^{\mr{al}}(R) \times F^{\mr{al}}(R)$ the value
\begin{equation}\label{eq:Delta(p)}
\langle (g,h), \ \Delta(p)(a\otimes b)\rangle:=\langle g h, \ p(a^h\, b)\rangle,
\end{equation}
which is, more precisely, the value of $(p\otimes \mr{id}_R)(a^h (b\otimes 1))$ at $gh \in F^{\mr{al}}_R(R)$. 
For $a \in C$ the element $S(p)(a)$ in $\cR(F)$ tales at every $R$-point $g \in F^{\mr{al}}(R)$ the value
\begin{equation}\label{eq:S(p)}
\langle g,\ S(p)(a) \rangle := \langle g^{-1}, \ p(S_C(a)^{g^{-1}})\rangle,
\end{equation}
which is, more precisely, the value of $(p\otimes \mr{id}_R)(S_C(a)^{g^{-1}})$ at
$g^{-1} \in F^{\mr{al}}_R(R)$. 

The formulas \eqref{eq:Delta(p)} and \eqref{eq:S(p)} are seen to be quite similar to 
\eqref{eq:sheaf_mor_Delta} and \eqref{eq:sheaf_mor_S}, respectively.

\begin{proof}[Proof of Proposition \ref{prop:calB}]
The formulas \eqref{eq:epsilon(p)}--\eqref{eq:S(p)}
are the same as those given by 
\cite[Lemma 4.14]{MS1}
as the structure maps of a Hopf super-algebra $\mathbf{A}$; see also \cite[Lemma 26]{M3}. 
This construction of $\mathbf{A}$ can be done in a slightly more generalized situation, as is remarked
by \cite[Remark 4.21]{MS1}. Proposition \ref{prop:calB} follows
from the last cited, generalized construction, with careful treatment.
Indeed, we must be careful to apply those Lemma 4.14 and Remark 4.21 of \cite{MS1}, since
there, the Hopf super-algebra is represented alternatively as
\begin{equation}\label{eq:HomU(g)}
\mr{Hom}_{\U(\g_0)}(\U(\g), \cR(F)), 
\end{equation}
where $\g=P(C)$, if we modify the notation so as to fit with the present situation; 
on the other hand, in \cite{M3} 
it is represented essentially in the same way as our $\mathcal{B}$ is. 
(Working over a commutative ring in general, the article \cite{MS1} must work with universal envelopes of Lie super-algebras.
But, when working over a field of characteristic $\ne 2$ as we are doing, it is more natural to
work with hyper-super-algebras.) 
Recall from Lemma \ref{lem:H_SMod-F} that $C$ is a Hopf-algebra object in $\ms{SMod}\text{-}F$
including $D$ as a sub-object.  
One can restrict the structure to $\g$, and then extend to $\U(\g)$, so that 
$\g$ turns into a Lie-algebra object, and 
$\U(\g)$ turns into a Hopf-algebra object including $\U(\g_0)$ as a sub-object.
Note that the Hopf super-algebra map $\U(\g) \to C$ extending
$\g\subset C$ is a morphism in $\ms{SMod}\text{-}F$
which maps $\U(\g_0)$ into $D$. Moreover, the super-vector space given in 
\eqref{eq:HomU(g)} is identified 
with our $\cB=\mr{Hom}(\wedge(\sV),\cR(F))$, 
since $\U(\g) \simeq \U(\g_0) \otimes \wedge(\sV)$ as left $\U(\g_0)$-super-module super-coalgebras. 
Then we can apply the cited results to prove the proposition;
choose as the paring $\g_0 \times \cR(F) \to \Bbbk$ required  
in \cite[Remark 4.21]{MS1}, the restriction of the pairing $D \times \cR(F) \to \Bbbk$ in \eqref{eq:Hopf_pairing}. 
\end{proof}

\begin{rem}\label{rem:modify_rem}
We record here how we should modify
\cite[Remark 4.21]{MS1}, which was applied above, so as to completely fit in with our situation. 

Suppose that the base field $\Bbbk$ is an arbitrary field of characteristic $\ne  2$.  
Let $\ms{F}$ be an affine group scheme represented by a Hopf algebra $O_{\ms{F}}$, and let
$\mathfrak{C}$ be a hyper-super-algebra of finite type with $\mathfrak{D}=\underline{\mathfrak{C}}$. 
One sees that 
the same formulas as \eqref{eq:epsilon(p)}--\eqref{eq:S(p)}
construct a Hopf super-algebra on $\mr{Hom}_{\mathfrak{D}}(\mathfrak{C},O_{\ms{F}})$,
if the following are satisfied (which are in fact modified from 
the assumptions posed 
in \cite[Remark 4.21]{MS1} in order to construct a Hopf super-algebra on $\mr{Hom}_{\U(\g_0)}(\U(\g), O_{\ms{F}})$):
\begin{itemize}
\item[(1)] $\mathfrak{C}$ is a Hopf-algebra object in the 
category $\ms{SMod}\text{-}{\ms{F}}$ of right $\ms{F}$-super-modules; 
it then necessarily includes $\mathfrak{D}$ as a purely even sub-object;
\item[(2)] We are given a Hopf pairing $\langle \ , \ \rangle : \mathfrak{D} \times O_{\ms{F}} \to \Bbbk$ 
such that 
\begin{itemize}
\item[(2a)] $\langle a^g,\ p\rangle_R = \langle a,\ {}^gp\rangle_R$ for all $a \in \mathfrak{D}$,\ 
$p \in O_{\ms{F}}$ and $g \in {\ms{F}}(R)$, where $R$ is an arbitrary algebra, and
\item[(2b)] through the pairing, the right ${\ms{F}}$-action on $\mathfrak{D}$ or the associated left 
$O_{\ms{F}}$-coaction $\mathfrak{D} \to O_{\ms{F}}\otimes \mathfrak{D},\ b \mapsto b^{(-1)}\otimes b^{(0)}$
induces the right adjoint action
(see \eqref{eq:adr}) in the sense that
$\mr{ad}_r(a)(b) = \langle a, b^{(-1)}\rangle \, b^{(0)}$ \ for all \ $a, b \in \mathfrak{D}$.
\end{itemize}
\end{itemize}
In (2a), ${}^gp=p_{(2)}\otimes g(S(p_{(1)})p_{(3)})$ denotes the left ${\ms{F}}$-adjoint action on $O_{\ms{F}}$,
and $\langle \ , \ \rangle_R$ denotes the base extension of $\langle \ , \ \rangle$ to $R$. 

In our situation we see easily that the affine group scheme $F^{\mr{al}}$ 
and the hyper-super-algebra $C$ satisfy
the assumptions (1), (2) above, whence Proposition \ref{prop:calB} follows. 
The Hopf pairing \eqref{eq:Hopf_pairing} is chosen 
as the one required by (2). We remark that (2a) is now reduced to
\begin{equation}\label{eq:agp2}
\langle a^g,\ p\rangle = \langle a,\ {}^gp\rangle, \quad a \in D,\ 
p \in \cR(F), \ g \in F,
\end{equation}
which is satisfied by \eqref{eq:agp1}. 
\end{rem}

Let us recall from \cite[Section 4.4]{MS1} (see also \cite[Section 4.6]{M3}) 
how naturally the structure maps of $\cB$ arise. The process will be used  
when we prove Proposition \ref{prop:super_Lie_from_HCP} and Theorem \ref{thm:G-super-mod_struc}. 
\medskip

{\bf Step 1.~~Construction of $\cA$ as a smash coproduct.}\quad
Let $\sW=\sV^*$ denote the dual vector space of $\sV$; 
these are both regarded as purely odd super-vector spaces.
Let $\T(\sV)=\bigoplus_{n\ge 0}\T^n(\sV)$ denote the tensor algebra on $\sV$. 
This is an $\mathbb{N}$-graded Hopf super-algebra (see Section \ref{subsec:Hopf_group_scheme})
in which every element in $\sV$ 
is supposed to be an odd primitive; it is super-cocommutative, but is not super-commutative
unless $\sV = 0$. The tensor coalgebra $\T_{\mr{c}}(\sW)$ on $\sW$ is the graded dual \cite[Section 11.2]{Sw}
\[ \T_{\mr{c}}(\sW)=\bigoplus_{n\ge 0}\T^n(\sV)^* \, (=\bigoplus_{n\ge 0}\T^n(\sW) ) \]
of $\T(\sV)$; it is, therefore, a super-commutative Hopf super-algebra. 

The left $\cR(F)$-coaction
on $\sV$ uniquely extends onto $\T(\sV)$ so that $\T(\sV)$ turns into an $\mathbb{N}$-\emph{graded Hopf-algebra
object} in $\cR(F)\text{-}\mathsf{SComod}$ (or namely, an $\mathbb{N}$-graded Hopf super-algebra such that
each $\mathbb{N}$-homogeneous component is a left $\cR(F)$-comodule and 
the structure maps are all $\cR(F)$-colinear). 
By duality $\T_{\mr{c}}(\sW)$ turns into an 
$\mathbb{N}$-graded Hopf-algebra object in $\mathsf{SComod}\text{-}\cR(F)$. The associated smash coproduct
\cite[Page 3455, line--1]{MS1} is denoted by
\[ \cA=\cR(F) \cmdblackltimes \T_{\mr{c}}(\sW). \]
This is a super-commutative $\mathbb{N}$-graded Hopf super-algebra, 
being the 
tensor product $\cR(F) \otimes \T_{\mr{c}}(\sW)$ as a super-algebra. The $n$-th component
is $\cA(n)=\cR(F)\otimes \T^n(\sW)$. 
\medskip

{\bf Step 2.~~Completion of $\cA$.}\quad  
The projection $\cA \to \cA(0)=\cR(F)$ is a Hopf super-algebra
map, and so the kernel $\cA^+:= \bigoplus_{n>0}\cA(n)$ is a homogeneous Hopf super-ideal. 
Therefore, the $\cA^+$-adic completion $\hat{\cA}$ of $\cA$ turns into a complete Hopf super-algebra. 
The last word is used in the wider sense as noted in Remark \ref{rem:complete_Hopf} (3). 
But this $\hat{\cA}$ is just the direct product $\prod_{n\ge 0}\cR(F) \otimes \T^n(\sW)$, and 
$\hat{\cA}~\hotimes~\hat{\cA}= \prod_{m,n \ge0}\cR(F) \otimes \T^m(\sW)\otimes \cR(F) \otimes \T^n(\sW)$;
the structure maps
\begin{equation}\label{eq:completed_structure}
\hat{\Delta} : \hat{\cA} \to \hat{\cA}~\hat{\otimes}~\hat{\cA},\quad 
\hat{\epsilon}: \hat{\cA} \to \Bbbk, \quad \hat{S} : \hat{\cA}\to \hat{\cA}
\end{equation}
only extend those of $\cA$ to the relevant direct products in the obvious manner. 

\begin{rem}\label{rem:hatA}
The right $F$-action on $\sV$ uniquely extends onto $\T(\sV)$ so that $\T(\sV)$ turns into an $\mathbb{N}$-graded
Hopf-algebra object in $\ms{SMod}\text{-}F$. The associated smash product $F \ltimes \T(\sV)$  
is a (not necessarily super-commutative) Hopf super-algebra, in which
every element of $F$ is supposed to be group-like. 
The dual $(F \ltimes \T(\sV))^*$ is a
complete Hopf super-algebra in the wider sense; see Remark \ref{rem:complete_Hopf} (3), again. See also 
Proposition \ref{prop:topological_comodule} for the complete topological dual (super-)coalgebra structure. 
The natural pairings $(F \times \T^n(\sV))\times (\cR(F) \otimes \T^n(\sW)) \to \Bbbk$, $n \ge 0$, 
amount to $(F \ltimes \T(\sV)) \times \hat{\cA} \to \Bbbk$, which induces a continuous injection
$\hat{\cA} \to (F \ltimes \T(\sV))^*$ of Hopf super-algebras, as is seen 
from the construction above.
\end{rem}

{\bf Step 3.~~Transfer of the completed structure.}\quad
Note from Lemma \ref{lem:Hopf_pairing} that $\T(\sV)$ is an $\mathbb{N}$-graded
Hopf-algebra object in $\ms{SMod}\text{-}D$, as well. 
The associated smash product
\[
\cC:=D \ltimes \T(\sV). 
\]
is a (not necessarily super-commutative) $\mathbb{N}$-graded Hopf super-algebra, 
being the tensor product $D \otimes \T(\sV)$ as a super-coalgebra. 
The $n$-th component is $\cC(n)=D\otimes \T^n(\sV)$.
Note that $\cC$ is a Hopf-algebra object in $\ms{SMod}\text{-}F$, just as
$C$ is by Lemma \ref{lem:H_SMod-F}. 
We have the natural identification
\begin{equation}\label{eq:A=Hom}
\hat{\cA}=\mr{Hom}_D(\cC, \cR(F))\, (=\mr{Hom}(\T(\sV),\cR(F))
\end{equation}
and, more generally, 
\begin{equation*}
\hat{\cA}^{\hat{\otimes}n}=\mr{Hom}_{D^{\otimes n}}(\cC^{\otimes n}, \cR(F)^{\otimes n}),\quad n \ge 1.
\end{equation*}
Transfer the structure maps \eqref{eq:completed_structure} of $\hat{\cA}$ to $\mr{Hom}_D(\cC, \cR(F))$.
As was proved by \cite[Proposition 4.11]{MS1},  
the resulting maps are given by the formulas \eqref{eq:epsilon(p)}--\eqref{eq:S(p)}
modified so that  
$a, b \in \cC$, \ $p\in \mr{Hom}_D(\cC, \cR(F))$,\
and $S_C$ is replaced with $S_{\cC}$. 
\medskip

{\bf Step 4.~~Restriction of the transferred structure.}\quad 
Note that $C$ is a quotient of $\cC$, as a Hopf-algebra object in $\ms{SMod}\text{-}F$ including $D$ as a 
sub-object. We then see that
$\hat{\cA}^{\hat{\otimes}n}$ naturally includes 
$\cB^{\otimes n}=\mr{Hom}_{D^{\otimes n}}(C^{\otimes n},$ $\cR(F)^{\otimes n})$. 
Moreover, the transferred structure maps of $\mr{Hom}_D(\cC, \cR(F))$ restrict to those of $\cB$ given by
\eqref{eq:epsilon(p)}--\eqref{eq:S(p)}, which, therefore, satisfy the Hopf-algebra axioms. Thus
$\cB$ turns into a Hopf super-algebra; see \cite[Lemma 4.14]{MS1}.

\begin{proof}[Proof of Proposition \ref{prop:super_Lie_from_HCP}]
Given open sets $U \subset |F| \times |F|$, $V \subset |F|$ with $\mu(U)\subset V$, 
we have the push-outs 
\[
\begin{xy}
(0,0)   *++{\cR(F)}  ="1",
(30,0)  *++{\cA}    ="2",
(0,-18) *++{\cF(V)} ="3",
(30,-18)*++{\cF(V)\otimes_{\cR(F)} \cA,} ="4",
(55,0)   *++{\cR(F)^{\otimes 2}}  ="5",
(90,0)  *++{\cA^{\otimes 2}}    ="6",
(55,-18) *++{\cF^{(2)}(U)}          ="7",
(90,-18)*++{\cF^{(2)}(U)\otimes_{\cR(F)^{\otimes 2}} \cA^{\otimes 2}}    ="8",
{"1" \SelectTips{cm}{} \ar @{->} "2"},
{"1" \SelectTips{cm}{} \ar @{->} "3"},
{"2" \SelectTips{cm}{} \ar @{->} "4"},
{"3" \SelectTips{cm}{} \ar @{->} "4"},
{"5" \SelectTips{cm}{} \ar @{->} "6"},
{"5" \SelectTips{cm}{} \ar @{->} "7"},
{"6" \SelectTips{cm}{} \ar @{->} "8"},
{"7" \SelectTips{cm}{} \ar @{->} "8"}
\end{xy}
\]
of super-algebras, where the vertical arrows on the left-hand side are the inclusions 
$\cR(F)^{\otimes n}\hookrightarrow\cF(|F|)^{\otimes n}\subset \cF^{(n)}(|F|^n)$, $n=1,2$, 
followed by the restriction maps. 
The two tensor products are identified with $\cF(V) \otimes \T_{\mr{c}}(\sW)$
and $\cF^{(2)}(U) \otimes \T_{\mr{c}}(\sW)\otimes \T_{\mr{c}}(\sW)$, respectively. 
Since the coproducts $\Delta_{\cR(F)}$ of $\cR(F)$ and $\Delta_{\cA}$ of $\cA$ and the
sheaf-morphism $\mu^*:\cF(V) \to \cF^{(2)}(U)$ associated with the product $\mu$ on $F$ 
are compatible with the relevant arrows, we have uniquely a map
\begin{equation*}
\Delta' : \cF(V) \otimes \T_{\mr{c}}(\sW)\to
\cF^{(2)}(U) \otimes \T_{\mr{c}}(\sW)\otimes \T_{\mr{c}}(\sW)
\end{equation*}
which, together with the preceding three, constitutes a map between the two push-outs. 
By completion (to direct products) followed by an identification just as for $\hat{\cA}$, it derives
\[
\hat{\Delta}' : \mr{Hom}_D(\cC, \cF(V)) \to \mr{Hom}_{D^{\otimes 2}}(\cC^{\otimes2},\cF^{(2)}(U)).
\]
Similarly, given open sets $U, V \subset |F|$ with $\iota(U) \subset V$, we have the push-out 
\[
S' : \cF(V)\otimes \T_{\mr{c}}(\sW)  \to \cF(U)\otimes \T_{\mr{c}}(\sW)
\]
of $\iota^*$ and the antipode $S_{\cA}$ of $\cA$ over the antipode $S_{\cR(F)}$ of $\cR(F)$,
which derives
\[
\hat{S}' : \mr{Hom}_D(\cC, \cF(V)) \to \mr{Hom}_D(\cC, \cF(U)).
\]

We have the sheaf of super-algebras on $|F|$
\begin{equation}\label{eq:Hom(H,F)}
\mr{Hom}_D(\cC, \cF) 
\end{equation}
which associates $\mr{Hom}_D(\cC, \cF(U))$ to every open set $U \subset |F|$.
For $n > 0$ the sheaves $\mr{Hom}_{D^{\otimes n}}(\cC^{\otimes n}, \cF^{(n)})$ on $|F|^n$
are analogously defined.
One sees that the $\hat{\Delta}'$ and $\hat{S}'$ above
give rise to sheaf-morphisms $\hat{\Delta}': \mr{Hom}_D(\cC, \cF)
\to \mu_*\mr{Hom}_{D^{\otimes 2}}(\cC^{\otimes 2}, \cF^{(2)})$ and 
$\hat{S}' : \mr{Hom}_D(\cC, \cF)\to \iota_*\mr{Hom}_D(\cC, \cF)$. 
These are presented by formulas analogous to \eqref{eq:sheaf_mor_Delta} and \eqref{eq:sheaf_mor_S},
just as the structure maps of $\hat{\cA}$ are
given by formulas analogous to \eqref{eq:epsilon(p)}--\eqref{eq:S(p)}.

Recall from Step 4 above that 
$C$ is a quotient of $\cC$ as a Hopf-algebra object 
in $\ms{SMod}\text{-}F$ including $D$.
One then sees that $\cP$ is a sub-sheaf of the sheaf $\mr{Hom}_D(\cC, \cF)$, 
and the sheaf-morphisms $\hat{\Delta}'$ and $\hat{S}'$ for the last sheaf
restrict to the sheaf-morphisms $\Delta$ and $S$ for $\cP$ which are
presented by \eqref{eq:sheaf_mor_Delta} and \eqref{eq:sheaf_mor_S}.

In view of Proposition \ref{prop:identify}
it remains to prove, passing to the completed stalks, that
\[ \hat{\Delta}_{g,h}:\hat{\cP}_{gh} \to \hat{\cP}_g~\hat{\otimes}~\hat{\cP}_h,\quad
\hat{S}_{g}:\hat{\cP}_{g^{-1}} \to \hat{\cP}_g,\quad g,h \in |F| \]
satisfy the axioms of coalgebra and of antipode. For example, the coassociativity means
$ \hat{\Delta}_{g,h}~\hat{\otimes}~\mr{id}_{\hat{\cP}_{\ell}}
=\mr{id}_{\hat{\cP}_g}~\hat{\otimes}~\hat{\Delta}_{h,\ell}$, and the left counit-property means
that the composite
\[
\hat{\cP}_{g} \xrightarrow{\hat{\Delta}_{e,g}}~\hat{\cP}_e~\hat{\otimes}~\hat{\cP}_g
\xrightarrow{\hat{\epsilon}_e\hat{\otimes}{\mr{id}}} \Bbbk~\hat{\otimes}~\hat{\cP}_g
=\hat{\cP}_g
\]
coincides with the identity, where $\hat{\epsilon}_e:\hat{\cP}_e \to \Bbbk$ denotes the projection to
the residue field; see \eqref{eq:proj_to_residue}. 
But one sees from the argument of the last paragraph that it suffices to prove the
analogous assertions for the sheaf $\mr{Hom}_D(\cC, \cF)$, or even for the sheaf 
$\cF \otimes \T_{\mr{c}}(\sW)$ which associates
$\cF(U) \otimes \T_{\mr{c}}(\sW)$ to every open set $U\subset |F|$. 
To be explicit the last mentioned,
analogous assertion is the following: the unique maps 
$\hat{\cF}_{gh}\otimes \T_{\mr{c}}(\sW) \to 
\hat{\cF}_{g}~\hat{\otimes}~\hat{\cF}_h \otimes \T_{\mr{c}}(\sW) \otimes \T_{\mr{c}}(\sW)$
and
$\hF_{g^{-1}}\otimes \T_{\mr{c}}(\sW) \to \hF_g\otimes \T_{\mr{c}}(\sW)$
that together with 
\[
\widehat{\mu^*}_{g,h} : \hat{\cF}_{gh}\to \hat{\cF}_{g}~\hat{\otimes}~\hat{\cF}_h,\quad
\Delta_{\cA} : \cA \to \cA \otimes \cA,\quad
\Delta_{\cR(F)}:\cR(F) \to \cR(F) \otimes \cR(F),
\]
and respectively with
\[
\widehat{\iota^*}_g : \hat{\cF}_{g^{-1}}\to \hat{\cF}_{g},\quad
S_{\cA} : \cA \to \cA,\quad
S_{\cR(F)}:\cR(F) \to \cR(F)
\]
constitute maps of push-outs
satisfy the axioms of coalgebra and of antipode. 
Indeed, the
assertion is true since the displayed six maps satisfy the axioms. 
Note that
as the counit on $\hat{\cF}_e \otimes \T_{\mr{c}}(\sW)$
one should choose the tensor product
$\hat{\epsilon}_e\otimes \epsilon_0 : \hat{\cF}_e \otimes \T_{\mr{c}}(\sW) \to \Bbbk\otimes \Bbbk
=\Bbbk$, where $\epsilon_0 :\T_{\mr{c}}(\sW) \to \T^0(\sW)=\Bbbk$ denotes the projection; it constitutes a map of push-outs,
together with $\hat{\epsilon}_e:\hat{\cF}_e\to \Bbbk$, $\epsilon_{\cA}:\cA \to \Bbbk$ and 
$\epsilon_{\cR(F)}:\cR(F) \to \Bbbk$. 
\end{proof}

\begin{lemma}\label{lem:Psi}
Let $\Psi(F,\sV)$ denote the super Lie group $(|F|,\cP)$ which is thus constructed from 
a Harish-Chandra pair $(F, \sV)$. The assignment gives rise to a functor
\[ \Psi:\ms{HCP}\to \ms{SLG}, \ \Psi(F,\sV)= (|F|,\cP). \]
\end{lemma}
\begin{proof}
With the notation above one sees easily that the constructions of
\begin{itemize}
\item[(i)] the inclusion $C \supset D$ of Hopf-algebra objects in $\ms{SMod}\text{-}F$, and
\item[(ii)] the sheaf $\cF$ of left $D$-module algebras
\end{itemize}
from $(F,\sV)$ are both functorial. This implies that $\Psi$ is functorial. 
\end{proof}

\begin{proof}[Proof of Theorem \ref{thm:equivalence} (First half)]
We shall prove that $\Psi$ is a quasi-inverse of $\Phi$.
But our aim here is only to prove $\Phi \circ \Psi \simeq \mr{id}$.
The following subsection is devoted to proving $\Psi \circ \Phi \simeq \mr{id}$.

Given $(F, \sV) \in \ms{HCP}$, let $C$ and $G:=(|F|,\cP)$ be the hyper-super-algebra and
the super Lie group, respectively, constructed as above. 
For every open set $U \subset |F|$ we have $\cP_{\mr{red}}(U)
=\mr{Hom}_D(D,\cF(U))= \cF(U)$ so that $G_{\mr{red}}=F\, (=(|F|, \cF))$. 
It remains to show that the right $F$-module $\sV$ and the bracket $[\ , \ ]$ of the Harish-Chandra pair
recover from $G$. 
The former $\sV$ appears as the odd component $P(C)_1$ of $P(C)$; 
recall from Lemma \ref{lem:H_SMod-F} that $F$ acts on $C$ from the right, whence 
$P(C)$ and $P(C)_1$ are right $F$-modules by restriction.  
In view of \eqref{eq:relations_H} it suffices to prove that
there is a natural $F$-equivariant isomorphism 
$C \overset{\simeq}{\longrightarrow} \hy(G)$ which is identical on $D=\hy(F)$, 
or equivalently, there is such an isomorphism  
$\hP_e \overset{\simeq}{\longrightarrow} C^*$ which is compatible with the natural maps
to $\hF_e$. Recall from \eqref{eq:P_g} and Proposition \ref{prop:Gred-super-module}
that $\hP_e=\mr{Hom}_D(C,\hF_e)$, and it is equipped with
such a left $F$-action that induces the right $F$-adjoint action on $\hy(G)$; see also
Remark \ref{rem:adjoint_action}. 

Let $\hat{\Delta}^{(2)}_{g^{-1},h,g}:
\hP_{g^{-1}hg}\to \hP_{g^{-1}} \, \hotimes \, \hP_h \, \hotimes \, \hP_g$ represent 
$(\op{id}\, \hotimes \, \hat{\Delta}_{h,g})\circ \hat{\Delta}_{g^{-1},hg} \,
(=(\hat{\Delta}_{g^{-1},h}\, \hotimes\op{id})
\circ \hat{\Delta}_{g^{-1}h,g})$. 
Then we see from \eqref{eq:sheaf_mor_Delta} that 
\[
(\hat{\epsilon}_{g^{-1}}~\hat{\otimes}~\hat{\epsilon}_h~\hat{\otimes}~\hat{\epsilon}_g)
(\hat{\Delta}^{(2)}_{g^{-1},h,g}(p)(a\otimes b\otimes c))
=\langle g^{-1}hg,\ p(a^{hg}\, b^g\, c)\rangle,
\]
where $g, h \in F$, $p\in \hP_{g^{-1}hg}$ and $a,b,c \in C$.
(We remark that to be precise, the $p$ on the right-hand side should read an element
of some/any $\cP(U)$ with $g^{-1}hg \in U$, whose germ is the $p \in \hP_{g^{-1}hg}$.) 
It follows that 
the $F$-action on $\hP_e$ is given by 
\[
({}^g p)(b) ={}^g(p(b^g)),\quad g \in F,\ p \in \hP_e,\ b \in C. 
\]
Define
$\xi : \hP_e \to C^*$ by
\[ \xi(p)(a) = \hat{\epsilon}_e(p(a)),\quad p \in \hP_e,\ a \in C.  \]
This $\xi$ is seen to be $F$-equivariant. Since one sees that $\hF_e=D^*$ as a left $D$-module algebra,
it follows that $\xi$ is a natural super-algebra isomorphism
with the desired compatibility; indeed, it is identified with the
canonical isomorphism $\op{Hom}_D(C, D^*) \simeq C^*$. 
Moreover, it preserves the coproduct since  
$(\xi~\hat{\otimes}~\xi)\circ \hat{\Delta}_{e,e}(p)(a\otimes b)$ coincides with the value 
$\Delta(p)(a\otimes b)(e,e)=\langle e, p(ab)\rangle$
given by \eqref{eq:sheaf_mor_Delta}, where $p \in \hP_e$ and $a, b \in C$. (To be
precise the same remark as above should apply to the $p$ in the last equation.)
\end{proof}

\subsection{Completion of the proof}\label{subsec:complete_proof}
Let $G=(|G|,\cO) \in \ms{SLG}$, and set 
\[ (F,\sV)=\Phi(G) \in \ms{HCP},  \]
or explicitly, $(F,\sV)=(G_{\mathrm{red}}, \Lie(G)_1)$. Let $\cF=\cO_{\mathrm{red}}$ 
be the structure sheaf of the Lie group $F=G_{\mr{red}}$, and set $D=\hy(F)\, (=\underline{\hy(G)})$. 

Let $C$ denote the hyper-super-algebra constructed from $(F,\sV)$ as in the preceding subsection;
see  \eqref{eq:relations_H}, \eqref{eq:H}. 

\begin{lemma}\label{lem:coincidence_of_H}
Recall from Lemma \ref{lem:H_SMod-F} and
Proposition \ref{prop:Gred-super-module} (2) that the hyper-super-algebras $C$ and $\hy(G)$ are both 
Hopf-algebra objects in $\ms{SMod}\text{-}F$.
These Hopf-algebra objects are naturally identified. 
\end{lemma}

\begin{proof}
By \eqref{eq:relations_H} the inclusions 
$D=\underline{\hy(G)} \hookrightarrow \hy(G)\hookleftarrow \Lie(G)_1=\sV$ 
amount to an $F$-equivariant map $C \to \hy(G)$ of hyper-super-algebras, which is an isomorphism by 
\cite[Theorem 3.6]{M2}; see also \eqref{eq:isom_module_coalgebra}. 
\end{proof}

By Lemma \ref{lem:coincidence_of_H} we have $C=(\hO_e)^{\star}$. The resulting canonical pairing
will be used in the proof below, denoted by $\langle \ , \ \rangle : C \times \hO_e \to \Bbbk$; it satisfies
\begin{equation}\label{eq:F-balanced}
\langle a^g,\ p\rangle = \langle a,\ {}^gp\rangle,\quad a\in C,\ p \in \hO_e,\ g \in F.
\end{equation}

Concerning the super Lie group $\Psi(F,\sV)$ the associated Lie group is $F=G_{\red}$. 
Let $\cP$ be the structure sheaf given in \eqref{eq:cP}, and so let
\[ (|F|,\cP)= \Psi(F,\sV) \in \ms{SLG}. \]

In what follows we let $\red : \cO \to \cO_{\red}=\cF$
denote the natural sheaf-epimorphism. Consistent with this is the symbol 
$\widehat{\red}_g:\hO_g \to \hF_g$
which was used  (see \eqref{eq:red_g}) and will be soon used to denote the natural projection.

\begin{prop}\label{prop:eta}
We have the following.
\begin{itemize}
\item[(1)]
For every open set $U \subset |G|$, define
\[ 
\eta_U(p)(a)=\red_U(a \triangleright p),\quad p\in \cO(U),\ a \in C\, (=\hy(G)), 
\]
where $a\triangleright p$ indicates the $\hy(G)$-action on $\cO$ given by Proposition \ref{prop:hy(G_red)-module}.
Then this indeed defines a morphism $\eta :\cO\to \cP$ of sheaves of superalgebras. 
\item[(2)]
$(\mr{id},\eta) : (|F|,\cP) \to (|G|,\cO)=G$ 
is an isomorphism of super Lie groups.
\end{itemize}
\end{prop}

\begin{proof}
(1)\ 
This is easy to see.

For the proof of Part 2 below we remark that 
for every $g \in |F|\, (=|G|)$, \ $\hat{\eta}_g : \hO_g \to \hP_g =\Hom_{D}(C, \hF_g)$ is given by
\begin{equation*}\label{eq:etag}
\hat{\eta}_g(p)(a) = (\widehat{\red}_g \hotimes~a)\circ \hat{\Delta}_{g,e}(p),\quad 
p \in \hO_g,\ a \in C. 
\end{equation*}
This formula is expressed by using the sigma notation 
$\hat{\Delta}_{g,e}(p)=p_{(1)}\otimes p_{(2)}$, so as
\begin{equation}\label{eq:hatetag}
\hat{\eta}_g(p)(a) = \widehat{\red}_g(p_{(1)})\, \langle a, p_{(2)}\rangle.
\end{equation}
Note that the sum on the right-hand side above is finite since $p_{(1)}\otimes p_{(2)}$
may be regarded as an element of some ordinary tensor product $\hO_g \otimes (\hO_e/\hat{\m}_e^r)$,
where $r>0$ is an integer such that $\langle a, \hat{\m}_e^r\rangle=0$.

(2)\ 
Let us prove that $(\mr{id}, \eta)$ is a morphism in $\ms{SLG}$,
using the notation of Section \ref{subsec:super_Lie_group}.  
Let $\eta^{(2)}: \cO^{(2)}=\cO_{G\times G} \to \cP^{(2)}$ denote the sheaf-morphism associated
with $(\mr{id},\eta)\times(\mr{id},\eta)$.
In view of Remark \ref{rem:super_Lie_morphism} (2) we have to prove
that for an open set $\emptyset \ne U \subset |G|$,
\[
\eta^{(2)}(\Delta(p))(a \otimes b)(g, h)= \eta(p)(a^h\, b)(g h), 
\]
where $p \in \cO(U),\ a, b  \in C$, and $(g, h) \in \mu^{-1}(U)$ (or $gh \in U$). 
If we let $q \in \hO_{gh}$ denote the germ given by $p$, it suffices in view of Remark 
\ref{rem:product_smfd} (2) to prove
\[
(\hat{\epsilon}_g~\hotimes~\hat{\epsilon}_h)
\big((\hat{\eta}_g~\hotimes~\hat{\eta}_h) \circ \hat{\Delta}_{g,h}(q) \, (a \otimes b)\big)
=
\hat{\epsilon}_{g h}(\hat{\eta}_{g h}(q)(a^h\, b)). 
\]
To verify this we use the sigma notation, such as used in \eqref{eq:hatetag} above,
\[ 
\hat{\Delta}_{g_1,\dots,g_s}(q)=q_{(1)}\otimes \dots \otimes q_{(s)},\quad g_1\dots g_s=gh.
\]
This is justified by the coassociativity of $\hat{\Delta}$; see the proof of Proposition \ref{prop:super_Lie_from_HCP}. 
Again in this situation the value may be supposed to be in some ordinary tensor product so that
the relevant sum is finite.  The verification is done, using \eqref{eq:F-balanced}, so as
\begin{align*}
\text{LHS} 
&= \hat{\epsilon}_g(q_{(1)})\, \langle a,q_{(2)}\rangle\,
\hat{\epsilon}_h(q_{(3)})\, \langle b,q_{(4)}\rangle\\
&= \hat{\epsilon}_g(q_{(1)})\, \hat{\epsilon}_h(q_{(2)})\, 
\hat{\epsilon}_{h^{-1}}(q_{(3)})\, \langle a,q_{(4)}\rangle\, 
\hat{\epsilon}_h(q_{(5)})\, \langle b,q_{(6)}\rangle\\
&= \hat{\epsilon}_g(q_{(1)})\, \hat{\epsilon}_h(q_{(2)})\, 
\langle a, {}^h q_{(3)}\rangle\, \langle b,q_{(4)}\rangle\\
&= \hat{\epsilon}_{gh}(q_{(1)})\, \langle a^h,q_{(2)}\rangle\, \langle b,q_{(3)}\rangle
=\text{RHS}.
\end{align*}

To complete the proof we will see by applying Lemma \ref{lem:isomorphism} 
that $\phi:=(\mr{id}, \eta)$ is an isomorphism. 
Obviously, $\phi_{\mr{red}} : F \to F=G_{\red}$ is the identity map. It remains to verify that
$\Lie(\phi)_1$ is an isomorphism.  Let $\sW =\sV^*$, as before. 
One sees from 
\[ \epsilon_e(v\triangleright q)=\langle v,q \rangle,\quad v\in\sV,\ q \in \cO_e \] 
that the map
$\cO_e \to \sW,\ q \mapsto \epsilon_e(\eta_e(q)|_{\sV})$ coincides with
the natural projection
$\cO_e\to (\m_e/\m_e^2)_1=\sW$. This implies that $\Lie(\phi)_1$ is the identity map.  
\end{proof}

Obviously, $\eta$ is natural in $G$. This completes the proof of Theorem \ref{thm:equivalence}.

\begin{rem}\label{rem:sign_be_removed}
The sheaf-isomorphism $\eta$ above looks like 
the $\eta^*$ given by \cite[Page 133, lines 2--3]{CCF} and \cite[Page 41, lines 17--18]{CF}
in the $C^{\infty}$- and the complex analytic situations. 
An essential difference is in that the $\eta^*$ contains the sign
$(-1)^{|X|}$ in its expression. But the sign is dispensable, since obviously, 
one remains to have a desired isomorphism after removing it. 
\end{rem}

\subsection{Consequences of the category equivalence}\label{subsec:consequences}
Here we have two corollaries to Theorem \ref{thm:equivalence}.  The first one is the following.

\begin{corollary}\label{cor:super_Lie_group_split}
Every super Lie group is s-split as a super-manifold; see Definition \ref{def:super-mfd} (2). 
\end{corollary}
\begin{proof}
By the theorem every super Lie group is of the form $(|F|,\cP)$,
whence it is s-split since $\cP=\cF\otimes \wedge(\sW)$; see \eqref{eq:cP}.
\end{proof}

To state the second one, we need some preliminaries.
Assume for a while that the base field $\Bbbk$ is an arbitrary infinite field
(e.g. a complete field), and let $\overline{\Bbbk}$
denote the algebraic closure of $\Bbbk$. Let $\ms{F}$ be
an affine algebraic group scheme $\ms{F}$ (over $\Bbbk$) represented by $O_{\ms{F}}$. 
We say that $\ms{F}$ is an 
\emph{algebraic matrix group} \cite[Sections 4.4--4.5]{W}, if the following equivalent conditions are satisfied:
\begin{itemize}
\item[(i)] The intersection of the maximal ideals of $O_{\ms{F}}$ of codimension $1$ is zero;
\item[(ii)] $O_\ms{F}$ is reduced, and 
the group $\ms{F}(\Bbbk)$ of $\Bbbk$-points
is Zariski dense in $\ms{F}(\overline{\Bbbk})$ or equivalently, Zariski dense in $\mr{Spec}(O_{\ms{F}})$;
\item[(iii)] The algebra map $O_{\ms{F}}\to \mr{Map}(\ms{F}(\Bbbk), \Bbbk)$
which associates to $p \in O_{\ms{F}}$, the map $\ms{F}(\Bbbk)\to \Bbbk,\ g \mapsto g(p)$
is an injection. 
\end{itemize}

The image of the algebra map in (iii) is denoted by $\mr{Poly}(\ms{F}(\Bbbk), \Bbbk)$, whose
elements are called \emph{polynomial functions}. Roughly speaking, 
an algebraic matrix group is an affine algebraic group scheme which arises 
from a Zariski closed matrix group.
If $\ms{F}$ is an algebraic matrix group, then we have 
$O_{\ms{F}}\otimes \overline{\Bbbk}\simeq \mr{Poly}(\ms{F}(\overline{\Bbbk}), \overline{\Bbbk})$, 
whence $O_{\ms{F}}\otimes \overline{\Bbbk}$ is reduced, or equivalently,
$O_{\ms{F}}$ (or $\ms{F})$ is smooth.  
In addition,  any element of $O_{\ms{F}}\otimes \overline{\Bbbk}$ that takes values in $\Bbbk$ on  
$\ms{F}(\Bbbk)$ is contained in $O_{\ms{F}}$; 
this, applied to idempotents, shows that the 
largest finite-dimensional separable subalgebra
$\pi_0(O_{\ms{F}})$ of $O_{\ms{F}}$ (see \cite[Section 6.5]{W}) is spanned by idempotents. 
If $\Bbbk$ is algebraically closed,
an algebraic matrix group is the same as a smooth affine algebraic group scheme. 
If in addition, $\mr{char}\, \Bbbk =0$, 
then every affine algebraic group scheme is smooth, and is, therefore, an algebraic matrix group. 

Assume $\mr{char}\, \Bbbk \ne 2$, and let 
$\ms{ASG}$ denote the category of affine algebraic super-group schemes. Suppose $\ms{G} \in \ms{ASG}$.
Let $O_{\ms{G}}$ denote the finitely generated Hopf super-algebra which represents $\ms{G}$, and 
let $\m_{\ms{G}}=\mr{Ker}\, \epsilon$ denote the kernel
of the counit of $O_{\ms{G}}$. 
The \emph{hyper-super-algebra} $\mr{hy}(\ms{G})$ \emph{of} $\ms{G}$ is the Hopf super-subalgebra
\[ \mr{hy}(\ms{G})=\bigcup_{n>0}(O_{\ms{G}}/\m_{\ms{G}}^n)^* \]
of the dual Hopf super-algebra $O_\ms{G}^{\circ}$ of $O_{\ms{G}}$; the dual 
$\hy(\ms{G})^*$ of this $\hy(\ms{G})$ is the
$\m_{\ms{G}}$-adic completion $\widehat{O_{\ms{G}}}$ of $O_{\ms{G}}$, or equally, 
the completion of the localized super-algebra
$(O_{\ms{G}})_{\m_{\ms{G}}}$. 
The \emph{Lie super-algebra} $\Lie(\ms{G})$ \emph{of} $\ms{G}$
is $P(\mr{hy}(\ms{G}))$. The constructions of $\hy(\ms{G})$ and of $\Lie(\ms{G})$ are functorial; cf. Proposition \ref{prop:hy_Lie}. 

As was touched before, an algebraic counterpart of Theorem \ref{thm:equivalence} (see \cite[Theorem 29]{M3}, \cite[Theorem 6.5]{M4})
shows a category equivalence between $\ms{ASG}$ and the category of
\emph{algebraic} Harish-Chandra pairs. 
Such a pair $(\ms{F},\sV)$ consists of an affine algebraic group scheme $\ms{F}$, a finite-dimensional $\ms{F}$-module
$\sV$ and an $\ms{F}$-equivariant linear map $[\ , \ ] :\sV \otimes \sV \to \Lie(\ms{F})$ which satisfy the
same formulas as (P1) and (P2) in Definition \ref{def:HCP}; $\Lie(\ms{F})$ is regarded as a right 
$\ms{F}$-module by the right adjoint action (induced from the right adjoint $\ms{F}$-action on $\ms{F}$). 

Let $\ms{G} \in \ms{ASG}$, and suppose that it corresponds to an algebraic
Harish-Chandra pair $(\ms{F},\sV)$. The affine algebraic group scheme $\ms{F}$ is the group-valued functor
obtained from $\ms{G}$, restricting the domain to the category of algebras, and it is represented by
$\overline{O_{\ms{G}}}=O_{\ms{G}}/((O_{\ms{G}})_1)$; see \eqref{eq:barA_grA}. 
We say that $\ms{G}$ is \emph{smooth} if $O_{\ms{G}}$ is smooth. This is equivalent to saying that 
each/one of $\ms{F}$, $\hy(\ms{F})$ and $\hy(\ms{G})$ is smooth (see \cite[Proposition A.3]{MZ}), 
and the equivalent conditions are necessarily satisfied if $\mr{char}\, \Bbbk =0$. 
We may say that $\ms{G}$ is \emph{connected}, precisely when
$\ms{F}$ is connected. 

\begin{definition}\label{def:AMSG}
We say that $\ms{G}$ is an \emph{algebraic matrix super-group}, if the intersection of the maximal
super-ideals in $O_{\ms{G}}$ of codimension $1$ is zero, or equivalently, if $\ms{F}$ is an algebraic matrix group.
We let $\ms{AMSG}$ denote the full subcategory of $\ms{ASG}$ consisting of all algebraic matrix super-groups.
\end{definition} 

Every algebraic matrix super-group is smooth. 
If $\Bbbk$ is an algebraically closed field of characteristic zero, we have $\ms{AMSG}=\ms{ASG}$.

Let us return to the situation that $\Bbbk$ is a complete field of characteristic $\ne2$.
The second corollary to Theorem \ref{thm:equivalence} is the following. 

\begin{corollary}\label{cor:algebraic_to_analytic}
We have a natural functor
\begin{equation}\label{eq:functor_an}
\ms{AMSG}\to \ms{SLG},\quad \ms{G} \mapsto \ms{G}^{\mr{an}}, 
\end{equation}
such that 
\begin{equation}\label{eq:hy_Lie}
\mr{hy}(\ms{G})=\mr{hy}(\ms{G}^{\mr{an}}),\quad \Lie(\ms{G})=\Lie(\ms{G}^{\mr{an}}). 
\end{equation}
The functor, restricted to the full subcategory of $\ms{AMSG}$ consisting of all connected objects,
is faithful. 
\end{corollary}
\begin{proof}
We will prove this in terms of Harish-Chandra pairs, depending on the established category
equivalences in algebraic and analytic situations. 

Let $\ms{F}$ be an affine algebraic group scheme. Then there is a closed embedding of $\ms{F}$ into some 
$\mr{GL}_n$. 
Let $N=n^2$. 
The $\mr{GL}_n$ is an algebraic matrix group (see \cite[Section 4.5, Corollary]{W}),
and the group $\mr{GL}_n(\Bbbk)$, naturally regarded as a Zariski closed subset of $\Bbbk^{N+1}$,
includes $\ms{F}(\Bbbk)$ as a Zariski closed subgroup. 
On the other hand, $\mr{GL}_n(\Bbbk)$, regarded as an open sub-manifold of $\Bbbk^{N}$,
is a Lie group, and it includes $\ms{F}(\Bbbk)$ as a closed
topological subgroup with respect to the Euclidian topology; it is finer than the Zariski topology. 

Assume that $\ms{F}$ is an algebraic matrix group, so that $O_{\ms{F}}$ is smooth, and
$O_{\ms{F}}=\mr{Poly}(\ms{F}(\Bbbk),\Bbbk)$. Let $d= \mr{Kdim}(O_{\ms{F}})$. 
The argument of \cite[Chapter II, Section 2.3]{Sh} shows that every point of
$\ms{F}(\Bbbk)$ has an open (in the Euclidian topology) neighborhood $U \subset \Bbbk^N$, and polynomials 
$f_1,\dots, f_{N-d}$ in $\Bbbk[\ms{T}_1,\dots, \ms{T}_N]$, such that 
$(f_1,\dots, f_{N-d}, \ms{T}_1,\dots,\ms{T}_d): U \to \Bbbk^N$ gives a homeomorphism onto an open set
of $\Bbbk^N$, and $\ms{F}(\Bbbk)\cap U=\{ \, x \in U \, | \, f_1(x)=\dots=f_{N-d}(x)=0\, \}$.  
It follows from Item 2) in
\cite[Part II, Chapter IV, Section 2]{S} that $\ms{F}(\Bbbk)$ is a Lie group, 
which we will denote by $\ms{F}^{\mr{an}}$. Being independent of choice of the embedding
$\ms{F}\subset \mr{GL}_n$, the topology on $|\ms{F}^{\mr{an}}|$ is defined so that the subsets
\[ 
B(p_1,\dots,p_r; \epsilon)=\{ \, g \in \ms{F}(\Bbbk) \, | \, |g(p_1)|<\epsilon,\dots, |g(p_r)|<\epsilon \, \}, 
\]
where $\epsilon >0$, and $p_1,\dots, p_r$ are finitely many elements of $O_{\ms{F}}$,  
give an open base; see \cite[Chapter VII, Section 1.1]{Sh}. 
The analytic structure of $\ms{F}^{\mr{an}}$ is precisely what is induced through those
local homeomorphisms $(q_1,\dots,q_d): U\to \Bbbk^d$ onto open sets of $\Bbbk^d$
which are given by polynomial functions $q_1,\dots, q_d$ in $O_{\ms{F}}$. 

We regard $\sf{F}(\Bbbk)$ naturally
as the ringed subspace of the affine scheme $\mr{Spec}(O_{\ms{F}})$; $\sf{F}(\Bbbk)$ thus consists of
the maximal ideals of codimension $1$, and its structure sheaf is the inverse image of the one 
$\mathcal{O}_{\mr{Spec}(O_{\ms{F}})}$ of $\mr{Spec}(O_{\ms{F}})$. 
Recall that $\ms{F}^{\mr{an}}$ is a ringed space, equipped with the sheaf $\cF_{\ms{F}^{\mr{an}}}$
of analytic functions. One sees that the identity map
\[
i=\mr{id} : \ms{F}^{\mr{an}} \to \ms{F}(\Bbbk)
\]
is a morphism of ringed spaces such that
$i^*_g : \mathcal{O}_{{\ms{F}(\Bbbk)}, g} \to \mathcal{F}_{\ms{F}^{\mr{an}},g}$
turns, by completion, into an isomorphism at every $g \in \ms{F}(\Bbbk)\, (=|\ms{F}^{\mr{an}}|)$. 
The isomorphism $\hat{i^*_e}$ at $e$, which is seen to preserve the $\ms{F}(\Bbbk)$-actions
induced by the right adjoint action of the group, gives the identifications 
\begin{equation}\label{eq:identify}
\mr{hy}(\ms{F})=\mr{hy}(\ms{F}^{\mr{an}}),\quad \Lie(\ms{F})=\Lie(\ms{F}^{\mr{an}})
\end{equation}
of $\ms{F}(\Bbbk)$-objects.

Giving a linear representation $\ms{F} \to \mr{GL}_r$ is the same as giving
a group homomorphism $\ms{F}(\Bbbk) \to \mr{GL}_r(\Bbbk)$ whose composites with the projections
$\op{pr}_{ij} : \mathrm{GL}_r(\Bbbk)\to \Bbbk$, $1\le i,j \le r$, (see Section \ref{subsec:repres_function}) are
polynomial functions; the group homomorphism is, therefore, regarded as an analytic representation 
$\ms{F}^{\mr{an}}\to \mr{GL}_r(\Bbbk)$. In view of \eqref{eq:identify}, as well, 
it follows that
given an algebraic Harish-Chandra pair $(\ms{F},\sV)$ with $\ms{F}$ an algebraic matrix group, 
the pair $(\ms{F}^{\mr{an}},\sV)$, equipped with the same $[\ , \ ]$ and the induced analytic module structure, 
is an (analytic) Harish-Chandra pair 
as defined by Definition \ref{def:HCP}. 
Moreover, the assignment $(\ms{F},\sV) \mapsto (\ms{F}^{\mr{an}},\sV)$ is functorial, and gives rise to the desired
functor $\ms{AMSG}\to \ms{SLG}$. 

By \cite[Theorem 10]{M2} every hyper-super-algebra uniquely arises from 
a dual Harish-Chandra pair \cite[Definiton 6]{M2} in a natural manner that 
dualizes the construction of super-groups from Harish-Chandra pairs, both in algebraic and analytic
situations.  
One sees easily that
$(\ms{F},\sV)$ and $(\ms{F}^{\mr{an}}, \sV)$ are dualized to the same dual Harish-Chandra pair, 
$(\mr{hy}(\ms{F}), \ms{V})=(\mr{hy}(\ms{F}^{\mr{an}}), \ms{V})$, from which, therefore, 
the same hyper-super-algebra,
$\hy(\ms{G})=\hy(\ms{G}^{\mr{an}})$, arises. Proved is the first equality of \eqref{eq:hy_Lie},
which, with $P(\ )$ applied, yields the second. 
The last assertion on faithfulness follows from the first equality,
since it is known the functor $\hy$ defined on $\ms{ASG}$ is faithful, when restricted to
the full subcategory consisting of all connected ones; see \cite[Proposition 21]{M3}.
\end{proof}

One sometimes encounters with such presentations of real or complex 
super Lie groups
that look presenting by use of matrices, algebraic matrix super-groups.
Corollary \ref{cor:algebraic_to_analytic} justifies such presentations;
they can be understood to
indicate the super Lie groups which arise from the actually presented algebraic matrix super-groups.
(In the real case one has to verify that the associated affine algebraic group scheme
is indeed an algebraic matrix group.)
Here is a familiar example.

\begin{example}\label{ex:GL}
Suppose that $\Bbbk$ is an infinite field of characteristic $\ne 2$. 
Let $\ms{M}=\ms{M}_0\oplus \ms{M}_1$ be a finite-dimensional super-vector space 
with $m = \dim(\ms{M}_0)$,\ $n = \dim(\ms{M}_1)$. 
Let $\ms{G}=\mr{GL}(\ms{M})$ 
denote the functor which associates to each super-algebra $R$ the group $\mr{GL}_R(\ms{M}\otimes R)$
of all $R$-super-module automorphisms of $\ms{M}\otimes R$. In fact this is a smooth affine algebraic super-group
scheme; it is represented by the finitely generated Hopf super-algebra
\[ O_{\ms{G}}= \Bbbk\big[\ms{X}_{ij}, \ms{Y}_{k\ell}, \mathrm{det}(\mathbf{X})^{-1},
\mathrm{det}(\mathbf{Y})^{-1}\big] 
\otimes \wedge(\ms{P}_{i\ell}, \ms{Q}_{kj}), \]
where $\ms{X}_{ij}$, $\ms{Y}_{k\ell}$ (resp., $\ms{P}_{i\ell}, \ms{Q}_{kj}$) 
are even (resp., odd) variables, and are presented as entries of the matrix 
\[
\left( \begin{matrix} \mathbf{X} & \mathbf{P} \\ \mathbf{Q} & \mathbf{Y} \end{matrix} \right)=
\left( \begin{matrix} \ms{X}_{ij} & \ms{P}_{i\ell} \\ \ms{Q}_{kj} & \ms{Y}_{k\ell} \end{matrix} \right),
\quad 1\le i, j \le m,\ 1 \le k, \ell \le n.
\]
The structure maps of $O_{\ms{G}}$ are given by
\begin{align*}
\Delta\left( \begin{matrix} \mathbf{X} & \mathbf{P} \\ \mathbf{Q} & \mathbf{Y} \end{matrix} \right) &=
\left( \begin{matrix} \mathbf{X} & \mathbf{P} \\ \mathbf{Q} & \mathbf{Y} \end{matrix} \right)
\otimes \left( \begin{matrix} \mathbf{X} & \mathbf{P} \\ \mathbf{Q} & \mathbf{Y} \end{matrix} \right),\quad
\varepsilon \left( \begin{matrix} \mathbf{X} & \mathbf{P} \\ \mathbf{Q} & \mathbf{Y} \end{matrix} \right) 
= \left( \begin{matrix} I & O \\ O 
& I \end{matrix} \right), \\
S\left( \begin{matrix} \mathbf{X} & \mathbf{P} \\ \mathbf{Q} & \mathbf{Y} \end{matrix} \right)
&=\left( \begin{matrix}
(\mathbf{X}-\mathbf{P}\mathbf{Y}^{-1}\mathbf{Q})^{-1} & -\mathbf{X}^{-1}\mathbf{P}\, S(\mathbf{Y})\\
-\mathbf{Y}^{-1}\mathbf{Q}\, S(\mathbf{X}) & (\mathbf{Y}-\mathbf{Q}\mathbf{X}^{-1}\mathbf{P})^{-1}
\end{matrix} \right). 
\end{align*}
The corresponding algebraic Harish-Chandra pair consists of the affine algebraic group scheme
\[ \sf{F}=\mr{GL}(\ms{M}_0) \times \mr{GL}(\ms{M}_1), \]
which is indeed an algebraic matrix group, and the super-vector space 
\[ \sV=\mr{Hom}(\ms{M}_0,\ms{M}_1)\oplus \mr{Hom}(\ms{M}_1,\ms{M}_0)\, (\subset \mr{End}(\ms{M})) \]
consisting of the parity-reversing linear endomorphisms of $\ms{M}$, on which $\ms{F}$ acts 
from the right so that
\begin{equation}\label{eq:HCP_GL}
v^g= g^{-1}\circ v \circ g\ \, \text{in}\ \, \mr{End}_R(\ms{M}\otimes R), 
\end{equation}
where $v \in \sV$, $g \in \ms{F}(R)$, and $R$ is an arbitrary algebra. The associated 
$[\ , \ ]: \sV \otimes \sV \to \Lie(\ms{F})= \mr{End}(\ms{M}_0) \oplus \mr{End}(\ms{M}_1)$ 
is given by the anti-commutator $[v,w]=v\circ w+w\circ v$. Note that the right $\ms{F}$-adjoint action 
on $\Lie(\ms{F})$ is given by the same formula as \eqref{eq:HCP_GL}. 

Return to the situation that $\Bbbk$ is a complete field. By Corollary \ref{cor:algebraic_to_analytic}
we have a super Lie group, $\ms{G}^{\mr{an}}$, which we denote by the same symbol $\mr{GL}(\ms{M})$. 
The corresponding Harish-Chandra pair is $(\ms{F}^{\mr{an}}, \sV)$ with which 
is associated the same $[\ , \ ]$ as above. We remark that $\ms{F}^{\mr{an}}= \ms{F}(\Bbbk)$ 
is the product of the general linear groups of $\ms{M}_0$ and of $\ms{M}_1$, which acts on $\sV$ and 
$\Lie(\ms{F}^{\mr{an}})\, (=\Lie(\ms{F}))$ by the same formula as \eqref{eq:HCP_GL} in the restricted 
situation $R=\Bbbk$.  
\end{example}

\begin{rem}\label{rem:analytification}
Suppose $\Bbbk=\mathbb{C}$. 
Corollary \ref{cor:algebraic_to_analytic} then follows alternatively by applying the following unpublished result due to A. Zubkov and the
second-named author, whose proof will be
published somewhere else; it is highly expected that an analogous result holds over any complete field of characteristic $\ne 2$. 
\emph{There is a natural functor from the category of smooth
locally algebraic super-schemes over $\mathbb{C}$ to the category of complex super-manifolds. It preserves finite products
and, therefore, group objects.}
\end{rem}

\subsection{The universal algebraic hull of a super Lie group}\label{subsec:super-repres_function}

Let $G$ be a super Lie group. Given a finite-dimensional super-vector space $\ms{M}$, 
a \emph{representation} of $G$ on $\ms{M}$ is 
a homomorphism $G \to \mr{GL}(\ms{M})$ of super Lie groups.
Given such a representation, $\ms{M}$ is called 
a \emph{finite-dimensional} (\emph{left}) $G$-\emph{super-module}. 
A (\emph{left}) $G$-\emph{super-module} is a filtered union of finite-dimensional (left) $G$-super-modules.  

Let $(F,\sV)=\Phi(G)$ be the Harish-Chandra pair corresponding to $G$. 
Let $\cB$ be the associated Hopf super-algebra constructed 
by Proposition \ref{prop:calB}, and let  
\[ G^{\mr{al}}=\op{SSp}(\cB) \]
denote the affine super-group scheme represented by $\cB$. 
A $G^{\mr{al}}$-super-module is naturally identified with a right $\cB$-super-comodule. 

\begin{theorem}\label{thm:G-super-mod_struc}
Let $\ms{M}$ be a super-vector space of possibly infinite dimension. 
There exist natural one-to-one correspondences among the
following three sets:
\begin{itemize}
\item[(a)] the set of all $G$-super-module structures on $\ms{M}$;
\item[(b)] the set of all $G^{\mr{al}}$-super-module structures on $\ms{M}$;
\item[(c)] the set of all pairs $(\cdot, \text{\SMALL{$\blacktriangleright $}})$ 
of an $F$-super-module structure $F \times \ms{M}\to \ms{M},\ (g,m)\mapsto g\cdot m$ 
on $\ms{M}$ and a super-linear map 
$\sV \otimes \ms{M} \to \ms{M}$, $v \otimes m \mapsto v\, \text{\SMALL{$\blacktriangleright $}}\, m$, with $\sV$ supposed to be purely odd, such that
\begin{align}
&g\cdot (v^g\, \text{\SMALL{$\blacktriangleright $}}\, m) = v\,
\text{\SMALL{$\blacktriangleright $}} (g\cdot m),\label{eq:FVsuper-module1}\\
&v\, \text{\SMALL{$\blacktriangleright $}}(w\, \text{\SMALL{$\blacktriangleright $}}\, m) + 
w\, \text{\SMALL{$\blacktriangleright $}}(v\, \text{\SMALL{$\blacktriangleright $}}\, m) = [v,w]\triangleright m, 
\label{eq:FVsuper-module2}
\end{align}
where $g \in F$, $v,w\in \sV$ and $m \in \ms{M}$. The last $\triangleright$ indicates 
the action by $\Lie(F)$ induced from the action $\cdot$ by $F$; cf. \eqref{eq:hy(F)-action}. 
\end{itemize}
\end{theorem}
\begin{proof}
The first (and main) step is to prove the one-to-one correspondence (b) $\leftrightarrow$ (c). 
Recall from Section \ref{subsec:quasi-inverse} (see Remark \ref{rem:hatA}, in particular)
that $\cB$ is constructed as a quotient of 
the complete Hopf super-algebra in the wider sense $\hat{\cA}$, which is regarded as
a complete Hopf super-subalgebra of $(F \ltimes \T(\sV))^*$. 

Suppose that we are given such a pair $(\cdot,\text{\SMALL{$\blacktriangleright $}})$
as in (c) that satisfies \eqref{eq:FVsuper-module1}, but may not satisfy \eqref{eq:FVsuper-module2}.
Then $\text{\SMALL{$\blacktriangleright $}}$ uniquely extends to a left 
$\T(\sV)$-super-module structure on $\ms{M}$; it will be denoted by the same symbol 
$\text{\SMALL{$\blacktriangleright $}}$.  
Moreover, 
\eqref{eq:FVsuper-module1} ensures that the extended structure and $\cdot$ make $\ms{M}$ into a
left $F\ltimes \T(\sV)$-super-module. Conversely, every left $F\ltimes \T(\sV)$-super-module 
structure on $\ms{M}$ uniquely arises in this way, 
provided the following condition (An) is satisfied:
(An)~the structure restricted to $F\times \ms{M}\to \ms{M}$ is analytic. 
\medskip

\noindent
{\bf Claim.}\ 
\emph{The pair $(\cdot,\text{\SMALL{$\blacktriangleright $}})$ above gives rise to a
topological right $\hat{\cA}$-super-comodule structure
on the discrete super-vector space $\ms{M}$ 
\[
\rho : \ms{M} \to \ms{M} \, \hat{\otimes}\, \hat{\cA} = \mr{Hom}(\mathfrak{T}(\sV), \ms{M} \otimes \cR(F)) 
\]
so that
\[ \rho(m)(z) = \rho_0(z \, \text{\SMALL{$\blacktriangleright $}} \, m),\quad m \in \ms{M},\ z \in\T(\sV), \]
where $\rho_0$ denotes the $\cR(F)$-super-comodule structure corresponding to the 
$F$- (or $F^{\mr{al}}$-)super-module structure $\cdot$ of the pair.
Conversely, every topological right $\hat{\cA}$-super-comodule structure on $\ms{M}$ 
uniquely arises in this way.}
\medskip
 
Indeed, the $F \ltimes \T(\sV)$-super-module structure on $\ms{M}$ which arises from the pair 
$(\cdot,\text{\SMALL{$\blacktriangleright $}})$
corresponds to the topological right $(F\ltimes \T(\sV))^*$-super-comodule structure
\[
\tilde{\rho} : \ms{M} \to \ms{M} \, \hat{\otimes}\, (F\ltimes \T(\sV))^* = 
\mr{Hom}(\mathfrak{T}(\sV), \mr{Map}(F, \ms{M})) 
\]
given by
\[
\tilde{\rho}(m)(z)(g)= g\cdot (z \, \text{\SMALL{$\blacktriangleright $}} \, m), 
\quad m \in \ms{M},\ z \in\T(\sV),\ g \in F;
\]
see Proposition \ref{prop:topological_comodule}. Here $\mr{Map}(F, \ms{M})$ denotes
the super-vector space of all maps $F \to \ms{M}$; it naturally includes $\ms{M} \otimes \cR(F)$.
One sees that
$\tilde{\rho}$ takes values in 
$\ms{M} \, \hat{\otimes}\, \hat{\cA} = \mr{Hom}(\mathfrak{T}(\sV), \ms{M} \otimes \cR(F))$; 
indeed, this is equivalent to the condition (An) above. 
This proves the Claim. 

It remains to show that the $\rho$ arising from the pair 
$(\cdot,\text{\SMALL{$\blacktriangleright $}})$ takes values in 
\[
\ms{M} \otimes \cB=\mr{Hom}_D(C, \ms{M}\otimes \cR(F))\ \, \text{in} \ \, 
\ms{M}\, \hat{\otimes}\, \hat{\cA}= \mr{Hom}_D(\mathcal{C}, \ms{M}\otimes \cR(F)),
\]
if and only if the pair
satisfies \eqref{eq:FVsuper-module2}, where $\ms{M}\otimes \cR(F)$ is 
regarded as a left $D$-module, with $D$ acting on the factor $\cR(F)$. 
This is easily seen from $\rho_0(a \triangleright m)= a \triangleright \rho_0(m)$, 
where $a \in \Lie(F)$,\ $m \in \ms{M}$. This completes the first step.

It results that given a pair of (c), $\ms{M}$ is a filtered union of finite-dimensional super-vector subspaces
which are stable under the pair of actions. Therefore, to prove (a) $\leftrightarrow$ (c), we may suppose
$\dim \ms{M} < \infty$. Then (c) is precisely the set of all morphism in $\ms{HCP}$, 
from $(F,\sV)$ to the Harish-Chandra pair of $\mr{GL}(\ms{M})$ given in Example \ref{ex:GL}. Hence  
the desired result follows by Theorem \ref{thm:equivalence}. 
\end{proof}

In view of Theorem \ref{thm:G-super-mod_struc}, $\cB$ and $G^{\mr{al}}$ may be called
\emph{the Hopf super-algebra of analytic representative functions}
on the super Lie group $G$, and \emph{the universal algebraic hull of} $G$, respectively.
These are super-analogues of the corresponding objects investigated by 
Hochschild and Mostow \cite{HM}; one of them is
the universal algebraic hull $F^{\mr{al}}$ of a Lie group $F$, discussed before.


\section{New construction of the quotient $G/H$}\label{sec:quotient}

\subsection{Super Lie subgroups}\label{subsec:super_Lie_sub}
Let $G$ be a super Lie group. By a \emph{super Lie subgroup} of $G$, 
we mean a super Lie group $H$, given a morphism $\phi : H \to G$ of super Lie groups
such that the associated morphism $(\phi_{\red}, \Lie(\phi)_1): (H_{\red},\Lie(H)_1) \to (G_{\red},\Lie(G)_1)$
of Harish-Chandra pairs satisfies
\begin{itemize}
\item[(i)] $\phi_{\red}: H_{\red} \to G_{\red}$ is a closed embedding, and
\item[(ii)] $\Lie(\phi)_1\, (=\Lie(\phi)|_{\Lie(H)_1}) : \Lie(H)_1 \to \Lie(G)_1$ is injective.
\end{itemize}
Obviously, (ii) may be replaced with one of the following conditions which are equivalent to each other:
\begin{itemize}
\item[(iii)] $\Lie(\phi):\Lie(H) \to \Lie(G)$ is injective;
\item[(iii)$'$] $\hy(\phi) : \hy(H) \to \hy(G)$ is injective.
\end{itemize}
We present $\phi$ so as $H \subset G$, and regard $\phi_{\red}$, $\Lie(\phi)$ and $\hy(\phi)$ as inclusions. 
The Lie group $H_{\red}$ is, therefore, a Lie subgroup of $G_{\red}$; see \cite[Definition 6.6, Proposition 6.7]{We}. 

Notice that $G_{\red}$ is a super Lie subgroup of $G$. One sees easily that the following are equivalent:
\begin{itemize}
\item[(a)] The inclusion $G_{\red}\hookrightarrow G$ splits as a morphism of super Lie groups;
\item[(b)] The Lie super-algebra $\Lie(G)$ of $G$ is $\mathbb{N}$-graded, or equivalently, $[\Lie(G)_1, \Lie(G)_1]=0$.
\end{itemize}
If these are satisfied we say that $G$ is \emph{$\mathbb{N}$-graded}. 
We remark that $G_{\red}\hookrightarrow G$ then has a canonical retraction since the corresponding morphism
$(G_{\red}, 0) \to (G_{\red},\Lie(G)_1)$ of Harish-Chandra pairs does. 

Modify the Harish-Chandra pair $(G_{\red}, \Lie(G)_1)$ corresponding to $G$, replacing the equipped bracket with the zero map.
To the modified Harish-Chandra pair corresponds an $\mathbb{N}$-graded super Lie group, which we denote by $\ms{gr}\, G$. 
This $\ms{gr}\, G$ is intrinsically constructed from the original $G$, by applying $\ms{gr}$ to the structure sheaf of $G$ and 
to the sheaf-morphisms $\Delta$, $S$, such as given by \eqref{eq:Delta_S}, 
that are associated with the product $\mu$ and the inverse $\iota$ on $G$; see \cite[Section 4]{V}, for example. 


\subsection{The theorem on $G/H$}\label{subsec:quotient_statement}
Let $G=(|G|,\cO)$ be a super Lie group with the associated Lie group $G_{\red}=(|G|,\cF)$.
We thus denote the structure sheaves $\cO_G$, $\cF_{G_{\red}}$ by $\cO$, $\cF$, simply. 
We set $\sW=(\Lie(G)_1)^*$. 

Recall that $G$ is a group-object in $\ms{SMFD}$. A left (resp., right) $G$-equivariant object will be called
a \emph{left} (resp., \emph{right}) \emph{$G$-equivariant super-manifold}. Note that $G$ itself is left and right $G$- and $G_{\red}$-equivariant. 
Recall from Proposition \ref{prop:eta} and \eqref{eq:cP}
the canonical sheaf-isomorphism $\eta$
composed with a non-canonical one,
\[ 
\cO \overset{\simeq}{\longrightarrow} \cF \otimes \wedge(\sW). 
\]
In view of \eqref{eq:Delta(p)} (suppose $a=1$ there), this last gives an isomorphism  
\begin{equation*}
G_{\red}\times (\{*\},~\wedge(\sW)) \overset{\simeq}{\longrightarrow} G
\end{equation*}
of left $G_{\red}$-equivariant super-manifolds, which is reduced to the identity map $\mr{id}_{G_{\red}}$ on $G_{\red}$. 
We have an analogous, right $G_{\red}$-equivariant isomorphism 
\begin{equation}\label{eq:Gred-equiv}
(\{*\},~\wedge(\sW)) \times G_{\red} \overset{\simeq}{\longrightarrow} G,
\end{equation}
which will be of use, in fact. 
If $G$ is $\mathbb{N}$-graded, we can choose as the last two isomorphisms, canonical isomorphisms of super Lie groups,
replacing the direct products on the left-hand side with semi-direct products; see
the remark given in the second last paragraph of the preceding subsection.

Recall that $H_{\red}$ is a Lie subgroup of $G_{\red}$. 
Notice from Proposition \ref{prop:Gred-super-module} 
that $H_{\red}$ acts analytically by left adjoint on the hyper-super-algebra $\hy(H)$ of $H$, so that we have
the super-cocommutative Hopf super-algebra of semi-direct product
\[ 
J:=\hy(H)\rtimes H_{\red}, 
\]
in which every element of $H_{\red}$ is supposed to be group-like. This includes
\[ 
K:=\hy(H_{\red})\rtimes H_{\red} 
\]
as the largest purely even Hopf super-subalgebra. 
Just as above and in what follows, as well,
 (super-)modules over a Lie group or its actions are explicitly said to be \emph{analytic}, if that is the case;
cf. Convention \ref{con:F-super-module}. 

It is known (see \cite[Theorem 1 on Page 108]{S}) that the set of left cosets
\[ 
Q:=G_{\red}/H_{\red},
\]
equipped with the quotient topology, is uniquely made into a manifold so that the quotient morphism
which we denote by
\[
\pi : G_{\red}\to Q=G_{\red}/H_{\red}
\]
is a submersion, that is, $T_g\pi : T_g(G_{\red}) \to T_{\pi(g)}(Q)$ is surjective for every $g \in G_{red}$; 
the symbols are and will be used to present the underlying topological spaces $|G_{\red}|$, $|H_{\red}|$,
$|Q|$ or the continuous map $|\pi|$, as well. 
Given an open set $U$ of $Q$, the open submanifold $\pi^{-1}(U)$ of $G_{\red}$ is right $H_{\red}$-stable 
(i.e., stable under the right multiplication by $H_{\red}$),
and so the induced action makes $\cF(\pi^{-1}(U))$ into a left $H_{\red}$-module algebra. 
With this action combined with the natural action by $\hy(H_{\red})$ on $\cF$ (see Proposition \ref{prop:hy(G_red)-module}), 
$\cF(\pi^{-1}(U))$ turns into a left $K$-module algebras. 
Moreover, $\pi_*\cF$ turns into a sheaf of those algebras; 
see the proof of Lemma \ref{lem:J-super-algebra} below. 
One sees that the structure sheaf $\cF_Q$ of $Q$ associates to every open set $U\subset Q$ the subalgebra 
\begin{equation}\label{eq:FQ}
\cF_Q(U)=\cF(\pi^{-1}(U))^K
\end{equation}
of $\cF(\pi^{-1}(U))$ which consists of the $K$-invariants; in general, an \emph{invariant} is defined to be an element
$f$ such that $a\triangleright f=\epsilon(a)f$ for all elements $a$ of the Hopf (super-)algebra in question.  
It is easy to see $\cF_Q(U)\subset \cF(\pi^{-1}(U))^K$, and the equality follows
by passing to the stalks, using the fact:
$\pi$ is a principal $H_{\red}$-bundle \cite[Example 8.21]{We}. 
This means that $Q$ is covered by admissible open subsets.
Here we say that an open set $\emptyset \ne U \subset Q$ is \emph{admissible} if there exists a \emph{trivialization}
\begin{equation}\label{eq:triv}
\ms{triv}_U : U\times H_{\red} \overset{\simeq}{\longrightarrow} \pi^{-1}(U),
\end{equation}
that is, a right $H_{\red}$-equivariant isomorphism which makes the diagram
\[
\begin{xy}
(0,0)   *++{U\times H_{\red}}  ="1",
(30,0)  *++{ \pi^{-1}(U)}    ="2",
(15,-18) *++{U}          ="3",
{"1" \SelectTips{cm}{} \ar @{->}^{\ms{triv}_U} "2"},
{"1" \SelectTips{cm}{} \ar @{->}_{\mr{pr}} "3"},
{"2" \SelectTips{cm}{} \ar @{->}^{\pi} "3"}
\end{xy}
\]
commutative, where $\mr{pr}$ denotes the natural projection.

Notice from Proposition \ref{prop:hy(G_red)-module} 
that $\cO$ is a sheaf of left $\hy(H)$-super-module super-algebras. 

\begin{lemma}\label{lem:J-super-algebra}
Given an open set $U$ of $Q$,\ $\pi^{-1}(U)$, regarded as an open super-submanifold of $G$, is right $H_{\red}$-stable.
The induced left $H_{\red}$-action on $\cO(\pi^{-1}(U))$, combined with the action by $\hy(H)$ noticed above,
makes $\cO(\pi^{-1}(U))$ into a left $J$-super-module super-algebra. This action by $J$ is natural in $U$. 
\end{lemma}
\begin{proof}
The first stability follows from \eqref{eq:Gred-equiv}, while the last naturality follows since the actions by
$H_{\red}$ and by $\hy(H)$ are both natural. 

To complete the proof, it suffices to prove that for $h \in H_{\red}$,\ $a\in \hy(H)$, 
the induced actions 
\[ {}^ha\quad \text{and}\quad h\circ a\circ h^{-1} \]
on completed stalks $\hO_g$, $g \in \pi^{-1}(U)$, coincide. Notice that the action by $h$ is
\[
\hO_{gh}\overset{\hat{\Delta}_{g,h}}{\longrightarrow}\hO_g\, \hat{\otimes}~\hO_h
\overset{\mr{id} \hat{\otimes} \hat{\epsilon}_h}{\longrightarrow} \hO_g. 
\]
Then the desired result follows easily from the relation \eqref{eq:F-balanced}
modified for the opposite-sided actions. 
\end{proof}

\begin{definition}\label{def}
We let $G/H$ denote the super-ringed space with the underlying topological space $Q=G_{\red}/H_{\red}$, 
equipped with the sheaf which associates to every open set $U$, 
\[
\cO_{G/H}(U):= \cO(\pi^{-1}(U))^J,
\]
the super-subalgebra of $J$-invariants in $\cO(\pi^{-1}(U))$; this is seen to define a sheaf by the naturality
shown in the preceding lemma. 

We define a morphism of super-ringed spaces
\[
\Pi =(\pi, \Pi^*) : G \to G/H, 
\]
which consists of the continuous map $\pi : |G|=G_{\red} \to Q= |G/H|$ and the
the sheaf-morphism
\[
\Pi^* : \cO_{G/H}(U) =\cO(\pi^{-1}(U))^J \hookrightarrow \cO(\pi^{-1}(U))=\pi_*\cO(U)
\]
given by the inclusion.
\end{definition}

\begin{theorem}\label{thm:quotient}
We have the following.
\begin{itemize}
\item[(1)]
The super-ringed space $G/H$ is a super-manifold, whence $\Pi : G \to G/H$ is a morphism in the category
$\ms{SMFD}$ of super-manifolds. This $\Pi$ has the following properties:
\begin{itemize}
\item[(i)]
For every $g \in |G|$,\ $d\Pi_g : T_gG \to T_{\pi(g)}(G/H)$ is surjective; 
\item[(ii)]
$G \times H \rightrightarrows G \overset{\Pi}{\longrightarrow} G/H$ is 
a co-equalizer diagram in $\ms{SMFD}$, where
the paired arrows indicate the product on $G$ restricted to $G \times H$ and the natural projection. 
\end{itemize}
\item[(2)] 
Given an admissible open set $U$ of $Q$ and a trivialization such as in \eqref{eq:triv}, 
$\pi^{-1}(U)$, regarded as an open super-submanifold of $G$, is right $H$-stable, and
there exists an isomorphism 
\begin{equation*}
\ms{Triv}_U : (U,\cO_{G/H}|_U) \times H \overset{\simeq}{\longrightarrow} (\pi^{-1}(U), \cO|_{\pi^{-1}(U)})
\end{equation*}
of right $H$-equivariant super-manifolds such that
\[
(\ms{Triv}_U)_{\red}=\ms{triv}_U,
\]
and the diagram
\[
\begin{xy}
(0,0)   *++{(U,\cO_{G/H}|_U) \times H}  ="1",
(44,0)  *++{(\pi^{-1}(U), \cO|_{\pi^{-1}(U)})}    ="2",
(22,-22) *++{(U,\cO_{G/H}|_U)}          ="3",
{"1" \SelectTips{cm}{} \ar @{->}^{\ms{Triv}_U} "2"},
{"1" \SelectTips{cm}{} \ar @{->}_{\mr{projection}} "3"},
{"2" \SelectTips{cm}{} \ar @{->}^{\Pi} "3"}
\end{xy}
\]
is commutative. 
\end{itemize}
\end{theorem}

\begin{rem}\label{rem:algebra_analogue}
(1)\
Due to the property (i) shown in Pat 1 above, the super-manifold $G/H$ may be called the \emph{quotient of $G$ by $H$}. 
Due to the property shown in Part 2, $\Pi : G\to G/H$ may be called a \emph{principal super $H$-bundle}. 
In fact, there is more shown: any non-super trivialization
can \emph{lift} to one in the super context. This result is likely new even when 
$\Bbbk = \mathbb{R}$ or $\mathbb{C}$. 

(2)\
Very recently, the second- and third-named authors \cite{MT} applied Hopf-algebraic 
techniques to give 
a new construction of the quotient $G/H$ in the algebraic situation
where $G$ is an affine algebraic super-group scheme over an arbitrary
field of characteristic $\ne 2$, and $H$ is a closed super-subgroup scheme.  
Pursuing an analogue of the construction in the present analytic situation 
resulted successfully in the theorem above; it was a surprise to the authors. 
\end{rem}


\subsection{Proof of the theorem}\label{subsec:quotient_proof}
Choose arbitrarily 
an admissible open set $U \subset Q$ and a trivialization such as in \eqref{eq:triv}, 
and fix them. 
One sees that 
the direct image $\mr{pr}_*\cF_{U\times H_{red}}$ of 
the structure sheaf $\cF_{U\times H_{red}}$ of $U\times H_{\red}$
is a sheaf of left $K$-module algebras, and the sheaf-isomorphism 
\[ 
(\ms{triv}_U)^* : \pi_*\cF|_U\overset{\simeq}{\longrightarrow}  \mr{pr}_*\cF_{U\times H_{red}}
\]
associated with $\ms{triv}_U$ preserves the $K$-action. 
Compose the inverse $((\ms{triv}_U)^*)^{-1}$ 
with the algebra map $\cF_{H_{\red}}(H_{\red}) \to \cF_{U\times H_{\red}}(U\times H_{\red})$ associated with
the projection $U \times H_{\red} \to H_{\red}$, and restrict the resulting $\cF_{H_{\red}}(H_{\red}) \to \pi_*\cF|_U$ to the Hopf
algebra $\cR(H_{\red})$ of the analytic representative functions on $H_{\red}$; see Section \ref{subsec:repres_function}. 
We denote the thus (non-canonically) obtained map by
\[ 
\varrho_{H_{\red}} :\cR(H_{\red}) \to \pi_* \cF|_U, 
\]
through which $\pi_* \cF|_U$ turns into a sheaf of $K$-module algebras over $\cR(H_{\red})$
(that is, every $\varrho_{H_{\red},V} :\cR(H_{\red}) \to \cF(\pi^{-1}(V))$ is a $K$-module algebra map) on $U$. 
On the other hand 
we let
\[
\sigma_{G_{\red}} : \cR(G_{\red}) \to \pi_* \cF|_U
\]
denote the canonical map $\cF(G_{\red})\to \cF|_U$ restricted to the Hopf algebra $\cR(G_{\red})$ of the
analytic representative functions on $G_{\red}$, through which $\pi_* \cF|_U$ turns into a sheaf of $K$-module
algebras over $\cR(G_{\red})$ on $U$. 

Since $\Lie(G)_1$ is an analytic right $G_{\red}$-module by adjoint (see Remark \ref{rem:adjoint_action}), 
$\sW=(\Lie(G)_1)^*$ is an analytic left module over $G_{\red}$ and hence over $H_{\red}$. 
Let $\sW_H=(\Lie(H)_1)^*$. This as well is an analytic left $H_{\red}$-module. 
Recall $\Lie(H) \subset \Lie(G)$, and define an analytic left $H_{\red}$-module by
\[
\sZ =(\Lie(G)_1/\Lie(H)_1)^*. 
\]
Alternatively, this is defined by the short exact sequence  
$0 \to \sZ \to \sW \to \sW_H \to 0$
of analytic left $H_{\red}$-modules. 
Recall that analytic left $H_{\red}$-modules are identified with right $\cR(H_{\red})$-comodules.
Such comodules are naturally regarded as left $K$-modules by virtue of
the canonical Hopf paring $K\times \cR(H_{\red})\to \Bbbk$; cf.~\eqref{eq:Hopf_pairing}.

By a \emph{$(K, \pi_*\cF|_U)$-module} (\emph{sheaf}) we mean a $\pi_*\cF|_U$-module $\mathcal{M}$ 
which is at the same time a sheaf of left $K$-modules such that the structure $\pi_*\cF|_U\otimes \mathcal{M}\to \mathcal{M}$
is $K$-linear. A \emph{morphism} of such modules is a $K$-linear and $\pi_*\cF|_U$-linear one.
We are going to define two morphisms of $(J,\pi_*\cF|_U)$-modules,  
\begin{equation}\label{eq:composite}
\sW \otimes \pi_*\cF|_U\overset{\tilde{\kappa}_{\sW}^{-1}}{\longrightarrow} \pi_*\cF|_U\otimes \sW
\overset{\tilde{\theta}}{\longrightarrow} \pi_*\cF|_U\otimes \sZ,
\end{equation}
the first of which is the inverse of an isomorphism $\tilde{\kappa}_{\sW}$. 

In general, given a right $\cR(H_{\red})$-comodule $\ms{M}$ with the structure
$\ms{M} \to \ms{M}\otimes \cR(H_{\red}),\ m \mapsto m^{(0)} \otimes m^{(1)}$, we define
a map by
\[ 
\nu_{\ms{M}} : \cR(H_{\red})\otimes \ms{M}\to \ms{M}\otimes \cR(H_{\red}),\
\nu_{\ms{M}}(p\otimes m)=m^{(0)}\otimes p m^{(1)}. 
\]
In fact this is an isomorphism of $\cR(H_{\red})$-Hopf modules
\cite[p.83]{Sw}.  
The inverse is given by
\[ 
\nu_{\ms{M}}^{-1}(m\otimes p)=p\, S_{\cR(H_{\red})}(m^{(1)})\otimes m^{(0)}. 
\]
In the obvious manner, $\cR(H_{\red})\otimes \ms{M}$ 
and $\ms{M}\otimes \cR(H_{\red})$ both are
regarded as $\cR(H_{\red})$-modules. The former is regarded as the tensor product of
two right $\cR(H_{\red})$-comodules, while the latter is regarded as the right $\cR(H_{\red})$-comodule
$\cR(H_{\red})$ tensored with the vector space $\ms{M}$. Being $\cR(H_{\red})$-colinear,
$\nu_{\ms{M}}$ is $K$-linear. 
Choose arbitrarily a linear retraction 
\[ \ms{ret} : \sW \to \sZ \]
of the inclusion $\sZ \hookrightarrow \sW$, and define $\theta$ to be the unique map that
makes
\[
\begin{xy}
(0,0)   *++{\cR(H_{\red}) \otimes \sW}  ="1",
(36,0)  *++{\sW \otimes \cR(H_{\red})}    ="2",
(0,-18) *++{\cR(H_{\red}) \otimes \sZ}          ="3",
(36,-18)*++{\sZ \otimes \cR(H_{\red})}    ="4",
{"1" \SelectTips{cm}{} \ar @{->}^{\nu_{\sW}} "2"},
{"1" \SelectTips{cm}{} \ar @{->}_{\theta} "3"},
{"3" \SelectTips{cm}{} \ar @{->}^{\nu_{\sZ}} "4"},
{"2" \SelectTips{cm}{} \ar @{->}^{\ms{ret} \otimes \mr{id}} "4"}
\end{xy}
\]
into a commutative diagram of $\cR(H_{\red})$-Hopf modules; it is easy to express this $\theta$ explicitly, but we will not use the
expression. 
Let
\[
\tilde{\theta} : \pi_*\cF|_U \otimes \sW \to \pi_*\cF|_U\otimes \sZ
\]
be the base extension of $\theta$ along $\varrho_{H_{\red}}$. 
This $\tilde{\theta}$ is the second one of \eqref{eq:composite}, and is seen, indeed, to be a morphism between
two $(K, \pi_*\cF|_U)$-modules. 

For the $\nu_{\sW}$ above, $\sW$ is regarded as a right $\cR(H_{\red})$-comodule. 
When it is regarded as an original, right $\cR(G_{\red})$-comodule, an analogous
isomorphism
\[
\kappa_{\sW} : \cR(G_{\red}) \otimes \sW \overset{\simeq}{\longrightarrow} \sW\otimes\cR(G_{\red}) 
\]
of $\cR(G_{\red})$-Hopf modules is defined in the obvious manner. We let
\[
\tilde{\kappa}_{\sW} : \pi_*\cF|_U \otimes \sW \overset{\simeq}{\longrightarrow} \sW\otimes\pi_*\cF|_U
\]
be its base extension along $\sigma_{G_{\red}}$, which is seen to be a morphism between
two $(K, \pi_*\cF|_U)$-modules. 
Its inverse is the first one of \eqref{eq:composite}.

Let
\[ 
\Theta=\tilde{\theta} \circ \tilde{\kappa}_{\sW}^{-1} : \sW \otimes \pi_*\cF|_U\to \pi_*\cF|_U\otimes \sZ
\]
be the composite presented by \eqref{eq:composite}.

\begin{lemma}\label{lem:retraction}
This $\Theta$ is a retraction of the injection of $(K, \pi_*\cF|_U)$-modules
\[
\tilde{\kappa}_{\sW}|_{\sZ \otimes \pi_*\cF|_U} : \sZ \otimes \pi_*\cF|_U\to \pi_*\cF|_U\otimes \sW, 
\]
that is, the restriction of $\tilde{\kappa}_{\sW}$ to $\sZ \otimes \pi_*\cF|_U$. 
\end{lemma}

\begin{proof}
This follows since one sees from the last commutative diagram that $\theta$, restricted to
$\sZ \otimes \pi_*\cF|_U$, turns into the identity. 
\end{proof}

\begin{rem}\label{rem:theta}
(1)\
Let
\[
\tilde{\nu}_{\sW} : \pi_*\cF|_U \otimes \sW \to \sW \otimes \pi_*\cF|_U,\quad
\tilde{\nu}_{\sZ} : \pi_*\cF|_U \otimes \sZ \to \sZ \otimes \pi_*\cF|_U
\]
be the base extensions of $\nu_{\sW}$ and $\nu_{\sZ}$ along $\varrho_{H_{\red}}$.
For later use define
\begin{equation}\label{eq:nuTheta}
\Xi:=\tilde{\nu}_{\sZ}\circ \Theta : \sW \otimes \pi_*\cF|_U\to \sZ\otimes \pi_*\cF|_U. 
\end{equation}
We remark that this $\Xi$ does not necessarily equal $\ms{ret} \otimes \mathrm{id}_{\pi_*\cF|_U}$
since $\tilde{\nu}_{\sW}\circ \tilde{\kappa}_{\sW}^{-1}\ne \mathrm{id}_{\sW\otimes \pi_*\cF|_U}$, in general.

(2)\ Our $\tilde{\theta}$, $\Theta$ are analogues of the $\theta$, $\theta'$ of \cite[(4.14), (4.16)]{MT},
respectively; see Remark \ref{rem:algebra_analogue} (2). 
\end{rem}

Recall that the sheaf $\cO$ is $\hy(H)$-super-module super-algebras. 
Given an open set $V$ of $|G|$, the set
\[ 
\mr{Hom}_{\hy(H_{\red})}(\hy(H),~\cO(V))
\] 
of all left $\hy(H_{\red})$-module maps $\hy(H)\to \cO(V)$ form
a super-algebra 
with respect to the convolution-product, which is non-canonically isomorphic to $\cO(V)\otimes \wedge(\sW)$.
This is, moreover, a left $\hy(H)$-super-module with respect to the structure 
induced from the right multiplication by $\hy(H)$ on itself.
Thus there arises another sheaf of $\hy(H)$-super-module super-algebras
\[
\mr{Hom}_{\hy(H_{\red})}(\hy(H),~\cO)
\] 
on $|G|$, which associates $\mr{Hom}_{\hy(H_{\red})}(\hy(H),~\cO(V))$ to every $V$.  
 
We see that
\[
\Lambda_V : \cO(V) \to \mr{Hom}_{\hy(H_{\red})}(\hy(H),~\cO(V)),~~\Lambda_V(p)(a)=a\triangleright p
\]
defines a $\hy(H)$-super-linear sheaf-morphism
\[
\Lambda : \cO \to \mr{Hom}_{\hy(H_{\red})}(\hy(H),~\cO).
\]

Recall that the direct image $\pi_*\cO$ of $\cO$ is a sheaf of $J$-super-module
super-algebras. 
We wish to see that
\[ 
\pi_*\mr{Hom}_{\hy(H_{\red})}(\hy(H),~\cO)\, (=\mr{Hom}_{\hy(H_{\red})}(\hy(H),~\pi_*\cO))
\]
is as well. 

\begin{lemma}\label{lem:J-super-linear}
This last sheaf on $Q$ turns into a sheaf of $J$-super-module super-algebras by defining the $H_{\red}$-action
\[
{}^h\varphi: \hy(H)\to \cO(\pi^{-1}(U)),\ a\mapsto {}^h(\varphi({}^{h^{-1}}a)),
\]
where $U\subset Q$ is any open set,\ $h \in H_{\red}$ and $\varphi \in \mr{Hom}_{\hy(H_{\red})}(\hy(H),$ $\cO(\pi^{-1}(U)))$. 
Moreover, the sheaf-morphism 
\[
\pi_*\Lambda : \pi_*\cO \to \pi_*\mr{Hom}_{\hy(H_{\red})}(\hy(H),~\cO).
\]
is then $J$-super-linear. 
\end{lemma}
\begin{proof}
This follows since one sees that
\[ 
{}^h(a \triangleright p)=({}^ha) \triangleright ({}^hp),\quad {}^h(({}^{h^{-1}}a) \triangleright p) 
={}^{h}({}^{h^{-1}}a) \triangleright ({}^hp)=a \triangleright ({}^hp), 
\]
where $h \in H_{\red}$, $a \in \hy(H)$ and $p \in \cO(\pi^{-1}(U))$. 
\end{proof}

Let us return to the situation that $U$ is a fixed, admissible open set of $Q$. 
Choose an isomorphism such as in \eqref{eq:Gred-equiv} and let
\begin{equation}\label{eq:cOcF}
\cO \overset{\simeq}{\longrightarrow} \wedge(\sW)\otimes \cF
\end{equation} 
be the associated sheaf-isomorphism. 
We let
\[
\Sigma_U: \pi_*\cO|_U \overset{\simeq}{\longrightarrow} \wedge(\sW)\otimes \pi_*\cF|_U 
\xrightarrow{\wedge_{\pi_*\cF|_U}(\Xi)}\wedge(\sZ)\otimes \pi_*\cF|_U, 
\]
be its direct image under $\pi$, composed with the sheaf-morphism $\wedge_{\pi_*\cF|_U}(\Xi)$
over $\pi_*\cF|_U$
which naturally arises from the $\Xi$ in \eqref{eq:nuTheta}. 
The thus obtained morphism is seen to be $K$-linear. Apply $\mr{Hom}_{\hy(H_{\red})}(\hy(H),\ \, )$
and compose with the $\pi_*\Lambda$ restricted to $U$. 
We let
\begin{equation*}
\Omega_U : \pi_*\cO|_U \to \mr{Hom}_{\hy(H_{\red})}(\hy(H),\ \wedge(\sZ)\otimes \pi_*\cF|_U) 
\end{equation*}
be the resulting sheaf-morphism.

\begin{prop}\label{prop:Omega}
We have the following.
\begin{itemize}
\item[(1)]
This $\Omega_U$ is a $J$-super-linear isomorphism of sheaves on $U$, such that
\[ (\Omega_U)_{\red}= \mr{id}_{\pi_*\cF|_U}. \]
\item[(2)]
The restriction to the $J$-invariants gives an isomorphism
\[
(\Omega_U)^J : \cO_{G/H}|_U=\cO(\pi^{-1}(\ ))^J\overset{\simeq}{\longrightarrow}\wedge(\sZ)\otimes \cF_{Q}|_U
\]
of sheaves on $U$. 
Hence, $G/H$, being locally a s-split super-manifold, is a super-manifold. 
\end{itemize}
\end{prop}
\begin{proof}
(1)\
Non-trivial is to prove that $\Omega_U$ is an isomorphism. For this it suffices to prove that
the analogous sheaf-morphism
\[
\Omega_U' : \pi_*\cO|_U \to \mr{Hom}_{\hy(H_{\red})}(\hy(H),\ \pi_*\cF|_U \otimes \wedge(\sZ)),
\]
which is constructed as $\Omega_U$ but with $\wedge(\Xi)$ replaced by
$\wedge(\Theta)$, is isomorphic. 
We have to prove that $\ms{gr}(\Omega_U')$ is isomorphic. 
But we wish, equivalently, to prove 
that $\Omega_U'$ is isomorphic, assuming $G=\ms{gr}\, G$, or that $G$ is $\mathbb{N}$-graded; see Section \ref{subsec:super_Lie_sub}. 
We should then take the isomorphism \eqref{eq:cOcF} (or \eqref{eq:Gred-equiv}), which was used to construct $\Omega_U'$,
as the canonical one.
The desired result follows by essentially the same argument as proving \cite[Proposition 4.8]{MT}
(see the argument of proving that the $(\omega_{\theta})_P$ in \cite[(4.19)]{MT} is isomorphic); 
it depends on the equivalence from the category $\mathsf{SMod}_{\wedge(\sW)}^{\wedge(\sW_H)}$ 
of $(\wedge(\sW_H),\wedge(\sW))$-Hopf super-modules to the category $\mathsf{SMod}_{\wedge(\sZ)}$ of 
$\wedge(\sZ)$-super-modules (see \cite[Theorem 3.1 (1)]{MT}), which associates to each $\mathsf{M} \in \mathsf{SMod}_{\wedge(\sW)}^{\wedge(\sW_H)}$, 
the $\wedge(\sZ)$-super-module $\mathsf{M}^{co\wedge(\sW_H)}=\{m\in \mathsf{M}\mid m^{(0)}\otimes m^{(1)}=m\otimes 1 \}$ 
of right $\wedge(\sW_H)$-coinvariants in $\mathsf{M}$; the equivalence
is ensured by the fact that $\sW\to \sW_H$ is surjective. 
As in the proof for $(\omega_{\theta})_P$ cited above, one sees that under the assumption above, $\Omega_U'$ is a morphism in 
$\mathsf{SMod}_{\wedge(\sW)}^{\wedge(\sW_H)}$. It will follow by the category-equivalence above
that $\Omega'_U$ is isomorphic, as desired, 
if one sees that the restriction  
\[
(\Omega_U')^{co\wedge(\sW_H)} : (\pi_*\cO|_U)^{co\wedge(\sW_H)} \to 
\big(\mr{Hom}_{\hy(H_{\red})}(\hy(H), \pi_*\cF|_U \otimes \wedge(\sZ))\big)^{co\wedge(\sW_H)}
\]
to the right $\wedge(\sW_H)$-coinvariants is isomorphic.
The domain is naturally identified with $\pi_*\cF|_U\otimes \wedge(\sZ)$,
while the range is identified with the left $\wedge(\sW_H^*)$-invariants in 
the range of $\Omega'_U$. One will see that the latter is further identified with 
$\pi_*\cF|_U\otimes \wedge(\sZ)$, again. 
Note $\sW_H^*= \Lie(H)_1$, and that $\wedge(\sW_H^*)$ is now a Hopf super-subalgebra of $\hy(H)$; it
is supposed to act on the last mentioned range through the right multiplication on $\hy(H)$. 
Since $\hy(H)\otimes_{\wedge(\sW_H^*)}\Bbbk=\hy(H_{\red})$, it follows that the left $\wedge(\sW_H^*)$-invariants thus coincide with
\[ 
\mr{Hom}_{\hy(H_{\red})}(\hy(H_{\red}),\ \pi_*\cF|_U \otimes \wedge(\sZ))=\pi_*\cF|_U \otimes \wedge(\sZ). 
\]
Now, the restriction of $\Omega_U'$ in question is seen to be the identity map on 
$\pi_*\cF|_U \otimes \wedge(\sZ)$, whence it is isomorphic, as desired.

(2)\
Regard $\Bbbk$ as the trivial module over $\hy(H_{\red})$ and over $K$. 
Since the $\hy(H)$-invariants in the range of $\Omega_U$ equal 
$\mr{Hom}_{\hy(H_{\red})}(\Bbbk, \wedge(\sZ)\otimes \pi_*\cF|_U)$, 
one sees from \eqref{eq:FQ} that $(\Omega_U)^J$ coincides with
\[
\cO(\pi^{-1}(\ ))^J\overset{\simeq}{\longrightarrow}\mr{Hom}_K(\Bbbk,\ \wedge(\sZ)\otimes \pi_*\cF|_U)
=(\wedge(\sZ)\otimes \pi_*\cF|_U)^K=\wedge(\sZ)\otimes\cF_{Q}|_U. 
\]
\end{proof}

\begin{proof}[Proof of Theorem \ref{thm:quotient}]
Since the first half of Part 1 is proved by Proposition \ref{prop:Omega} (2), 
it suffices to prove Part~2, from which the rest follows easily.  

Suppose that the situation is as above the proof. 
Define a morphism of manifolds by
\[
\ms{sec}_U : U \to \pi^{-1}(U),\ \ms{sec}_U(u)=\ms{triv}_U(u, e). 
\]
By Proposition \ref{prop:Omega} (2) we can identify $\cO_{G/H}|_U$ with $\wedge(\sZ)\otimes \cF_{Q}|_U$ 
through $(\Omega_U)^J$. 
A close look at the construction of the isomorphism $\Omega_U$ above,
without taking care of the relevant $H_{\red}$-actions, gives an isomorphism
\[
\cO|_{\pi^{-1}(U)} \overset{\simeq}{\longrightarrow} \mr{Hom}_{\hy(H_{\red})}(\hy(H),\ \wedge(\sZ)\otimes \cF|_{\pi^{-1}(U)}) 
\]
of sheaves on $\pi^{-1}(U)$. The last $\cF|_{\pi^{-1}(U)}$ on the right-hand side may be replaced with the direct image
of $\cF_{U \times H_{\red}}$ under $\ms{triv}_U$. Let us consider the sheaf-isomorphism obtained after the replacement.
Choose an arbitrary element $(x,h)\in U \times H_{\red}$, and let $f=\ms{sec}_U(x)$,\ $g=\ms{triv}_U(x,h)$, and so $g=fh$. 
Then one sees that the sheaf-isomorphism gives rise to the isomorphism between the completed stalks
\[
\hO_g \overset{\simeq}{\longrightarrow}\wedge(\sZ)\otimes \hF_{Q,x}~\hat{\otimes}~\mr{Hom}_{\hy(H_{\red})}(\hy(H),\hF_{H_{\red},h}),
\]
which decomposes so as
\begin{align*}
\hO_g&\xrightarrow{\hat{\Delta}_{f,h}}\hO_{f}\, \hat{\otimes}\, \hO_h
\xrightarrow{\hat{\Sigma}_{U, f}\hat{\otimes}\, \mr{cano}}
\wedge(\sZ)\otimes \hF_{f}\, \hat{\otimes}\, \hO_{H,h}\\
&\xrightarrow{\mr{id}_{\wedge(\sZ)}\otimes (\widehat{\ms{sec}^*_U})_x \hat{\otimes}\, \hat{\eta}_{H,h}}\wedge(\sZ)\otimes 
\hF_{Q,x}~\hat{\otimes}~\mr{Hom}_{\hy(H_{\red})}(\hy(H),\hF_{H_{\red},h}),
\end{align*}
where $\mr{cano} : \hO_h\to \hO_{H,h}$ is the canonical map arising from the inclusion $G \supset H$, and
$\eta_H : \cO_H\overset{\simeq}{\longrightarrow} \mr{Hom}_{\hy(H_{\red})}(\hy(H),\cF_{H_{\red}})$ is the
canonical sheaf-isomorphism for $H$, as given by Proposition \ref{prop:eta} (1); the last two sheaves
are identified through $\eta_H$. 
It follows that $\pi^{-1}(U)$ is right $H$-stable in $G$, and the sheaf-isomorphism 
in question is associated with an isomorphism
\begin{equation*}
\ms{Triv}_U : (U,\cO_{G/H}|_U) \times H \overset{\simeq}{\longrightarrow} (\pi^{-1}(U), \cO|_{\pi^{-1}(U)})
\end{equation*}
of right $H$-equivariant super-manifolds, which is seen to have the two properties described in the theorem. 
\end{proof}

\begin{rem}[added in revision]
The referee kindly informed the authors of the article \cite{FLV} 
by R.~Fioresi et al., which proves 
an analogous result (see \cite[Section 3]{FLV}) of our Theorem \ref{thm:quotient}
in the (real) $C^{\infty}$-situation. Their proof, which uses 
the Frobenius Theorem \cite[Chapter 4]{CCF}, constructs the quotient $G/H$ and
proves that $G \to G/H$ is a principal super $H$-bundle (in the 
$C^{\infty}$-situation). 
But seemingly, it does not 
prove the property that any non-super trivialization is liftable to a super trivialization;
see Remark \ref{rem:algebra_analogue} (1). We remark that our proof slightly
modified proves
their result including that liftability; the modification uses the fact that
(real) Lie groups uniquely turn into real analytic groups, 
and their finite-dimensional representations
are necessarily analytic. 
\end{rem}

\section*{Acknowledgments}
The authors gratefully acknowledge that
the work was supported by JSPS Grant-in-Aid for Scientific Research (C), 26400035 and 17K05189.

\end{document}